\documentclass[12pt]{amsart}

\usepackage[margin=1.5in]{geometry}
\usepackage{tikz-cd} 
\usetikzlibrary{cd}

\usetikzlibrary{arrows.meta,calc}
\usetikzlibrary{decorations.markings} 
\tikzset{->-/.style={decoration={
      markings,
      mark=at position .6 with {\arrow{>}}},postaction={decorate}}}

\tikzset{-<-/.style={decoration={
  markings,
  mark=at position .6 with {\arrow{<}}},postaction={decorate}}} %

\newcommand{\LoopLeft}[2]{%
  \draw[-<-]
    (#1.north west)
    .. controls ($(#1)+(-\LoopW,\LoopH)$)
                and ($(#1)+(-\LoopW,-\LoopH)$)
    .. node[left] {#2}
    (#1.south west);
} 

\newcommand{\LoopRight}[2]{%
  \draw[->-]
    (#1.north east)
    .. controls ($(#1)+(\LoopW,\LoopH)$)
                and ($(#1)+(\LoopW,-\LoopH)$)
    .. node[right] {#2}
    (#1.south east);
}

\newcommand{\LoopAbove}[2]{%
  \draw[->-]
    (#1.north west)
    .. controls ($(#1)+(-\LoopH,\LoopW)$)
                and ($(#1)+(\LoopH,\LoopW)$)
    .. node[above] {#2}
    (#1.north east);
}

\newcommand{\LoopAboveDash}[2]{%
  \draw[->-, dashed]
    (#1.north west)
    .. controls ($(#1)+(-\LoopH,\LoopW)$)
                and ($(#1)+(\LoopH,\LoopW)$)
    .. node[above] {#2}
    (#1.north east);
}

\newcommand{\LoopBelow}[2]{%
  \draw[-<-]
    (#1.south west)
    .. controls ($(#1)+(-\LoopH,-\LoopW)$)
                and ($(#1)+(\LoopH,-\LoopW)$)
    .. node[below] {#2}
    (#1.south east);
}

\newcommand{\LoopLeftRotCCW}[3]{%
  \draw[->-]
    (#1.{135+#3})
    .. controls
      ($ (#1) + ({cos(#3)*(-\LoopW) - sin(#3)*(\LoopH)},
                 {sin(#3)*(-\LoopW) + cos(#3)*(\LoopH)}) $)
      and
      ($ (#1) + ({cos(#3)*(-\LoopW) - sin(#3)*(-\LoopH)},
                 {sin(#3)*(-\LoopW) + cos(#3)*(-\LoopH)}) $)
    .. node[pos=0.5, left] {#2}
    (#1.{225+#3});
}

\newcommand{\LoopBelowRotCCW}[3]{%
  \draw[->-]
    (#1.{225+#3})
    .. controls
      ($ (#1) + ({cos(#3)*(-\LoopH) - sin(#3)*(-\LoopW)},
                 {sin(#3)*(-\LoopH) + cos(#3)*(-\LoopW)}) $)
      and
      ($ (#1) + ({cos(#3)*(\LoopH) - sin(#3)*(-\LoopW)},
                 {sin(#3)*(\LoopH) + cos(#3)*(-\LoopW)}) $)
    .. node[pos=0.5, below] {#2}
    (#1.{315+#3});
} 

\usepackage[utf8]{inputenc}
\usepackage[english]{babel}
\usepackage{graphicx}
\usepackage{amsmath,amssymb}
\usepackage{amsthm}
\usepackage{xcolor} 
\usepackage[T1]{fontenc}
\usepackage{etoolbox}
\usepackage{mathtools} 

\mathtoolsset{showonlyrefs=true,showmanualtags=true}

\usepackage{enumitem}
    \setlist[itemize]{leftmargin=5mm} 
\usepackage[normalem]{ulem} 
\usepackage{verbatim} 

\usepackage{stix} 

\usepackage[bookmarks=false]{hyperref}
\hypersetup{
    colorlinks,
    citecolor=black,
    filecolor=black,
    linkcolor=black,
    urlcolor=black
}

\numberwithin{equation}{section} 

\numberwithin{equation}{section}
\newtheorem{theo}{Theorem}[section] 

\newtheorem{thm}[theo]{Theorem}
\newtheorem{cor}[theo]{Corollary}
\newtheorem{lem}[theo]{Lemma}
\newtheorem{prop}[theo]{Proposition}

\newtheorem{mainthm}{Theorem}

\theoremstyle{definition}
\newtheorem{defn}[theo]{Definition}

\newtheorem*{standingassumptions*} 
{Standing Assumptions}
\newtheorem{notation}[theo]{Notation}

\theoremstyle{remark}
\newtheorem{rem}[theo]{Remark} 
 
\theoremstyle{definition}
\newtheorem{exa}[theo]{Example}

\newcommand{\defeq}{\vcentcolon=}

\newcommand{\N}{\mathbb{N}}

\newcommand{\T}{\mathbb{T}}
\newcommand{\Z}{\mathbb{Z}}

\newcommand{\calG}{\mathcal{G}} 
\newcommand{\cGss}[1][G]{\mathcal{G}_{#1,E}}

\newcommand{\calS}{\mathcal{S}}
\newcommand{\ssym}{\calS_{\text{\normalfont sym}}} 
\newcommand{\salt}{\calS_{\text{\normalfont alt}}}
\newcommand{\ssymdual}{\widehat{\calS}_{\text{\normalfont sym}}}
\newcommand{\calT}{\mathcal{T}} 
\newcommand{\calZ}{\mathcal{Z}}
\newcommand{\cs}{C^*}
\newcommand{\csr}{\cs_{r}}

\newcommand{\Iso}{\operatorname{Iso}}

\newcommand{\mfkg}{\mathfrak{g}}
\newcommand{\mfkh}{\mathfrak{h}}
\newcommand{\mfkt}{\mathfrak{t}}
\newcommand{\mfks}{\mathfrak{s}}

\newcommand{\z}{^{(0)}}

\newcommand{\comp}{^{(2)}}
\newcommand{\inv}{^{-1}}

\makeatletter
\def\bign#1{\mathclose{\hbox{$\left#1\vbox to8.5\p@{}\right.\n@space$}}\mathopen{}}
\makeatother

\newcommand{\HleftX}{
        \mathchoice
              {\displaystyle\mathbin\smalltriangleright}
              {\textstyle\mathbin\smalltriangleright}
              {\scriptstyle\mathbin\smalltriangleright}
              {\scriptscriptstyle\mathbin\smalltriangleright}
        } 
        
\newcommand\mvisiblespace[1][.7em]{%
        	\makebox[#1]{%
        		\kern.07em
        		\vrule height.3ex
        		\hrulefill
        		\vrule height.3ex
        		\kern.07em
        	} 
        }

\title{Cartan subalgebras in self-similar graph $C^*$-algebras}
\author[D. Archey]{Dawn Archey} 
    \address{Department of Mathematics, University of Detroit Mercy, 4001 W McNichols Rd., Detroit, MI 48221, United States}
    \email{archeyde@udmercy.edu}
\author[A. Duwenig]{Anna Duwenig}
\address{School of Mathematics and Statistics, UNSW Sydney, Sydney NSW 2052, Australia}
\email{a.duwenig@unsw.edu.au} 
\author[S. Hua]{Shanshan Hua}
\address{Mathematisches Institut, Fachbereich Mathematik und
Informatik der Universität Münster, Einsteinstrasse 62, 48149 Münster, Germany.} 
\email{shua@uni-muenster.de} 
\author[K. McCormick]{Kathryn McCormick} 
    \address{California State University Long Beach, 1250 Bellflower Blvd, Long Beach CA 90840, United States}
    \email{kathryn.mccormick@csulb.edu}
\author[R. Norton]{Rachael Norton}
    \address{St. Olaf College, 1520 St. Olaf Ave, Northfield, MN 55057}
    \email{norton10@stolaf.edu}
\author[D. Yang]{Dilian Yang}
    \address{Department of Mathematics and Statistics, University of Windsor, ON N9B
3P4, Canada}
    \email{dyang@uwindsor.ca}

\begin{document} 
\begin{abstract}
For a self-similar graph $(G, E)$, we find a distinguished subgroupoid of the associated path groupoid $\cGss $ -- the symmetric cycline subgroupoid $\ssym$. If the acting group $G$ is abelian, we show that $\ssym$ is open, abelian, and normal. For $G=\mathbb{Z}$, we describe the dual bundle $\ssymdual$ of $\ssym$ which can be used to provide a different groupoid model for the self-similar graph $C^*$-algebra $\mathcal{O}_{\mathbb{Z}, E}\cong \csr(\cGss[\Z])$. For a large class of self-similar graphs $(\mathbb{Z}, E)$, we further prove that $\ssym$ is maximal among open abelian subgroupoids of $\Iso(\cGss[\Z])^{\circ}$ and closed in $\cGss[\Z]$, so that it gives rise to a Cartan subalgebra of $\mathcal{O}_{\mathbb{Z}, E}$. This result seems new even for genuine actions. Our proofs heavily rely on careful studies of dynamical behaviours of cycline triples of $(\mathbb{Z}, E)$ and on a dynamical-flavour classification for the vertices of $E$. Some results hold in more general settings and may be of independent interest.

\end{abstract} 

\maketitle

\section{Introduction} 

Self-similar graph $C^*$-algebras $\mathcal{O}_{G, E}$ were  introduced by Exel and Pardo in~\cite{Self-similar:EP17} as a unifying framework encompassing both Katsura algebras $\mathcal{O}_{A,B}$ (\cite{Katsura_08}) and Nekrashevych algebras arising from self-similar groups (\cite{Nekrashevych_generic, Nekrashev2009, Nek05}). 
A self-similar graph is a pair $(G, E)$ consisting of a discrete countable group $G$ and a directed graph $E$ such that $G$ acts on $E$ by graph automorphisms from the left, while $E$ `acts' on $G$ from the right through  the \emph{restriction map} $(g, \mu)\mapsto g|_{\mu}$, in such a way that for all $g\in G$  
and finite paths $\mu, \nu$ in $E$,  
the formula 
\[
    g\cdot (\mu \nu) = (g\cdot \mu) (g|_{\mu}\cdot \nu)
\]
holds. Under the additional assumptions that $G$ is amenable and the action is 
\emph{pseudo-free} (see Definition~\ref{defn:pseudo-free}), the 
self-similar graph $C^*$-algebra $\mathcal{O}_{G,E}$ admits a natural groupoid model $\cGss$, meaning that $\mathcal{O}_{G,E}\cong C^*(\cGss)\cong C^*_r(\cGss)$ (\cite[Theorem 9.6, Corollary 10.18]{Self-similar:EP17}); the groupoid  $\cGss$ is called the {\em self-similar path groupoid}. This places these algebras within the broader framework of {\'e}tale groupoid $C^*$-algebras. Such constructions have since been extended to settings such as topological graphs (\cite{BKQ:2017}) and higher-rank graphs (\cite{Self-similar:LY19,Self-similar:LY21}). 

Self-similar graph $C^*$-algebras provide a rich source of examples, bridging combinatorial, dynamical, and operator-algebraic perspectives. A special case is 
where $E$ `acts' trivially on $G$ meaning that $g|_{\mu}=g$ for all $\mu\in E$; we then 
say that $G$ acts on $E$ {\em genuinely} (rather than self-similarly). Here, the algebra $\mathcal{O}_{G, E}$ is just the crossed product $C^{*}(E)\rtimes G$, where the action is induced by the given action of $G$ on $E$. 

Another fundamental example of self-similar graph $\cs$-algebras are the so-called {\em Katsura algebras $\mathcal{O}_{A, B}$}, associated to pairs $A,B$ of matrices. These all fall into the special case when $G=\mathbb{Z}$, and they play an important role in the classification program: every UCT Kirchberg algebra (i.e., any $C^*$-algebra in the UCT class that is simple, separable, nuclear, purely infinite) can be realized as a Katsura algebra (\cite{Katsura_08}), and hence admits a groupoid model. This perspective is central to Katsura’s result that actions of cyclic groups on $K$-theory can be lifted to genuine actions on UCT Kirchberg algebras \cite[Theorem 3.5]{Katsura_08}. 

Given the breadth of examples and applications of self-similar graph $C^*$-algebras, substantial effort has been devoted to understanding their structural properties (see, for example, \cite{KK_duality_ss, Self-similar:EP17, Self-similar:LY19, Self-similar:LY21, Valente-Yang25}). In \cite{Self-similar:LY19}, the authors (including the last author of the present paper) study the KMS-states of self-similar $k$-graph $C^*$-algebras. Under hypotheses ensuring well-behaved periodicity of the underlying self-similar $k$-graph (formulated via a notion called \emph{local
faithfulness}), they identify a Cartan subalgebra arising from the interior of the isotropy subgroupoid of the associated self-similar path groupoid model. Recall that a \emph{Cartan subalgebra} of a $C^*$-algebra is a maximal abelian subalgebra that is regular and admits a faithful conditional expectation onto it. The Cartan subalgebra identified in \cite{Self-similar:LY19} plays a central role in describing equilibrium states: the KMS-simplex is affinely isomorphic to a distinguished compact convex subset of the tracial state space of the Cartan subalgebra. 
 
More broadly, Cartan subalgebras have attracted considerable attention in recent years (see \cite{Duwe:2025:Diag-pp, DGNRW, pitts2021normalizers} and references therein) due to their deep connections with the classification program (\cite{BarlakLi, Li:EveryClass, Tu:BCConj}), rigidity phenomena (\cite{LiRenault}), and the structure of intermediate subalgebras (\cite{BCF25, BEFPR:2021}). In particular, the presence of a Cartan subalgebra realizes the ambient $C^*$-algebra as the reduced twisted groupoid $C^*$-algebra of an effective groupoid (\cite{Raad, Renault:2008}). 

Cartan subalgebras are far from unique, even when one restricts to a fixed spectrum (\cite{Xin_non-unique, LiRenault, Renault:2008}, etc.). On the other hand, the existence is not guaranteed in general, though it is often more accessible in settings where a non-effective groupoid model is available (\cite{BNRSW:Cartan, DWZ:Twist}). 
Motivated by these considerations, it is natural to investigate Cartan subalgebras  of self-similar graph algebras $\mathcal{O}_{G, E}$. Our main goal is to identify, for a broad class of such $C^*$-algebras, a subgroupoid $\mathcal{S}$ of $\cGss$ such that $C^{*}_{r}(\mathcal{S})$ is a Cartan subalgebra of $\mathcal{O}_{G, E}$. Note that asking for $C^*_{r}(\mathcal{S})$ to be a $C^*$-subalgebra
forces $\mathcal{S}$ to be open, and commutativity of $C^*_{r}(\mathcal{S})$  
is equivalent to the open $\mathcal{S}$ being an abelian subgroupoid of $\Iso(\cGss)^{\circ}$, 
the interior of the isotropy in $\cGss$. 

One of the most natural subgroupoids to consider in any Hausdorff \'etale groupoid $\mathcal{G}$ is $\Iso(\calG)^{\circ}$ itself.
It is shown in \cite{BNRSW:Cartan} that,
provided $\Iso(\calG)^\circ$ is abelian and closed, the associated subalgebra is a Cartan subalgebra of $C^*_r(\calG)$. Closedness
of $\Iso(\cGss)^{\circ}$ can be characterized in terms of a dichotomy of dynamical properties of infinite paths in $E$, namely Condition~\ref{mainthm_cond1} in Theorem~\ref{mainthm_genuine_action} below; see Subsection~\ref{ssec:closedness} for definitions and proofs. Moreover, in
the special case that $(G, E)$ is a genuine action and $G$ is abelian, it turns out that 
$\Iso(\cGss)^{\circ}$ is always abelian 
(Corollary~\ref{C:can-act}). In this setting, we
thus obtain the following natural and tractable characterization of when $C^*_r(\Iso(\cGss)^{\circ})$ is a Cartan subalgebra of $C^*_{r}(\cGss)\cong C^{*}(E)\rtimes G$. 

\begin{mainthm}[Theorem \ref{Thm_genuine_action_Cartan}] 
\label{mainthm_genuine_action}  
Let $E$ be a source-free, finite directed graph and let $G$ be an abelian group. If $(G, E)$ is a genuine action, then the following are equivalent: 
\begin{enumerate}[label=\textup{(\ref{mainthm_genuine_action}.\arabic*)}] 
\item 
\label{mainthm_cond1}
For any $x\in E^{\infty}$ and $g\in G$, $x$ is either $g$-generic or $g$-rare.
\item $C^*_r(\Iso(\cGss)^\circ)$ is a Cartan subalgebra in 
$C^{*}(E)\rtimes G$.
\end{enumerate} 
\end{mainthm} 

To the best knowledge of the authors, this characterization of Cartan subalgebras in crossed products is new. Beyond the cases where 
Condition~\ref{mainthm_cond1} is satisfied, it remains unclear how to identify Cartan subalgebras in $C^{*}(E)\rtimes G$. 

For an arbitrary self-similar graph $(G,E)$, the open subgroupoid $\Iso(\cGss)^{\circ}$ is not necessarily abelian, especially when $(G, E)$ is not assumed to be locally faithful. 
This is in contrast to the setting of \cite{Self-similar:LY19}. In order to identify a Cartan subalgebra of $C^*_r(\cGss)$, we need to instead find maximal abelian open subgroupoids in $\Iso(\cGss)^{\circ}$. To better understand the structure of $\Iso(\cGss)^{\circ}$, we classify vertices in terms of their dynamical properties (Corollary~\ref{conclusion}) and systematically study \emph{cycline triples} of $(G,E)$, that is, triples $(\alpha, g, \beta)$ for which the associated open compact bisections $\calZ(\alpha,g ,\beta)$ are contained in $\Iso(\cGss)^{\circ}$. These bisections collectively form an open cover of $\Iso(\cGss)^{\circ}$, when $(G, E)$ is pseudo-free (see the discussion following 
Definition~\ref{defn:cycline_triple}).
We let $\ssym$ denote the union of $\calZ(\alpha, g, \alpha)$ ranging over all symmetric cycline triples $(\alpha, g, \alpha)$. This leads to the following concrete description of $\Iso(\cGss)^{\circ}$, which encapsulates the major original insights of the paper. 

\begin{mainthm}[Structure of $\Iso(\cGss)^\circ$; Proposition~\ref{prop:Ssym:open,normal,abelian,subgpd}, Corollary~\ref{cor:the green corollary}] 
\label{mainthm_structure} 
Let $(G, E)$ be a pseudo-free self-similar graph, where $E$ is a source-free, finite directed graph and $G$ is a discrete countable group with an element of infinite order. Then $\ssym$ is an open abelian normal subgroupoid of $\Iso(\cGss)^\circ$ and 
\begin{equation} 
\label{Iso:concrete_description} 
\Iso(\cGss)^\circ = \ssym \bigsqcup \left(
\bigcup_{\substack{(\alpha, g, \beta) \text{ non-symmetric} \\
\text{cycline triple}}} 
\calZ(\alpha, g, \beta)\right). 
\end{equation} 
Moreover, $\calZ(\alpha, g, \beta)$ is a singleton for every non-symmetric cycline triple. 
\end{mainthm} 

Theorem~\ref{mainthm_structure}
provides a natural candidate $\ssym$ for obtaining a Cartan subalgebra. Note further that $\ssym$ is maximal abelian in $\Iso(\cGss)^{\circ}$ precisely when all singletons in $\Iso(\cGss)^{\circ}$ arising from non-symmetric cycline triples fail to commute with $\ssym$. 

We illustrate the efficacy of the concrete description of $\Iso(\cGss)^{\circ}$ in the case $G = \mathbb{Z}$, where sufficient conditions for maximal abelianness of $\ssym$ in $\Iso(\cGss[\Z])^{\circ}$ are obtained in Section~\ref{sec: Cartan subalgebras}. Closedness of $\ssym$ in $\cGss$ is shown to be characterized again by condition~\ref{mainthm_cond1}. Combining these ingredients, an application of \cite[Corollary 4.5]{DWZ:Twist} leads to the main result on Cartan subalgebras of $\mathcal{O}_{\mathbb{Z}, E}$. 

\begin{mainthm}[Theorem~\ref{Thm_Ssym_Cartan}] 
\label{mainthm_Cartan_ss} 
Let $(\mathbb{Z}, E)$ be a pseudo-free self-similar graph, where $E$ is a source-free, finite directed graph. Assume additionally that the following conditions hold: 
\begin{enumerate}[label=\textup{(\ref{mainthm_Cartan_ss}.\arabic*)}] 
\item
For any $x\in E^{\infty}$ and $g\in \Z$, $x$ is either $g$-generic or $g$-rare. 
\item
\label{mainthm3_cond2} 
For any cycle $C$ in $E$ with no entrance, if $p$ is the minimal positive integer such that $a^p$ fixes a vertex on $C$, we have $a^p|_{C} \neq a^{\pm p}$. 
\end{enumerate} 
Then $C^*_{r}(\ssym)$ is a Cartan subalgebra of $C^*_{r}(\cGss[\Z])$. 
\end{mainthm} 

Condition~\ref{mainthm3_cond2} may appear 
{\em ad hoc}; indeed, it arises from explicit computations in the special case of $(\mathbb{Z}, E)$ ensuring that $\ssym$ is maximal abelian (see Lemmas~\ref{lem:TypeIITypeIDontCommute} and~\ref{lem:TypeIIITypeIDontCommute}). Although the assumptions in Theorem~\ref{mainthm_Cartan_ss} seem restrictive, they cover a substantially larger class than those treated in the existing literature; see for example Example~\ref{exa:closed_g_generic}.
In fact, Theorem~\ref{mainthm_Cartan_ss} should be regarded as a concrete application of the more structural result, Theorem~\ref{mainthm_structure}. This same general strategy also allows for case-by-case analyses tailored to specific examples, 
illustrated for instance in Appendix~\ref{sec:AnotherCandidate} when Condition~\ref{mainthm3_cond2}
in Theorem~\ref{mainthm_Cartan_ss} is relaxed. 

Another consequence of Theorem~\ref{mainthm_structure} is the existence of `easier models' for $\mathcal{O}_{G, E}$: The `standard' groupoid model $\calG_{G,E}$ for $\mathcal{O}_{G, E}$ that we mentioned earlier
is considerably more complicated than those associated naturally to, say, 
the $C^*$-algebra $C^*(E)$ of the directed graph $E$. In graph groupoids, elements encode shift equivalence between pairs of infinite paths together with an integer lag. By contrast, elements of $\calG_{G,E}$ 
keep track not only of the lag between infinite paths but also the self-similarity of the action of $G$. This amounts to the lag-group in $\calG_{G,E}$ being vastly more complicated than the integers: it is a semi-direct product of $\mathbb{Z}$ acting on a quotient 
space of functions $\mathbb{N}\to G$ (more details can be found in Section~\ref{defn_gpd}). Hence, it is of interest to find other, hopefully more tractable, groupoid models for $\mathcal{O}_{G, E}$. 

Our approach to constructing such `easier models’ is to leverage a result of 
Renault’s (\cite[Theorem 5.5(ii)]{Renault:2023:Ext}, quoted as Theorem~\ref{thm:Renaults5.5} below), which is related to his Weyl groupoid construction (\cite{Renault:2008}) via \cite{Duwe:2025:Diag-pp}. The idea is that the reduced 
$C^*$-algebra $\csr(\calG)$ of a groupoid can be realized as the twisted reduced $C^*$-algebra of a transformation groupoid $(\calG/\calS) \ltimes \widehat{\calS}$, where $\calS$ is a normal abelian subgroupoid of $\calG$ that is sufficiently well-behaved topologically, and $\widehat{\calS}$ denotes its dual groupoid, naturally acted upon by $\calG/\calS$. In our setting, $\ssym$ provides a good candidate for $\calS$, given the nice properties established in Theorem~\ref{mainthm_structure}. As illustrated by examples in Appendix~\ref{app:Weyl examples}, the resulting new transformation groupoid $(\calG/\ssym) \ltimes \ssymdual$ is significantly more tractable than $\cGss[\Z]$, capturing only the essential features of the underlying self-similar graph $(\mathbb{Z},E)$. 

\subsection*{Structure of the paper} 
The preliminaries section,
Section~\ref{S:Pre},
recalls the basic definitions of self-similar graph $C^*$-algebras and groupoid models. In Subsection~\ref{subsec:ct_lf}, we review cycline triples and refine the notion of (non)local faithfulness for our purposes. The section concludes with two corollaries identifying candidates for abelian subgroupoids in the genuine action case and the general setting (Corollaries~\ref{C:can-act} and~\ref{cor:cycline comm}). 

In Section~\ref{S:sym}, we show that $\ssym$ is an open, abelian, normal subgroupoid of $\Iso(\cGss)^{\circ}$ when $G$ is abelian (Proposition~\ref{prop:Ssym:open,normal,abelian,subgpd}), establishing the first part of Theorem~\ref{mainthm_structure}. In the case $G = \mathbb{Z}$, the dual bundle $\ssymdual$ is described explicitly in Section \ref{sec:Renault's reconstruction} (Theorem~\ref{thm:ssymdual}). 
The second part of Section~\ref{sec:Renault's reconstruction} uses this description to give a more concrete account of the action of $\cGss[\mathbb{Z}]/\ssym$ on $\ssymdual$. 

Section~\ref{sec: Cartan subalgebras} contains the main results of the paper. In Subsection~\ref{ssec:cycline triples and masa}, we analyze cycline triples at four types of vertices with distinct dynamical behaviors (Corollary~\ref{conclusion}). This leads to Corollary~\ref{cor:the green corollary}
and
establishes the second part of Theorem~\ref{mainthm_structure}. A detailed study of commutativity of elements in compact open bisections associated to cycline triples yields sufficient conditions for $\ssym$ to be maximal among open abelian subgroupoids of $\Iso(\cGss)^{\circ}$ (Theorem~\ref{thm:Ssym is maximal}). Subsection ~\ref{ssec:closedness} characterizes the closedness of subgroupoids $\calS$ of $\Iso(\cGss)^{\circ}$ containing $\ssym$ (Theorem ~\ref{thm:closedness-equivalent}). 
Combined 
with Corollary~\ref{C:can-act}, one direct consequence is the equivalent characterization of when $\Iso(\cGss)^{\circ}$ gives a Cartan subalgebra for genuine actions (Theorem~\ref{mainthm_genuine_action}). Finally, in Subsection~\ref{ssec:Ssym satisfies ricc}, we prove Theorem~\ref{mainthm_Cartan_ss} by assembling the necessary ingredients to apply \cite[Corollary 4.5]{DWZ:Twist}, including Theorems~\ref{mainthm_structure},
~\ref{thm:Ssym is maximal}, 
~\ref{thm:closedness-equivalent}
and Lemma~\ref{lem:Ssym (almost) immediately centralizing}. 

The paper concludes with two appendices. In Appendix~\ref{app:Weyl examples}, we compute the transformation groupoids $(\cGss[\Z]/\ssym) \ltimes \ssymdual$ introduced in Section~\ref{sec:Renault's reconstruction} for several explicit examples of self-similar graphs $(G, E)$. In Appendix~\ref{sec:AnotherCandidate}, we identify an alternative subgroupoid $\salt$ of
$\Iso(\cGss[\Z])^{\circ}$, which produces a Cartan subalgebra $C^*(\salt) \subseteq C^*(\cGss[\Z])$ for a certain class of self-similar graphs not covered by the results of Section~\ref{sec: Cartan subalgebras}. 

\subsection*{Notation and 
standing assumptions} 
\label{Notation and conventions}
We let $\N$ denote the set of all non-negative integers, and write $\N^+$ for all strictly positive integers. 

In this paper, $G$ is always a discrete countable group. Its identity is denoted by~$1_{G}$. If $G=\Z$, we will write $\Z=\langle a \rangle$ and its operation multiplicatively. 

Our graphs are always assumed to be directed, finite, and source-free.

\subsection*{Acknowledgements}
This collaboration was initiated at the third {\em Women in Operator Algebras} conference, held at the Banff International Research Station in June 2023. 
AD was partially supported by a Senior Postdoctoral Fellowship of the FWO (fellowship number 1206124N).  
SH was partially supported by the Mathematical Institute Scholarship during her DPhil studies at the University of Oxford, and was partially funded by the Deutsche Forschungsgemeinschaft (DFG, German Research Foundation) under Germany's Excellence Strategy EXC 2044-390685587, Mathematics M\"unster: Dynamics–Geometry–Structure, by the SFB 1442 of the DFG. 
DY was partially supported by an NSERC Discovery Grant 823065.
The authors would further like to thank Stuart White for hosting the authors for a collaborative week at the University of Oxford (grant number EPSRC EP/X026647/1), the Fields Institute for Research in Mathematical Sciences,
and the mathematical research institute MATRIX in Australia where part of this research was performed. 

\section{Preliminaries}
\label{S:Pre} 

\subsection{Self-similar graphs}
\label{SS:ssg} 

A directed graph $E=(E^{0} ,E^1,r,s)$ consists of a set of vertices, $E^{0} $, a set of directed edges, $E^1$, and the maps $r,s\colon  E^1\to E^{0} $, which specify the range and the source of each edge, respectively. 
For each $n\in \mathbb N\cup \{\infty\}$, we use $E^n$ to denote directed paths in $E$ with length $n$. We also use $|\mu|$ to denote the length of a finite path $\mu$. We use $E^*\defeq \cup_{n\geq 0} E^n$ to denote the set of all finite directed paths. 
For $v\in E^{0} $ and $n\in \mathbb N\cup\{\infty\}$, $vE^n$ (respectively $E^nv$) is the set of all paths of length $n$ with range $v$ (respectively source $v$). We also use $vE^*$ and $E^*v$ to denote all finite paths with range and source $v$ respectively.

So, by definition, $E^\infty$ is the set of all infinite paths of $E$. 
If $x=e_1e_2\dots\in E^\infty$, for natural numbers $m<n$, take
$x(m,n)\coloneqq e_{m+1}\dots e_n$, and define $x(n,n)\coloneqq r(e_{n+1})$. 
For $p \in \mathbb{N}$, define the left shift map $\sigma^p: E^\infty \to E^\infty$ by $\sigma^p(x)\coloneqq e_{p+1}e_{p+2} \dots \in E^\infty$.
For $\mu \in E^*$, we denote $\calZ(\mu)\defeq \{\mu x: \: x \in E^\infty, s(\mu)=r(x) \}$. We can give $E^\infty$ the topology generated by $\calZ(\mu)$, and since $\calZ(\mu)$ will be compact this will make $E^\infty$ into a second-countable locally compact Hausdorff space. 

\begin{defn}\label{D:action}
    Given a graph $E$, we call a bijective map $E^* \to E^*$ an \emph{automorphism} of the finite path space $E^*$ if it maps $E^n$ to $E^n$ for $n\in\N$ (i.e., the map preserves lengths of finite paths) and if it commutes with the source and range maps. We use $\operatorname{Aut}(E^*)$ to denote the set of all such automorphisms. 
    
    For a (discrete) group $G$, we say that \textit{$G$ acts on $E$} if there is a group homomorphism $\varphi\colon G\to \operatorname{Aut}(E^*)$.\footnote{This is a slight abuse of terminology; see Remark~\ref{rmk:abuse of terminology}.} We will usually just write $g\cdot \mu$ for the action of $g \in G$ on $\mu \in E^{*}$, with no reference to the map $\varphi$. We call the action {\em genuine} if it is entirely determined by the action of $G$ on $E$ in the sense that the equality
    \[
    \label{genuine_extend} 
        g\cdot (e_0\dots e_n) = (g\cdot e_0)\dots (g\cdot e_n)
    \]
    holds for all $g\in G$ and all finite paths $e_0\dots e_n\in E^*$.
\end{defn}

We first give the definition of self-similar graphs in terms of group actions on finite paths, where self-similarity 
explicitly appears in Condition \ref{item:S?} of Definition \ref{D:ss}. An alternative way of defining self-similar actions, more in the spirit of \cite{Self-similar:EP17}, is to extend the action self-similarly from that on edges (see Remark~\ref{rmk:abuse of terminology} for a detailed exposition).  
 
\begin{defn}\label{D:ss}
    Let $G$ be a (discrete) group. We call action of $G$ on $E$ \emph{self-similar} if there is a restriction map $G\times E^* \to G, \ (g, \mu)\mapsto g|_\mu$, such that the following conditions hold.
    \begin{enumerate}[label=\textup{(\arabic*)}]
    \item
    \label{item:S5}
    $(gh)\vert_\mu=g \vert_{h \cdot \mu} h \vert_\mu$ for all $g,h \in G,\mu \in E^*$;

    \item
    \label{item:S2}
    $g \vert_v =g$ for all $g \in G,v \in E^{0} $;
    
    \item
    \label{item:S4}
    $1_{G} \vert_{\mu}=1_{G}$ for all $\mu \in E^*$;
    
    \item
    \label{item:S3}
    $g \vert_{\mu\nu}=g \vert_\mu \vert_\nu$ for all $g \in G,\mu,\nu \in E^*$ with $s(\mu)=r(\nu)$;
    
    \item
    \label{item:S?}
    $g\cdot (\mu\nu)=(g \cdot \mu)(g \vert_\mu \cdot \nu)$ for all $g \in G,\mu,\nu \in E^*$ with $s(\mu)=r(\nu)$.
    \end{enumerate}
    In this case, we call $E$ a \textit{self-similar graph over $G$}, or 
    we say that $(G,E)$ is a \textit{self-similar graph}.
\end{defn}

\begin{exa}
    A genuine action of $G$ on $E$ in the sense of Definition~\ref{D:action} is self-similar with the restriction map
    $g|_\mu:=g$ for all $g\in G$ and all $\mu\in E^*$.
\end{exa}

\begin{rem}\label{rmk:abuse of terminology}
    The data of a self-similar action in our sense is already contained in the data on the level of the graph $E$ instead of its finite path space $E^*$. To be more precise, suppose we only know how $G$ acts on the set of edges and on the set of vertices (still preserving ranges and sources), and we only know how to restrict elements of $G$ by vertices and edges. Suppose further that this restriction map satisfies the following axioms:
    \begin{enumerate}[label=\textup{(\arabic*')}]
    \item \label{cond:S1}
    $(gh)|_e = g|_{h\cdot e} h|_e$ for all $g,h\in G$ and $e \in E^1$;
    \item\label{cond:S2} $g|_v=g$ for all $g\in G$ and $v\in E^{0}$;
    \item\label{cond:S3} $1_{G}|_e=1_{G}$ for all $e\in E^1$;
    \item\label{cond:S4} 
     $g|_e\cdot s(e) =g\cdot s(e)$ for all $g\in G$ and $e\in E^{1}$. 
    \end{enumerate}
    Then there is a
    natural
way of extending the action and restriction maps to action and restriction maps on the level of $E^*$ which fall into the scope of Definition~\ref{D:ss}, namely by inductively defining
    \[
        g\cdot (e_0e_1\dots e_n) \defeq (g\cdot e_0) \, g|_{e_{0}}\cdot (e_1\dots e_n)
        \quad\text{ and }\quad
        g|_{e_0e_1\dots e_n} \defeq (g|_{e_0})|_{e_1} \dots |_{e_n}.
    \]
    And again, only if the restriction map satisfies $g|_{\mu}=g$ for all 
    $\mu\in E^1$ and all $g\in G$ does the induced action $G\to \operatorname{Aut}(E^*)$ become a genuine action in the sense of Definition~\ref{D:action}.
\end{rem} 

A direct consequence of Definition \ref{D:action} is that loops are mapped to loops by group actions, as they respect $r, s$. This and some additional easy facts are recorded in Lemma~\ref{lem:cycles to cycles} below; they will be useful in Section \ref{ssec:cycline triples and masa} and Appendix~\ref{sec:AnotherCandidate} for instance. 

Recall that a path $C=\mu_1\dots \mu_k\in E^k$ is called a \textit{cycle} if $r(C)=s(C)$, and that a cycle $C$ is said to \textit{have no entrance} if $e\in E^1$ with $r(e)=r(\mu_i)$ for some $1\le i\le k$, then $e=\mu_i$. 

\begin{lem}\label{lem:cycles to cycles}
Suppose $G$ acts on 
$E$, 
and let $C$ be a cycle in $E$. 
Then for any fixed $h\in G$, the following hold.
\begin{enumerate}[label=\textup{(\arabic*)}]
        \item\label{it:cycles to cycles:length}
        $h\cdot C$ is a cycle of the same length as $C$. \item\label{it:cycles to cycles:entrance} $h\cdot C$ has no entrance if and only if $C$ has no entrance.
        \item\label{it:cycles to cycles:self-similarity} Suppose the action of $G$ is self-similar. If $C$ has no entrance, then $h \cdot C^\infty = h|_{C^k} \cdot C^\infty = (h \cdot C)^\infty$ for all $k \in \mathbb{N}$.
\end{enumerate}
\end{lem}

\begin{proof}
    Fix $h\in G$. Let $C$ be a cycle of length $n$ and define $C' := h \cdot C$.
    
    \ref{it:cycles to cycles:length}
    Since
    $r(h \cdot C) = h \cdot r(C)= h \cdot s(C) = s(h \cdot C)$, 
    we see that
    $C'$ is a cycle of length at most $n$.  If $|C'| = m < n$, then $C(0,0) = (h^{-1} \cdot C')(0,0) = h^{-1} \cdot C'(0,0) = h^{-1} \cdot C'(m,m) = (h^{-1} \cdot C')(m,m) = C(m,m)$, which is a contradiction. Thus $|C'| = n$. 
    
    \ref{it:cycles to cycles:entrance}
    Now assume $C$ has no entrance, and assume for a contradiction that $C'$ does have an entrance, say 
    $e'$. Then $r(e') \in C'$, so $r(h^{-1} \cdot e') = h^{-1} \cdot r(e') \in C$. This implies that $h^{-1} \cdot e'$ is an edge in $C$, since $C$ has no entrance. But then $h \cdot (h^{-1} \cdot e') = e'$ must be an edge in $C'$ since $h \cdot C = C'$, a contradiction. Thus $C'$ has no entrance.
    By symmetry, we deduce that $C'$ has no entrance if and only if $C$ has no entrance.

    \ref{it:cycles to cycles:self-similarity}
    If $C'$ has no entrance, then  $|v'E^\infty| = 1$ for all $v' \in C'$. 
   Since $h \cdot C = C'$, then by self-similarity, $h \cdot C^\infty = (h \cdot C)(h|_C \cdot C^\infty) = C'(h|_C \cdot C^\infty)$. Since $s(C')$ has only one incoming path, namely $(C')^\infty$, then we must have $h|_C \cdot C^\infty = (C')^\infty$. Therefore $h \cdot C^\infty = (C')^\infty$. Finally, since $h \cdot C^\infty = (C')^\infty$, then $(h \cdot C^k)(h|_{C^k} \cdot C^\infty) = (C')^k(C')^\infty$. Since paths $h \cdot C^k$ and $(C')^k$ have the same length, then they are equal. So we must have $h|_{C^k} \cdot C^\infty = (C')^\infty$. 
\end{proof}

Note from \cite[Section~3]{Self-similar:LY21}
that a self-similar action $(G, E)$ defined above induces a canonical action of $G$ on the infinite path space $E^\infty$: for any $g\in G$ and infinite path $x$, we take 
\begin{align}
(g \cdot x)(p,q) & \defeq  g|_{x(0,p)} \cdot x(p,q)\quad \text{for all } 0\le p<q\in \mathbb N, 
\end{align} 
so that $g\cdot x\in E^\infty$. Moreover, for any $g\in G$ and $x\in E^{\infty}$, a map $g|_{x}\colon \N\to G$ is induced by the restriction in the following way: 
\begin{align}
g|_x(p) & \defeq  g|_{x(0,p)} \quad \text{for all } p \in \N.
\end{align}
The map records how the group element $g$ changes when traversing through $x$ according to the restriction map. Note that our definition of $g|_{x}$ leads to the formula 
\[
    (gh)|_{x}= g|_{h\cdot x} h|_{x} \quad \text{ for any }g, h \in G, \text{ and } x\in E^{\infty}, 
\] 
by standard computations. This formula for infinite paths is consistent with Definition~\ref{D:ss}~\ref{item:S5} for finite paths.
We include the following lemma, which will be used repeatedly in Sections \ref{sec:Renault's reconstruction} and \ref{sec: Cartan subalgebras}.

\begin{lem}
    \label{it:powers go in} 
    Suppose $h\in G$ and $x\in E^\infty$ are such that $h\cdot x=x$. Then for $n\leq p\in \N$ and $\ell\in \Z$, we have
    \(
        (h|_{x(0,p)})^{\ell }
        =
        (h|_{x(0,n)})^{\ell }|_{x(n,p)}
     \). In particular,
    \([h|_{ x }]^{\ell }
        =
        [h^{\ell }|_{ x }].
    \) 
\end{lem} 

\begin{proof}
    Follows from an easy induction on $|\ell |$, with the second claim following from the case $n=0$.
\end{proof} 

We now recall the notion of self-similar graph $C^*$-algebras introduced by Exel and Pardo in \cite[Definition 3.2]{Self-similar:EP17} (see also \cite[Definition 3.1]{Self-similar:EP17}
or \cite{Rae05}
for the definition of a Cuntz-Krieger family  $\{s_\mu\}_{\mu \in E^*}$ for a graph $E$). 

\begin{defn} 
\label{D:ss_C*_algebra} 
Let $(G,E)$ be a self-similar graph. The \textit{self-similar graph $C^*$-algebra} $\mathcal{O}_{G,E}$
is defined to be the universal unital $C^*$-algebra generated by a family of unitaries $\{u_g\}_{g \in G}$ and
a Cuntz-Krieger family $\{s_\mu\}_{\mu \in E^*}$ satisfying
\begin{enumerate}[label=\textup{(\roman*)}]
\item 
$u_{gh}=u_g u_h$ for all $g$ and $h \in G$;
\item 
\label{ss_relator} 
$u_g s_\mu=s_{g \cdot \mu} u_{g \vert_\mu}$ for all $g \in G$ and $\mu \in E^*$.
\end{enumerate}
\end{defn} 

\begin{rem}  
\label{genuine_C*} 
When the self-similar action is genuine in the sense of Definition~\ref{D:action}, Condition~\ref{ss_relator} of Definition~\ref{D:ss_C*_algebra} becomes $u_g s_\mu u_g^{*}=s_{g \cdot \mu}$ for all $g \in G$ and $\mu \in E^*$, which is the defining relation for the crossed product $C^*(E) \rtimes G$. In other words, $\mathcal{O}_{G,E} \cong C^*(E) \rtimes G$ in the special case of genuine actions. 
\end{rem} 

\subsection{The groupoid of a self-similar graph} 
\label{defn_gpd} 
A \emph{groupoid} is a small category in which every element has an inverse. Given a groupoid $\calG$, we let $\calG\z  $ denote the unit space of $\calG$, $s\colon \calG \to \calG\z  $ the source map $s(\mfkg)=\mfkg^{-1}\mfkg$, $r\colon \calG \to \calG\z  $ the range map $r(\mfkg)=\mfkg\mfkg^{-1}$, and $\calG\comp \subseteq \calG \times \calG$ the set of composable pairs.
For a unit $x\in \calG\z$, we will write $\calG_{x}\defeq s^{-1}(x)$, $\calG^{x}\defeq r^{-1}(x)$, and $\calG(x)\defeq \calG_{x}\cap\calG^{x}$ 
will denote the isotropy group at $x$. 
The isotropy subgroupoid of a groupoid $\calG$ will be denoted $\Iso(\calG)\defeq \{\mfkg \in \calG : s(\mfkg)=r(\mfkg)\}
=
\cup_{x\in \calG\z} \calG(x)
$.

A \emph{topological groupoid} is a groupoid that is also a topological space, in which inversion, the range map, and the source map are continuous, and multiplication is continuous with respect to the relative topology on $\calG\comp \subseteq \calG \times \calG$. A \emph{bisection} $U \subseteq \calG$ is a subset where $r|_U$ and $s|_U$ are injective. 
A locally compact Hausdorff groupoid
is called \emph{ample} when there is a basis consisting of compact open bisections. 
 
Given a self-similar graph $(G,E)$,
we follow the notation for $\cGss$ as 
in \cite[Definition 5.2]{Self-similar:LY21}
for the case $k=1$
(also refer to 
\cite[Theorem 8.19]{Self-similar:EP17}, where $\cGss$ was originally defined). 
To encode the self-similarity, 
we replace the `usual' lag group that is the integers in a graph-groupoid (and $\Z^k$ in a $k$-graph groupoid) with a group that includes
information about how the restriction map interacts with infinite paths 
via a semi-direct product by a right shift map. To this end, we consider the group $C(\N ,G)$
of functions $f\colon  \N  \to G$ under pointwise multiplication. Define the tail equivalence relation by $f\sim g$ exactly when there exists some $N\in \N$ such that $f(k)=g(k)$ for all $k \geq N$. We then get a well-defined quotient group $Q(\N ,G)=C(\N ,G)/\sim$, and we write $[f]$ for the equivalence class of $f$ in $Q(\N ,G)$. For $n \in \Z $, $f\in C(\N ,G)$, and $p \in \N $, we can define the right translation
\[ \calT_n(f) (p) \defeq  \begin{cases} f(p-n) & \text{if } p-n\geq 0, \\ 1_{G} & \text{if }p-n< 0. \end{cases}\] 
The induced map $\calT_n \colon  Q(\N ,G) \to Q(\N ,G)$ is a group automorphism and the map $\calT \colon  \Z  \to \operatorname{Aut}(Q(\N ,G))$ given by $n \mapsto \calT_n$ is a homomorphism. We can then form the semidirect product $Q(\N ,G)\rtimes_\calT \Z $.

Now we define our groupoid $\cGss$ by
\begin{align*}
\cGss \defeq  \left \{ (x; [f], p-q; y):
\begin{matrix}
\sigma^p(x) = f(p)\cdot \sigma^q(y) \\
f(p+n)= f(p)|_{\sigma^q(y)}(n) \text{ for all }n \in \N  
\end{matrix}
\right \},
\end{align*}
or equivalently, as
explained in \cite[Remark~5.3 (ii)]{Self-similar:LY21},
\begin{align*}
\cGss = \left \{ (\mu(g \cdot x); \calT_{|\mu|}([g|_x]), |\mu|-|\nu|; \nu x) : 
\begin{matrix}
g \in G, \mu, \nu \in E^* \\
s(\mu)=g \cdot s(\nu), x \in s(\nu)E^\infty 
\end{matrix}
\right\}.
\end{align*} 
The operations of the groupoid $\cGss$ are given as follows: 
\begin{align}\label{DefGroupoidOps}
    s(x;[f],p-q;y) &= (y;[1_{G}],0;y),\\
    r(x;[f],p-q;y) &= (x;[1_{G}],0;x),\\
    (x;[f],p-q;y)^{-1} &= (y;\calT_{q-p}([f]^{-1}),q-p;x),\label{eq:mfkg inv}
    \\
    (x;[f],p-q;y)(y; [g],n-m;z) &= (x; [f]\calT_{p-q}([g]) ,p-q+n-m; z).
\end{align} 
Given $g \in G$, $\mu, \nu \in E^*$ with $s(\mu) = g \cdot s(\nu)$, we define the
\emph{cylinder sets} by
\[ 
\calZ(\mu,g,\nu) \defeq  \big\{ (\mu(g \cdot x); \calT_{|\mu|}([g|_x]), |\mu|-|\nu|; \nu x ): \: x \in s(\nu)E^\infty \big\}. 
\] 
We equip $\cGss$ with the topology generated by the cylinder sets which makes it a second-countable topological groupoid.
Under the following assumption of {\em pseudo-freeness}, the cylinder sets $\calZ(\mu, g, \nu)$ form a basis for the topology
and $\cGss$ is ample and Hausdorff.  

\begin{defn} 
\label{defn:pseudo-free} 
A self-similar graph $(G, E)$ is  
\textit{pseudo-free} when the following holds:
if $g\in G$ and $\mu \in E^*$ satisfy $g\cdot \mu = \mu$ and $g|_{\mu} = 1_{G}$, then $g = 1_{G}$. 
\end{defn}

We again illustrate in the case where $G$ acts on $E$ by a genuine action. 

\begin{rem} 
For a genuine action, $(G,E)$ is necessarily pseudo-free, since $g|_{\mu} = 1_{G}$ implies $g = g|_{\mu} = 1_{G}$ for any $g\in G$ and $\mu \in E^{*}$. 

In this setting, 
the self-similar groupoid $\cGss$ can be identified with an action groupoid $\calG_{E} \rtimes G$ associated to a canonical action of $G$ on the path groupoid
\begin{equation} 
\calG_{E} = \{(\mu x, |\mu| - |\nu|, \nu x): s(\mu) = s(\nu), x\in s(\nu) E^{\infty}\},
\end{equation} 
which provides an alternative proof to Remark~\ref{genuine_C*}. Indeed, $G$ acts on $\calG_{E}$
by
\begin{align*} 
g\cdot (\mu x, |\mu| - |\nu|, \nu x) 
:= &\big((g\cdot \mu) (g\cdot x), |g\cdot \mu| - |g\cdot \nu |, (g\cdot \nu) (g\cdot x)\big)
\\
=& \big(g\cdot (\mu x), |\mu| - |\nu |, g\cdot (\nu x)\big), 
\end{align*}
for any $g\in G$, $\mu, \nu \in E^{*}$ with $s(\mu) = s(\nu)$ and $x\in s(\nu) E^{\infty}$. 
Recall that, 
as a topological space, the action groupoid is the Cartesian product $\calG_{E} \times G$. We let the groupoid's unit space be
$(\calG_{E} \rtimes G)\z  = \calG_{E}\z 
\times \{1_{G}\}
\cong E^{\infty}$, and define
source and range maps 
\begin{equation} 
r(\gamma, g) = r(\gamma)
\quad\text{and}\quad
s(\gamma, g) = s(g^{-1} \cdot \gamma)
\end{equation} 
for $g\in G, \gamma \in \calG_{E}$, 
and multiplication 
\begin{equation} 
(\gamma, g) (\zeta, h) = (\gamma (g\cdot \zeta), gh) 
\quad\text{whenever }s(\gamma, g) = r(\zeta, h).
\end{equation} 
 An isomorphism $\calG_{E} \rtimes G\rightarrow \calG_{G,E}$ is given by 
\begin{equation} 
\big((\mu x, |\mu|-|\nu|, \nu x), g\big) \mapsto \big(\mu x; [g^{-1}], |\mu|-|\nu|; (g^{-1}\cdot \nu)(g^{-1}\cdot x)\big), 
\end{equation} 
where,
by a slight abuse of notation,
$[g^{-1}]$ denotes the equivalence class of the constant function with value $g^{-1}$. 
\end{rem} 

Pseudo-freeness plays a crucial role in the study of self-similar graph $C^*$-algebras. 
Under the assumption of pseudo-freeness, the cylinder sets $\calZ(\mu, g, \nu)$ form a basis for a topology, under which $\cGss$ is a second-countable Hausdorff ample groupoid. 
This is first established through a topological groupoid isomorphism $\cGss \cong \calG_{\mathrm{tight}} (\calS_{G, E})$, where $\calG_{\mathrm{tight}} (\calS_{G, E})$ is the tight groupoid associated to the inverse semigroup coming from $(G, E)$ 
(\cite[Theorem 8.19, Proposition 9.4]{Self-similar:EP17}); see also the more direct argument in \cite[Theorem~5.8]{Self-similar:LY21}. Moreover, in \cite{Self-similar:EP17} the $C^*$-algebra isomorphism $\mathcal{O}_{G, E}\cong C^{*}(\cGss)$ is obtained  by first identifying $\mathcal{O}_{G, E}$ with $C^*(\calG_{\mathrm{tight}} (\calS_{G, E}))$, and then assuming pseudo-freeness and using the groupoid isomorphism $\cGss \cong \calG_{\mathrm{tight}} (\calS_{G, E})$.

While non-pseudo-free self-similar actions have also been studied (see, for instance, \cite{KK_duality_ss}), this approach involves
non-Hausdorff groupoids and inverse semigroup techniques. In the present work, we restrict our attention to the pseudo-free setting, thereby working with Hausdorff groupoids admitting a basis of cylinder sets. To clarify the role of pseudo-freeness in relation to the existing literature, we indicate explicitly each instance in which this assumption is used and how it enters our arguments. 

Given a self-similar pseudo-free action $(G,E)$,
if $G$ is amenable, then $\cGss$ is amenable.
In particular, its full $C^*$-algebra $C^*(\cGss)$ and reduced $C^*$-algebra $\csr(\cGss )$ coincide, and thus isomorphic to $\mathcal{O}_{G,E}$. 

We end this subsection with the following standard lemmas for later use, which follows essentially from the argument of \cite[Theorem 5.8]{Self-similar:LY21}. We include a quick proof for completeness. 

\begin{lem}\label{lem:opensetexists} Let $(G,E)$ be a pseudo-free self-similar graph, and let \[\mfkg 
   = \bigl(\mu (g\cdot x)   ;  \calT_{|\mu|}([g|_{x}])  ,  |\mu|-|\nu|   ;  \nu x\bigr)
   \in \cGss. 
   \] 
Then every
neighborhood of $\mfkg$ contains a neighborhood of the form
\[
  \calZ\bigl(\mu (g\cdot x)(0,n),  g|_{x(0,n)},  \nu x(0,n)\bigr)
\]
for some $n \in \mathbb{N}$.
\end{lem}

\begin{proof} 
Since $(G, E)$ is pseudo-free, it follows that the cylinder sets form a basis of the topology. Thus, any neighborhood of $\mfkg$ contains a set of the form $\calZ(\alpha,h,\beta)$ that, in turn, contains $\mfkg$. Note that for any $m\in \mathbb{N}$, we have 
\begin{equation}
\mfkg \in \calZ\bigl(\alpha (h\cdot y(0,m)), h|_{y(0,m)}, \beta y(0,m)\bigr)
\subseteq \calZ(\alpha,h,\beta) .
\end{equation} 
Thus, by replacing $(\alpha,h,\beta)$ with the triple $(\alpha (h\cdot y(0,m)), h|_{y(0,m)}, \beta y(0,m))$, we can without loss of generality assume that $n \coloneqq |\alpha|- |\mu| > 0$.

It follows from $\mfkg\in\calZ(\alpha,h,\beta)$ that 
\begin{equation} 
\mfkg = \bigl(\alpha (h\cdot y);  \calT_{|\alpha|}([h|_{y}]), |\alpha|-|\beta|; \beta y\bigr) 
\quad
\text{ for some }y\in s(\beta) E^{\infty}.
\end{equation} 
Comparing the coordinates of the two different representations of $\mfkg$, we conclude that $\mu (g\cdot x) = \alpha (h\cdot y)$, $|\mu| - |\nu| = |\alpha| - |\beta|$, and $\nu x = \beta y$. As $n>0$, it follows from the first equality that  $\mu (g\cdot x(0,n))=\alpha$ and $g|_{x(0,n)} \cdot \sigma^{n}(x) = h \cdot y$. The second equality implies $|\beta| - |\nu| = n$ which, combined with the third equality, yields $\nu x(0,n)=\beta $ and $\sigma^{n}(x) = y$.
In particular, $h\cdot y = g|_{x(0,n)}\cdot y$.

Lastly, comparing the two different forms of $\mfkg$ implies that $\calT_{|\mu|}([g|_{x}]) = \calT_{|\alpha|}([h|_{y}])$. The left-hand side can be reformulated as follows by applying \cite[Lemma 5.1]{Self-similar:LY21}:
\begin{equation} 
\calT_{|\mu|}([g|_{x}]) = \calT_{|\mu|}([g|_{x(0,n) \sigma^{n}(x)}]) = \calT_{|\mu|+n}([g|_{x(0,n)}|_{\sigma^{n}(x)}]) = \calT_{|\alpha|}([g|_{x(0,n)}|_{y}]). 
\end{equation} 
Since this equals $\calT_{|\alpha|}([h|_{y}])$, we conclude $[g|_{x(0,n)}|_{y}] = [h|_{y}]$. By pseudo-freeness, we can invoke \cite[Corollary 5.6]{Self-similar:LY21}, which implies that $g|_{x(0,n)}=h$. All in all, we have shown that
\[
(\alpha,h,\beta)
=
\bigl(\mu (g\cdot x(0,n)), g|_{x(0,n)}, \nu x(0,n)\bigr),
\]
which finishes the claim. 
\end{proof}

\subsection{Cycline triples and local faithfulness} 
\label{subsec:ct_lf} 
We first recall the notion of cycline triples, which is introduced in \cite[Section 3]{Self-similar:LY19}, to study \emph{periodicity} of $(G,E)$. 

\begin{defn}
\label{defn:cycline_triple} 
Given a self-similar graph $(G,E)$, finite paths $\mu,\nu \in E^*$ and $g \in G$, we say that $(\mu,g,\nu)$ is a \textit{cycline triple at} $s(\nu)$ if 
$s(\mu)=g \cdot s(\nu)$ and we have
$\mu(g \cdot x)=\nu x$ for all $x \in s(\nu)E^\infty$.
A cycline triple 
 is called \textit{symmetric} if it is of the form $(\mu,g,\mu)$
and \textit{trivial} if it is of the form $(\mu, 1_{G}, \mu)$. 
\end{defn} 

Note that a cylinder set $\calZ (\mu,g,\nu)$ is contained in $\Iso(\cGss)$ if and only if $(\mu,g,\nu)$ is a cycline triple. When $(G, E)$ is pseudo-free, cylinder sets form a basis of the topology of $\cGss$, which implies that 
\begin{equation}\label{eq:int of Iso} 
    \Iso(\cGss)^{\circ} = \bigcup_{\substack{(\mu,g,\nu)\\\text{cycline}}}\calZ (\mu,g,\nu).
\end{equation}
This fact can also be found in \cite[Lemma 5.2]{Self-similar:LY19}. 

As noted in the Introduction, our main goal is to identify an open abelian subgroupoid of $\Iso(\cGss)^{\circ}$ that yields a Cartan subalgebra of $C^*_r(\cGss)$. This requires understanding commutativity within $\Iso(\cGss)^{\circ}$. The relevant criterion is given in the following lemma, whose proof can be extracted from the first part of \cite[Lemma 5.3]{Self-similar:LY19} in the $1$-graph case. 

\begin{lem}
\label{lem:comm_rela} 
If $(\mu, g, \nu)$ and $(\alpha, h, \beta)$ are cycline triples, then all elements belonging to the cylinder sets
$\calZ(\mu,g,\nu)$ and $\calZ(\alpha, h, \beta)$ commute if and only if for every $x\in s(\nu)E^{\infty}$ and $y\in s(\beta)E^{\infty}$, we have
\begin{equation} \label{comm_rela} 
\big[g|_{x(0,|\alpha|)}|_{\sigma^{|\alpha|}(x)} \, h|_{y(0,|\nu|)}|_{\sigma^{|\nu|}(y)} \big] = 
\big[h|_{y(0,|\mu|)}|_{\sigma^{|\mu|}(y)}
g|_{x(0,|\beta|)}|_{\sigma^{|\beta|}(x)}\big].  
\end{equation}
\end{lem} 

An easy observation is that~\eqref{comm_rela} always holds in the genuine action case. 

\begin{cor}
\label{C:can-act} 
Let $G$ be an abelian group and let $(G, E)$ be a genuine action in the sense of Definition~\ref{D:action}. Then the subgroupoid $\Iso(\cGss)^\circ$ of $\cGss$ is abelian. 
\end{cor} 
In general, identifying maximal abelian subgroupoids of $\Iso(\cGss)^\circ$ is more intricate. Under the additional assumption that $(G, E)$ is \emph{locally faithful}, it is shown in \cite[Lemma~5.3]{Self-similar:LY19} that the equality~\eqref{comm_rela} is always satisfied, thus implying that $\Iso(\cGss)^\circ$ is abelian. In this setting, $\Iso(\cGss)^\circ$ gives rise to a Cartan subalgebra when 
$\Iso(\cGss)^\circ$ is closed (for instance when $(G, E)$ is strongly connected by \cite[Proposition 5.5]{Self-similar:LY19}), as an application of \cite[Corollary 4.5]{BNRSW:Cartan}.  

Now we turn to the notion of local faithfulness, introduced in \cite{Self-similar:LY19}. Recall that a self-similar graph $(G, E)$ is \emph{locally faithful} if for every vertex $v$ in $E$, whenever $g\in G$ is such that $g\cdot \mu = \mu$ for all $\mu\in vE^{*}$, then $g = 1_{G}$. As we will see later, 
in the $1$-graph case, 
local faithfulness forces all cycline triples at every vertex to be trivial (Corollary~\ref{conclusion}), 
thereby recovering the abelianness of $\Iso(\cGss)^\circ$ directly from Lemma~\ref{lem:comm_rela}. In this paper, however, we are interested in more general cases, where the nice phenomena of local faithfulness occurs only at certain vertices. To capture this, we introduce the following vertex-wise notion of local faithfulness. 

\begin{defn}\label{def:(n)lf} 
Let $(G, E)$ be a self-similar graph. A vertex $v\in E^{0} $ is \textit{locally faithful} if, whenever $g\in G$ is such that $g\cdot \mu = \mu$ for all $\mu\in vE^*$,
then $g = 1_{G}$. 
Otherwise, the vertex 
$v$ 
is called \textit{not locally faithful}. 
\end{defn} 

Note that $(G, E)$ is locally faithful if and only if every $v\in E^{0}$ is locally faithful. When both locally faithful and non-locally faithful vertices are present, $\Iso(\cGss)^{\circ}$ is not necessarily abelian, and it is not clear how to identify maximal abelian subgroupoids of $\Iso(\cGss)^{\circ}$. This motivates us to consider a special abelian subset $\ssym\subseteq \Iso(\cGss)^{\circ}$, consisting of cylinder sets corresponding to all symmetric cycline triples. Its abelianness follows from the following direct corollary of Lemma~\ref{comm_rela}. 

\begin{cor} \label{cor:cycline comm}
Let $(G, E)$ be a self-similar graph with $G$ abelian. 
Given two {\em symmetric} cycline triples $(\mu, g, \mu)$ and $(\alpha, h, \alpha)$,
all groupoid elements in $\calZ(\mu, g, \mu)$ and $\calZ(\alpha, h, \alpha)$ commute. 
\end{cor}  

In Section~\ref{S:sym}, we show that $\ssym$ is a normal subgroupoid and thus serves as a candidates for yielding a Cartan subalgebra.

While the proofs in the rest of this paper do not make explicit use of the following lemma, we record it here since it came in handy repeatedly when developing intuition about self-similar graphs and constructing (counter)examples. 

\begin{lem} \label{lem:range is nlf}
    Let $(G, E)$ be a pseudo-free self-similar graph, and assume that $\mu \in E^*$. If $r(\mu)$ is not locally faithful, then neither is $s(\mu)$.
\end{lem}

\begin{proof}
   By assumption, there exists $g\in G\setminus\{1_{G}\}$ such that $g\cdot \alpha=\alpha$ for all $\alpha\in r(\mu)E^*$. Since $(G,E)$ is pseudo-free and $g\neq 1_{G}$, this implies that 
   \begin{equation}\label{eq:fix but nontrivial}
       \mu\beta = g\cdot (\mu\beta) 
       \quad\text{and}
       \quad
       g|_{\mu\beta}\neq 1_{G}
       \quad
       \text{ for any }\beta\in s(\mu)E^*
       .
   \end{equation}
   As $s(\mu)\in s(\mu)E^*$, it follows that
   \[
        \mu\beta
        \overset{\eqref{eq:fix but nontrivial}}{=}
        (g\cdot \mu)(g|_{\mu}\cdot\beta)
        =\mu(g|_{\mu}\cdot\beta), 
        \text{ so }\beta=g|_{\mu}\cdot \beta,       
       \quad\text{and}
       \quad
       g|_{\mu}\neq 1_{G}.
   \]
   This shows that $g|_{\mu}\in G_{s(\mu)}\setminus\{1_{G}\}$, so that $s(\mu)$ is not locally faithful. 
\end{proof}

\section{The symmetric cycline subgroupoid $\ssym$} 
\label{S:sym}
Let $(G,E)$ be a self-similar graph. We will use multiplicative notation for our countable discrete group $G$,
even when we are assuming that $G$ is abelian.
As mentioned in Subsection~\ref{subsec:ct_lf}, we consider the subset $\ssym\subseteq \Iso(\cGss)^{\circ}$ given by 
\begin{align} 
    \ssym
    \defeq 
    \bigcup_{\substack{(\mu,g,\mu)\\\text{cycline}}}\calZ (\mu,g,\mu), 
\end{align} 
which is always abelian by Corollary~\ref{cor:cycline comm}. Before we state the main result of this section, we include the following notation. 

\begin{defn}
    For a vertex $v\in E^{0}$, let 
    \begin{equation}\label{eq:G_v}
    G_{v} \defeq  \{g\in G: g\cdot x = x \text{ for all } x\in vE^{\infty}\},
    \end{equation} 
    the set of group elements fixing all infinite paths with range $v$. 
\end{defn} 

We record here the following easy consequence of \cite[Lemma 3.7(i)]{Self-similar:LY19}.
\begin{lem}\label{lem:G_v subgp}
   For $v\in E^{0}$, $G_{v}$ is a subgroup of $G$.
\end{lem} 

Notice that a triple $(\mu, g,\nu)$ is a symmetric cycline triple at $s(\nu)$ if and only if $\mu = \nu$ and $g\in G_{s(\nu)}$. This leads to the following useful reformulation of $\ssym$: 
\begin{align} 
\ssym = \bigcup_{v\in E^{0}} \bigcup_{g\in G_v } \bigcup_{\mu\in E^* v}\calZ(\mu, g, \mu).
\end{align} 

Our main goal in this section is to prove the following.

\begin{prop}\label{prop:Ssym:open,normal,abelian,subgpd}
    Let $G$ be a countable, discrete, abelian group, and let $(G,E)$ be a self-similar graph satisfying our standing assumptions on page~\pageref{Notation and conventions}. Then the set $\ssym$ is an open, abelian, normal subgroupoid of $\cGss$. 
\end{prop} 
Subgroupoids with such nice properties allow us to construct new (twisted) groupoid models for self-similar graph $C^*$-algebras;
see Section~\ref{sec:Renault's reconstruction}.
We will prove Proposition~\ref{prop:Ssym:open,normal,abelian,subgpd} in multiple steps.

\begin{lem}\label{lem:Ssym is subgroupoid}
The subset $\ssym$ of $\Iso(\cGss)^\circ$ 
is an abelian subgroupoid. 
\end{lem} 

\begin{proof}
Since we assumed $G$ to be abelian, it follows from Corollary~\ref{cor:cycline comm} that the elements of $\ssym$ commute. 
We can write $\mfkt\in \ssym$ as
\begin{equation*} 
\mfkt = (\lambda y;  \ \calT_{|\lambda|}([h|_{y}]), \ 0; \ \lambda y), 
\end{equation*} 
where $h$ fixes all infinite paths ending at $s(\lambda)$. Then since $h\cdot y = y = h^{-1}\cdot y$, 
\begin{equation}\label{eq:t-inv in ssym}
\mfkt^{-1} = (\lambda y;  \ \calT_{|\lambda|}([h|_{y}])^{-1}, \ 0; \ \lambda y) = (\lambda y; \ \calT_{|\lambda|}([h^{-1}|_{y}]), \ 0;  \ \lambda y)\in \ssym. 
\end{equation} 
Another element $\mfks  = (\mu x; \ \calT_{|\mu|}([g|_{x}]), \ 0; \ \mu x)\in \ssym$ is composable with $\mfkt$ only when $\lambda y = \mu x$. Without loss of generality, we can assume that $|\mu| < |\lambda|$, i.e.\ $\lambda = \mu \gamma$ for some $\gamma \in E^*$ and thus $x = \gamma y$. Then we have the following by \cite[Lemma 5.1]{Self-similar:LY21},
\begin{multline*}
    \mfkt\mfks  = (\lambda y;  \ \calT_{|\lambda|}([h|_{y}]) \calT_{|\mu|}([g|_{x}]), \ 0; \ \lambda y)
= (\lambda y; \ \calT_{|\lambda|}([h|_{y}])  \calT_{|\lambda|}([g|_{\gamma}|_{y}]), \ 0;  \ \lambda y)\\ 
= (\lambda y;  \ \calT_{|\lambda|}([h|_{y}  g|_{\gamma}|_{y}]), \ 0; \ \lambda y)
= (\lambda y;  \ \calT_{|\lambda|}([(h\cdot g|_{\gamma})|_{y}]), \ 0; \ \lambda y). 
\end{multline*}
Moreover, $h\cdot g|_{\gamma}$ fixes every infinite path ending at $s(\lambda)$.
Indeed, suppose $y \in s(\lambda)E^\infty$. Since $g$ fixes all infinite path whose range is $r(\gamma) = s(\mu)$ and $h$ fixes all infinite paths whose range is $s(\gamma) = s(\lambda)$, then 
 \begin{align*}
h (g|_{\gamma}) \cdot y 
&=  h(g|_{\gamma})\cdot \left(\sigma^{|\gamma|}(\gamma y)\right)
\overset{(\dagger)}{=} h \cdot \sigma^{|\gamma|} \left(g \cdot (\gamma y)\right) 
=h \cdot \sigma^{|\gamma|} ( \gamma y) 
 =h \cdot y 
 =y,
 \end{align*}
 where $(\dagger)$ follows from \cite[Lemma 3.7 (ii)]{Self-similar:LY21}. 
\end{proof}

The above motivates the following definition.

\begin{defn}
$\ssym$ is called the {\em symmetric cycline subgroupid of $\cGss$}.
\end{defn}

    Let $\calG$ be a groupoid and $\calG\z \subseteq\calS \subseteq \Iso
       (\calG)$ be a subgroupoid. Recall that we say $\calS$ is \textit{normal} if for any $\mfkg\in\calG$,  the set $\mfkg\calS \mfkg^{-1} \defeq  \{\mfkg \mfkt \mfkg^{-1}: \:\mfkt\in \calS, \,  r(\mfkt) = s(\mfkg)\}$ is contained in $\calS$.

\begin{lem}\label{lem:Ssym is normal} 
The subgroupoid $\ssym\subseteq \cGss $ is normal. 
\end{lem} 

\begin{proof} 
Suppose
$\mfkt\in \ssym$ and $\mfkg\in \cGss$ are composable, meaning that we can write
\begin{equation*} 
\mfkt  = (\lambda y; \ \calT_{|\lambda|}([h|_{y}]), \ 0; \ \lambda y) \quad\text{and}\quad \mfkg  = (\mu (g\cdot x);  \ \calT_{|\mu|}([g|_{x}]), \ |\mu|-|\nu|; \ \nu x)
\end{equation*}
for some $\lambda,\mu, \nu \in E^*$, $g\in G$, $h\in G_{s(\lambda)}$ and $x\in s(\nu)E^{\infty}, y\in s(\lambda)E^\infty$ which satisfy $s(\mu) = g \cdot s(\nu)$ and $\lambda y = \mu (g\cdot x)$. 
Using Equation~\eqref{eq:mfkg inv} and that  the group is abelian, we have the following by standard computations,
\begin{align} 
\label{eq:end of ginv t g, both cases} 
\mfkg^{-1}\mfkt \mfkg  = \left(\nu x;  \calT_{|\nu| - |\mu|+|\lambda|}([h|_{y}]), \ 0; \ \nu x\right). 
\end{align} 
We must now prove that $\mfkg^{-1}\mfkt \mfkg$ is an element of $\ssym$. To this end, we will make a case distinction.

First, assume that $|\lambda|-|\mu|\leq 0$. In this case, it suffices to prove that
\begin{align}\label{eq:Ssym is normal:case1}
    \mfkg^{-1}\mfkt \mfkg  
    =
       \left(\nu x; \calT_{|\nu|}([h|_{\sigma^{|\lambda|}(\mu)}|_{x}]), \ 0; \ \nu x
    \right),
    \end{align} 
    and that $h|_{\sigma^{|\lambda|}(\mu)}$ fixes all infinite paths at $s(\nu)$. Since $\lambda y = \mu(g\cdot x)$, it follows from $|\mu| \geq |\lambda|$ that  $y = \gamma (g\cdot x)$ where  $\gamma =\sigma^{|\lambda|}(\mu)$, so that
\begin{align} 
\mfkg^{-1} \mfkt \mfkg 
&= \left(\nu x;  \calT_{|\nu| - |\mu|+|\lambda|}([h|_{\gamma (g\cdot x)}]), \ 0; \ \nu x
\right)
&&\text{by~\eqref{eq:end of ginv t g, both cases} since $y = \gamma (g\cdot x)$}
\\&
= \left(\nu x;  \calT_{|\nu|}([h|_{\gamma}|_{g\cdot x}]), \ 0; \ \nu x
\right)
&&\text{by \cite[Lemma 5.1]{Self-similar:LY21}}.
\label{eq:end of ginv t g, case 1, part 2}
\end{align}
To see that $\mfkg^{-1} \mfkt \mfkg \in\ssym$, we first prove $[h|_{\gamma}|_{x}] = [h|_{\gamma}|_{g\cdot x}]$. To this end, note that the assumptions $y = \gamma (g\cdot x)$ and $h\cdot y = y$ imply $h|_{\gamma} \cdot (g\cdot x) = g\cdot x$. Acting by $g^{-1}$ on the left of both sides of the latter equation yields, since $G$ is abelian, that $h|_{\gamma} \cdot x=x$. 
With that, we have for $n\in \N $,
\begin{align*} 
(g|_{x(0,n)})(h|_{\gamma}|_{g\cdot x(0,n)}) &=(h|_{\gamma}|_{g\cdot x(0,n)})(g|_{x(0,n)})
&&\text{since $G$ is abelian}
\\ 
& = (h|_{\gamma} g)|_{x} (n) 
&& \text{by self-similarity and definition of $|_{x}$} 
\\ 
&= (g|_{x(0,n)}) (h|_{\gamma}|_{x(0,n)}) 
&&\text{since }G \text{ is abelian and }h|_{\gamma} \cdot x = x. 
\end{align*}

Canceling $g|_{x(0,n)}$ on both sides explains $(\dagger)$ in the following:
\begin{align*}
    h|_{\gamma}|_{g\cdot x}(n)
    &=
    h|_{\gamma}|_{(g\cdot x)(0,n)}
    =
    h|_{\gamma}|_{g|_{x(0,0)}\cdot x(0,n)}
    =
    h|_{\gamma}|_{g\cdot x(0,n)}
    \overset{(\dagger)}{=}
    h|_{\gamma}|_{x(0,n)}
    =
    h|_{\gamma}|_{x}(n).
\end{align*}
In particular, it follows that $[h|_{\gamma}|_{x}] = [h|_{\gamma}|_{g\cdot x}]$, so by~\eqref{eq:end of ginv t g, case 1, part 2}, Equation~\eqref{eq:Ssym is normal:case1} holds.
Since $\mfkt \in\ssym$,  $h$ fixes all infinite paths at $s(\lambda)$, which implies that $h|_{\gamma}$ fixes all infinite paths at $s(\mu)$.
We claim $h|_\gamma$ fixes all infinite paths with range $s(\nu)$ as well. Indeed, 
if $z\in s(\nu)E^\infty$, then $g\cdot z \in (g\cdot s(\nu))E^\infty=s(\mu)E^\infty$. 
Since $G$ is abelian, we have
    $g\cdot z
    =
    h|_{\gamma}\cdot (g\cdot z)
    =
    (h|_{\gamma}g)\cdot z
    = (gh|_{\gamma})\cdot z
    =
    g\cdot (h|_{\gamma}\cdot z)$.
Acting on both sides by $g\inv$ shows $z=h|_{\gamma}\cdot z$.
Thus, $(\nu, h|_\gamma, \nu)$ is cycline. Therefore $\mfkg^{-1}\mfkt \mfkg \in \calZ({\nu, h|_{\gamma}, \nu})
\subseteq \ssym$, as claimed.
This finishes the proof for the case $|\lambda|-|\mu|\leq 0$.

Next, assume that $p\defeq |\lambda|-|\mu|>0$. It suffices to prove that
\begin{align}\label{eq:Ssym is normal:case2}
    \mfkg^{-1}\mfkt \mfkg  
    =
        \left(\nu x(0,p) \sigma^{p}(x);  \calT_{|\nu|+p}([h|_{\sigma^{p}(x)}]), \ 0; \ \nu x(0,p) \sigma^{p}(x)\right),    
\end{align}
and that $h$ fixes all infinite paths at $r(\sigma^{p}(x))$.
It follows from $\lambda y = \mu(g\cdot x)$ and from \cite[Lemma 3.7(ii)]{Self-similar:LY21} that
\[
y = \sigma^{p}(g\cdot x) = g|_{x(0,p)}\cdot \sigma^{p}(x)
.
\]
Let $\nu'=\nu x(0,p)$ and $x'=\sigma^{p}(x)$, so that by~\eqref{eq:end of ginv t g, both cases} and since $p=|\lambda|- |\mu|$,
\begin{align} 
\mfkg^{-1} \mfkt \mfkg 
&= \left(\nu x;  \calT_{|\nu| +p}([h|_{g|_{x(0,p)}\cdot x'}]), \ 0; \ \nu x
\right)\notag \\
&=
\left(\nu' x';  \calT_{|\nu'|}([h|_{g|_{x(0,p)}\cdot x'}]), \ 0; \ \nu' x'
\right)
.
\label{eq:end of ginv t g, case 2, part 2}
\end{align}
Since $G$ is abelian, it follows from $y = g|_{x(0,p)}\cdot \sigma^{p}(x)$ and $h\cdot y = y$ that $h\cdot x'=x'$; in particular, $h\cdot x'(0,n)=x'(0,n)$ for all $n\in \N $, since $h|_{x'(0,0)}=h$. A computation analogous to that above therefore yields that
\[
    h|_{g|_{x(0,p)}\cdot x'} = h|_{x'},
\]
and so by~\eqref{eq:end of ginv t g, case 2, part 2}, Equation~\eqref{eq:Ssym is normal:case2} holds.
Since $\mfkt \in\ssym$,  $h$ fixes all infinite paths at $s(\lambda) =  r(y)$. By similar reasoning to above, $h$ fixes all infinite paths at $r(x')=s(\nu')$ as well since $g|_{x(0,p)} \cdot r(x') = r(y)$. So $(\nu' , h, \nu')$ is cycline. Hence $\mfkg^{-1}\mfkt \mfkg \in 
\calZ({\nu', h, \nu'})\subseteq\ssym$.
\end{proof} 

Since $\ssym$ is open as a union of basic open sets, this finishes the proof of Proposition~\ref{prop:Ssym:open,normal,abelian,subgpd}.

\section{The dual of $\ssym$ and a new model for $\csr(\cGss[\Z])$} \label{sec:Renault's reconstruction}

Given a pseudo-free self-similar graph $(\Z,E)$, we will describe the dual bundle $\ssymdual$ of the abelian groupoid $\ssym$ in Theorem~\ref{thm:ssymdual}. This is of particular interest when one wants to provide a different groupoid model for $\csr(\cGss[\Z])$: a transformation groupoid of the quotient $\cGss[\Z]/\ssym$ acting `by conjugation' on $\ssymdual$.

Note that $\ssymdual$ is the spectrum of the abelian subalgebra $\csr(\ssym)$ of $\csr(\cGss)$ which, in some situations, is even a Cartan subalgebra (see Theorem~\ref{Thm_Ssym_Cartan}).
For more details on this action groupoid construction in the general setting, see \cite[Theorem 5.5 (ii)]{Renault:2023:Ext} (quoted in Theorem~\ref{thm:Renaults5.5} below); for examples of self-similar graphs for which we compute this action groupoid explicitly, see Appendix~\ref{app:Weyl examples}.

\subsection{The dual of $\ssym$}

We start by reminding the reader of the definition of the dual bundle for an abelian (locally compact Hausdorff) groupoid $\calS$. For a unit $x\in \calS\z$, we let $\widehat{\calS}(x)$ denote the Pontryagin dual of the abelian group $\calS(x)$, i.e. 
\begin{equation} 
\widehat{\calS}(x) \coloneqq \{\kappa: 
\calS(x)\to \mathbb{T}: \kappa \text{ continuous group homomorphism}\}.
\end{equation} 
The group operation of $\widehat{\calS}(x)$ is defined pointwise, and, equipped with the topology of uniform convergence on compact sets, $\widehat{\calS}(x)$ is locally compact. If $\calS$ is \'etale, so that $\calS(x)$ is discrete, then $\widehat{\calS}(x)$ is compact and convergence $\kappa_{n}\to\kappa$ of a sequence in $\widehat{\calS}(x)$ is equivalent to convergence $\kappa_{n}(\mfks)\to \kappa(\mfks)$ in $\mathbb{T}$ for every $\mfks\in\calS(x)$.

The group bundle $\widehat{\calS}$ of all groups $ \widehat{\calS}(x)$ for $x\in \calS\z$ is given the topology as described in \cite[Proposition
3.3]{MRW:1996:CtsTrace}, making it a topological groupoid. The precise description of this topology is not essential for us, as our focus will be on describing the fibres $ \ssymdual(x)$ for any fixed infinite path $x\in E^\infty=\ssym\z$ rather than $\ssymdual$. 

We start by finding an easier description of the group 
\begin{align}\notag
\ssym(x) &=
\left\{
    \left(
        x;
        \calT_{n}([f|_{\sigma^n (x)}]), 0;
        x
    \right)
    \mid
    n\in \N,
    f\in \Z_{x(n,n)}
\right\} \leq \cGss[\mathbb{Z}](x) .
\end{align}
By its definition, multiplication in $\ssym$ ``works as in $Q(\N,\Z)$,'' meaning that $\ssym(x)$ naturally corresponds to the subgroup of $Q(\N,\Z)$ given by 
\begin{align}\label{eq:def frak-S}
\mathfrak{S}(x)
\defeq
\left\{
        \calT_{n}([f|_{\sigma^n (x)}])
    \mid
    n\in \N,
    f\in \Z_{x(n,n)}
\right\}.
\end{align}
For convenience, we record how elements of this subgroup multiply:
\begin{align}\label{eq:mult in mathfrakS}
    \calT_{n}([f|_{\sigma^{n} (x)}]) \, \calT_{m}([g|_{\sigma^m (x)}])
    =
    \calT_{\ell}
    \left(
    \Bigl[
        \bigl(
        f|_{x(n,\ell)}g|_{x(m,\ell)}
        \bigr)|_{\sigma^{\ell}(x)}
    \Bigr]
    \right)
    ,
\end{align}
where $\ell\in\N$ is any number for which $n,m\leq \ell $.

To describe $\mathfrak{S}(x)$ (and, in turn, $\ssym(x)$), we analyse $\Z_{x(n,n)}$ for $n\in \N$. Since $\Z_{x(n,n)}$ is a subgroup of $\Z$, it is singly generated, say\footnote{Here, we abuse notation when we write $h_{n}$ instead of, say, $h_{x(n,n)}$ or $h_{x,n}$, since the generator of $\Z_{x(n,n)}$ depends not on the natural number $n$ but rather on the vertex $x(n,n)$.}
\begin{align}\label{eq:def'n h_n}
\Z_{x(n,n)} = \langle h_{n}\rangle
\text{ for some (unique) }h_{n}\in\N.
\end{align}
 
Using the axioms for self-similarity, it is easy to prove that these integers relate to one another as described in the following lemma:
\footnote{
We remind the reader of our convention to write $1_{G}$ for the identity element of $\Z$ (see page~\pageref{Notation and conventions}). Moreover, our chosen symbol $k_{n}$ is again not as precise as (but less cumbersome than) $k_{x(n,n)}$ or $k_{x,n}$ would have been.}

\begin{lem}\label{lem:k_n}
    For each $n\in \N$, there exists $k_{n}\in \Z$ such that $h_{n}|_{x(n,n+1)}=h_{n+1}^{k_{n}}
    $. If $h_{n+1}\neq 1_{G}
    $, then $k_{n}$ is unique; if $h_{n+1}=1_{G}
    $, then we
    choose $k_{n}=0$.
\end{lem}

To state our first intermediate result, which 
describes
a generating set for $\ssym(x)$ and how these generators multiply, we need one more bit of notation. For $i,j\in\N$, $i<j$, we define the numbers
\begin{align}\label{def:k_i,j}
    k_{i,i}=1
    \quad\text{ and }\quad
    k_{i,j}:=\prod_{i\leq n< j}k_{n}
    .
\end{align}
Note that $k_{i,i+1}=k_{i}$, and that
\begin{equation}\label{eq:h_n = 1 means k_i,j = 0}
    h_{n}=1_{G} \implies k_{i,j}=0\text{ for all }i\leq n< j,
\end{equation}
since the equality $1_{G}=h_{n}|_{x(n,n+1)}=h_{n+1}^{k_n}$ means that $k_{n}=0$, which appears in the product that defines $k_{i,j}$.

\begin{prop}\label{prop:mult in ssym, revised}
   Let $(\Z,E)$ be a self-similar graph, and fix $x\in E^\infty$. 
   If we let 
   \begin{equation}\label{def:xi_n}
        \mfks_{n}        
        \defeq \left(
            x;
            \calT_{n}([h_{n}|_{\sigma^n (x)}]), 0;
            x
        \right)
        \text{ where $h_{n}$ generates }
        \Z_{x(n,n)},
    \end{equation}
    then $\ssym(x)$ is generated by $\{\mfks_{n}:n\in\N\}$. Moreover, multiplication in $\ssym(x)$ is given 
    for $p,q\in\Z$ and $ m , n \leq \ell\in\N$ by
    \begin{align}\label{eq:mult in ssym(x)}
    \mfks_{ m }^p \, \mfks_{ n }^q = \mfks_{\ell}^{p k_{ m ,\ell} + q k_{ n ,\ell}}.
    \end{align}
\end{prop}
Note that~\eqref{eq:mult in ssym(x)} applied to $q=0$, $p=1$, and $m< \ell\in\N$ yields in particular:
    \begin{align}\label{eq:s_m vs s_ell}
    \mfks_{ m } = \mfks_{\ell}^{k_{m}k_{m+1}\dots k_{\ell-1}}.
    \end{align}

To prove the proposition, we will need the following identities:

\begin{lem}\label{it:powers go in:relationship of generators of ssym}
    Suppose $m\leq n\leq p<q$ are natural numbers, and let $k_{n}    $ be as 
    in Lemma~\ref{lem:k_n}.
    Then for all $\ell \in\Z$, we have
       $h_{m}|_{x(m,q)}=(h_{m+1}|_{x(m+1,q)})^{k_{m}}$.
    In particular,
    \[
    \calT_{m}([h_{m}|_{\sigma^{m}(x)}])=\calT_{m+1}([h_{m+1}|_{\sigma^{m+1}(x)}])^{k_{m}}.
    \]
\end{lem}

\begin{proof}
    Follows from combining Lemma~\ref{it:powers go in} 
    with an induction on $q-m$.
\end{proof}

\begin{proof}[Proof of Proposition~\ref{prop:mult in ssym, revised}]
We deduce the following equation of sets within $Q(\N,\Z)$:
\begin{align}
    \mathfrak{S}(x)
    &=
    \{
    \calT_{n}([g|_{\sigma^{n}(x)}])
    : n\in\N, g\in \Z_{x(n,n)}
    \}
    &&\text{by def'n of  }\mathfrak{S}(x)
    \\
    &=
    \{
    \calT_{n}([h_{n}^{k}|_{\sigma^{n}(x)}])
    : n\in\N, k\in \Z
    \}
    &&\text{by choice of }h_{n}
    \\
    &=
    \{
    \calT_{n}([h_{n}|_{\sigma^{n}(x)}])^{k}
    : n\in\N, k\in \Z
    \}
    &&\text{by Lemma~\ref{it:powers go in}}
    .
\end{align}
    By Equation~\eqref{eq:mult in mathfrakS}, we have
    \begin{align}
    \calT_{n}([h_{n}^{p}|_{\sigma^{n} (x)}]) \cdot \calT_{m}([h_{m}^{q}|_{\sigma^m (x)}])
    =
    \calT_{\ell}
    \left(\Bigl[
        \bigl(
        h_{n}^{p}|_{x(n,\ell)} \, h_{m}^{q}|_{x(m,\ell)}
        \bigr)|_{\sigma^{\ell}(x)}
    \Bigr]
    \right)
\end{align}
in $Q(\N,\Z)$. By Lemma~\ref{it:powers go in} and
repeated use of~\ref{it:powers go in:relationship of generators of ssym},
we have
    \[
        h_{n}^{p}|_{x(n,\ell)}
        =
        (h_{n}|_{x(n,\ell)} )^{p}
        =
        h_{\ell}^{pk_{n,\ell}}
        \quad
        \text{ and }
        \quad
        h_{m}^{q}|_{x(m,\ell)}
        =
        h_{\ell}^{qk_{m,\ell}},
    \]
    so that
    \[
        \bigl(
        h_{n}^{p}|_{x(n,\ell)} \, h_{m}^{q}|_{x(m,\ell)}
        \bigr)|_{\sigma^{\ell}(x)}
        =
        h_{\ell}^{p k_{n,\ell}+q k_{m,\ell}}|_{\sigma^{\ell}(x)}.
    \]
    Another application of Lemma~\ref{it:powers go in} then proves
    \begin{align}
    \calT_{n}([h_{n}|_{\sigma^{n} (x)}])^{p} \calT_{m}([h_{m}|_{\sigma^m (x)}])^{q}
    =
    \calT_{\ell}
    \left(\left[
        h_{\ell}|_{\sigma^{\ell}(x)}
    \right]\right)^{p k_{n,\ell}+q k_{m,\ell}},
\end{align}
which was the claim.
\end{proof}

Out of the numbers $k_{n}$, we will now construct an inverse system $(X_{i},f_{i,j})$ of topological groups which, as long as the self-similar graph is pseudo-free, is isomorphic to the fibre $\ssymdual(x)$. 

If $h_{i}=1_{G}$ for all $i\in \N$, let $X_{i}(x)\defeq X_{i}\defeq\{1\}$; otherwise,  there exists a minimal $N\in\N$ with $h_{i}\neq 1_{G}$. We let
\[
    X_{i}(x)\defeq X_{i} \defeq \begin{cases}
        \{1\} & \text{ if } i< N
        \\
        \mathbb{T}  & \text{ if } i\geq N
    \end{cases},
\]
    and for $i\leq j$,
\[
    f_{i,j}\colon X_{j}\to X_{i},
    \quad
    z\mapsto z^{k_{i,j}}.
\]
Note that this is well-defined, for if $i<N\leq j$, then $k_{i,j}=0$ by~\eqref{eq:h_n = 1 means k_i,j = 0}, so that $z^{k_{i,j}}=1\in X_{i}$.
As $((X_{i})_{i\in\N},(f_{i,j})_{i\leq j\in \N})$ is an inverse system of topological groups, we may let 
\begin{equation}\label{eq:def X}
    X(x)\defeq X
    \defeq\varprojlim_{i}\, (X_{i}, f_{i,j})
\end{equation}
denote its inverse limit, and let $\pi_{i}\colon X\to X_{i}$ be the natural projections satisfying $\pi_{i}=f_{i,j}\circ\pi_{j}$.
Recall that $X(x)$ is equipped with the coarsest topology that makes all maps $f_{i,j}$ continuous.

We can finally describe the fibre $\ssymdual(x)$ of the dual group bundle $\ssymdual$:

\begin{thm} \label{thm:ssymdual}
   Let $(\Z,E)$ be a self-similar graph, and fix any infinite path $x\in E^\infty$. 
   There exists an injective continuous homomorphism
  \[
  \Omega\colon \ssymdual(x)\to X(x)=\varprojlim_{i}\, (X_{i} (x) , f_{i,j})
    \text{ uniquely determined by } (\pi_{n}\circ\Omega)(\kappa)
  =
  \kappa (\mfks_{n})
  \]
  for   $\kappa \in \ssymdual(x)$, where $\mfks_{n}$ is at~\eqref{def:xi_n}.
  
  If $(\Z,E)$ is pseudo-free, then $\Omega$
   is an isomorphism of topological groups whose inverse maps an infinite sequence $\vec{z}\in X$ to the continuous homomorphism
  \[\Omega\inv(\vec{z})\colon\ssymdual(x)\to\mathbb{T}
  \]
  that is uniquely determined by 
  \begin{equation}
      \Omega\inv (\vec{z})\colon
      \quad
      \mfks_{n} \mapsto \pi_{n}(\vec{z})
      \quad\text{ for all }n\in\N.
  \end{equation}
\end{thm}

Note that we are using that the elements $\mfks_{n}$ generate $\ssym(x)$ (Proposition~\ref{prop:mult in ssym, revised}), so that $\Omega$ is indeed uniquely determined by the given formula.  To prove Theorem~\ref{thm:ssymdual}, let us first check that the map $\Omega$ exists.

\begin{lem}[Existence of $\Omega$]\label{lem:Omega exists}
    The continuous maps
    \begin{align*}
    \mfks^*_{i}\colon \ssymdual(x)&\to \T,
    &
    \mfks^*_{i}(\kappa)&\defeq \kappa (\mfks_{i}),
\end{align*}
    satisfy $\mfks^*_{i}=f_{i,j}\circ \mfks^*_{j}$ for all natural numbers $i\leq j$. Moreover, there exists a unique continuous group homomorphism $\Omega\colon \ssymdual(x)\to X(x)$
    such that the diagram
    \[
    \begin{tikzcd}
    & \ssymdual(x)\ar[rdd, "\mfks^*_{i}"]\ar[ldd, "\mfks^*_j"'] 
    \ar[d, "\Omega"]
    &
    \\[15pt]
    & X(x)\ar[rd, "\pi_{i}"']\ar[ld, "\pi_j"] &
    \\
    X_{j}(x)\ar[rr, "f_{i,j}"] && X_{i}(x)
    \end{tikzcd}
    \]
    commutes for all $i\leq j$, where as before $X(x)$ is the inverse limit as defined at~\eqref{eq:def X}.
\end{lem}
\begin{proof}
    First, we point out that $\mfks^*_{i}$ is indeed continuous by definition of the topology on $\ssymdual(x)$.
    Next note that the diagram makes sense: If $j\in\N$ is such that $X_{j}\defeq X_{j}(x)$ is $\{1\}$ and not $\T$, then we must have that $h_{j}=1_{G}$. This implies $\mfks_{j}=1_{\ssym(x)}$ and thus $\mfks^*_{j}(\kappa)=1 \in X_{j}$.

    We have
\begin{align*}
    f_{i,i+1}\circ \mfks^*_{i+1} (\kappa)
    &=
    f_{i,i+1}\circ\kappa (\mfks_{i+1})
    =
    \kappa(\mfks_{i+1})^{k_{i}}
    \\
    &\overset{(\dagger)}{=}
    \kappa\left(\mfks_{i+1}^{k_{i}}\right)
    \overset{(\ddagger)}{=}
    \kappa (\mfks_{i})
    =
    \mfks^*_{i} (\kappa)
\end{align*}
where $(\dagger)$ follows since $\kappa$ is a homomorphism and $(\ddagger)$ from 
Lemma~\ref{it:powers go in:relationship of generators of ssym}.
The general claim for arbitrary $i\leq j$ now follows from the fact that $f_{i,j}=f_{i,k}\circ f_{k,j}$ for all $i\leq k\leq j$. The universal property of $\varprojlim_{i\in\N} (X_{i}, f_{i})$ 
then implies the existence and uniqueness of the continuous map $\Omega$.
\end{proof}

We need one more intermediate lemma, which gives the reason why we assumed pseudo-freeness in the second half of Theorem~\ref{thm:ssymdual}.

\begin{lem}[Cascading down]\label{lem:cascading}
    Suppose that $(\Z,E)$ is pseudo-free, and let $m\leq n\in\N$ and $p,q\in\Z$.
    \begin{enumerate}[label=\textup{(\roman*)}]
        \item\label{it:cascading:minimal non-zero h_m} If $h_{m+1}=1_{G}$, then $h_{i}=1_{G}$ for all $0\leq i\leq m$ also.
        \item\label{it:cascading:Omega is well defined} If     $\mfks_{n}^{q}=\mfks_{m}^{p}
    $ in $\ssym(x)$,
    then either $h_{n}=1_{G}$ or $q=pk_{m,n}$.
    \end{enumerate}
\end{lem}

\begin{proof}
    \ref{it:cascading:minimal non-zero h_m} Lemma~\ref{lem:k_n} implies that $h_{m}|_{x(m,m+1)}$ is a power of $h_{m+1}$ and hence equal to $h_{m+1}=1_{G}$. 
    Since $h_{m}\in \Z_{x(m,m)}$, we have
    $h_{m}\cdot x(m,m+1)=x(m,m+1)$, and so by pseudo-freeness, we must have $h_{m}=1_{G}$. Iteratively, we see that $h_{i}=1_{G}$ for all $0\leq i\leq m$.

   \ref{it:cascading:Omega is well defined}
    Using Lemma~\ref{it:powers go in:relationship of generators of ssym} repeatedly, we get
    \begin{align}
    \calT_{n} ([h_{n}|_{\sigma^{n}(x)}])^{q}
    =
    \calT_{m}([h_{m}|_{\sigma^{m}(x)}])^{p}
    \overset{\ref{it:powers go in:relationship of generators of ssym}}{=}
    \left(\calT_{n}[h_{n}|_{\sigma^{n}(x)}]^{k_{m,n} }\right)^{p},
    \end{align}
    or in other words,
    \[
        [\mathrm{const}_{1_{G}}]
        =
        \calT_{n}([h_{n}|_{\sigma^{n}(x)}])^{pk_{m,n}-q}.
    \]
    Let $s\defeq pk_{m,n}-q$. The above means that, for large enough $r\in\N_{\geq n}$, we have
    \[
        1_{G}
        =
        (h_{n}|_{\sigma^{n}(x)})^{ s } (r-n)
        \overset{\ref{it:powers go in} }{=}
        h_{n}^{ s }|_{\sigma^{n}(x)} (r-n)
        =
        h_{n}^{ s }|_{x(n,r)}.
    \]
    Now, $h_{n}^{ s }\in \langle h_{n}\rangle= \Z_{x(n,n)}$, so $h_{n}^{ s }\cdot \sigma^{n}(x)=\sigma^{n}(x)$. From this, it follows 
    that $h_{n}^{ s }\cdot x(n,r)=x(n,r)$. Combined with $1_{G}=h_{n}^{ s }|_{x(n,r)}$ and pseudo-freeness, we deduce that $h_{n}^{ s }=1_{G}$. Since $G=\Z$ is an integral domain, we must either have $h_{n}=1_{G}$ or $pk_{m,n}-q = s =0$, which proves the claim.
\end{proof}

We are ready to prove Theorem~\ref{thm:ssymdual}.
\begin{proof}[Proof of Theorem~\ref{thm:ssymdual}]
    The continuous homomorphism $\Omega$ exists by Lemma~\ref{lem:Omega exists}. Moreover, since every element of $\ssym(x)$ can be written as $\mfks_{n}^{k}$ for $n\in\N$ and $k\in\Z$ by Proposition~\ref{prop:mult in ssym, revised}, it is clear that $\Omega$ is injective. It remains to show that, if $(\Z,E)$ is pseudo-free, then $\Omega$ has a continuous inverse with the formula as claimed in the statement of the theorem.

    First, let us check that $\Omega\inv(\vec{z})$ is well-defined. If $ i \leq j $,
    then 
    \begin{align}\label{eq:z_n and k_i,j}
        z_{ i }
        &\defeq \pi_{ i }(\vec{z})
        =
        \bigl(f_{ i , j }\circ\pi_{ j }\bigr) (\vec{z})
        =
        z_{ j }^{k_{ i , j }}.
    \end{align}
    Now, if $p,q\in\Z$ and $m\leq n \in \N$ are such that
    \(\mfks_{ m }^{p}=\mfks_{ n }^{q}\), then by Lemma~\ref{lem:cascading}\ref{it:cascading:Omega is well defined}, there are two scenarios: 
    Either $h_n=1_{G}$, or $q=p k_{ m , n }$. In the first case, it follows by definition of $X$ and by Lemma~\ref{lem:cascading}\ref{it:cascading:minimal non-zero h_m}that $X_{m}=X_{n}=\{1\}$; thus, $z_{m}^{p}=1=z_{n}^q$. In the second case, we get
   \[z_{ m }^p
    \overset{\eqref{eq:z_n and k_i,j}}{=}
    z_{ n }^{p k_{ m , n }}
    \overset{\ref{lem:cascading}\ref{it:cascading:Omega is well defined}}{=}
    z_{ n }^{q}.
    \]
    This proves that the assignment $\kappa\defeq \Omega\inv(\vec{z})\colon \mfks_{ n }^q \mapsto z_{ n }^q$ is unambiguous. Next, we  show that $\kappa$ is an element of $\ssymdual(x)$: for any $ m , n \in\N$, any $\ell\geq  m , n $, and any $p,q\in\Z$, we have
    \begin{align*}
        \kappa(\mfks_{ m }^p \, \mfks_{ n }^q)
        &=
        z_{\ell}^{p k_{ m ,\ell}+ q k_{ n ,\ell}}
        &&\text{
        by Proposition~\ref{prop:mult in ssym, revised} and
        by def'n of $\kappa$}
        \\
        &=
        z_{ m }^p\, z_{ n }^{q}
        &&\text{by Equation~\eqref{eq:z_n and k_i,j} (twice)}
        \\
        &=
        \kappa(\mfks_{ m }^p)\kappa(\mfks_{ n }^q)
        &&\text{by def'n of $\kappa$}.
    \end{align*}
    We further have that $\kappa(\mfks_{ m }^0)=z_{ m }^0=1$, so that $\kappa(1_{\ssym})=1$, which finishes the proof that $\kappa$ is a homomorphism. Since it is furthermore $\T$-valued and continuous (as $\ssym(x)$ is discrete from the groupoid being \'{e}tale), we have shown that $\kappa=\Omega\inv(\vec{z})$ is an element of $\ssymdual(x)$.

    To verify that $\vec{z}\mapsto \Omega\inv(\vec{z})$ is continuous, let $(\vec{z}_{\lambda})_{\lambda}$ be a net in $X$ that converges to $\vec{z}$. To see that $\kappa_{\lambda}\defeq  \Omega\inv(\vec{z}_{\lambda})$ converges to $\kappa\defeq \Omega\inv(\vec{z})$, we must show 
    that $\kappa_{\lambda}(\mfkt)\to \kappa(\mfkt)$ for all $\mfkt$ in $\ssym(x)$. (Here, we have made use of the fact that $\ssym(x)$ is discrete, as explained earlier.)
    If we write $\mfkt=\mfks_{ n }^p$ for some $ n \in\N$ and some $p\in\Z$, then $\kappa_{\lambda}(\mfkt)=\pi_{ n }(\vec{z}_{ \lambda })^p$, which by continuity of $\pi_{ n }$ converges to $\pi_{ n }(\vec{z})^p=\kappa(\mfkt)$, as claimed.  

    Lastly, note that the maps $\Omega$ and $\Omega\inv$ are indeed mutually inverse: $\bigl(\Omega\circ\Omega\inv\bigr) (\vec{z})$ is the infinite sequence whose $ n $th coordinate is given by
    \begin{align*}
        (\pi_{ n }\circ\Omega)\bigl(\Omega\inv (\vec{z})\bigr)
        =
        \Omega\inv (\vec{z})\bigl(\mfks_{ n }\bigr)
        =
        z_{ n },
    \end{align*}
    meaning it is the infinite sequence $\vec{z}$.
   Likewise, 
    \begin{align*}
        (\Omega\inv \circ \Omega)(\kappa)
        =
        \bigl(
        \mfks_{ n }
        \mapsto
        \pi_{ n } (\Omega(\kappa))
        =
        \kappa(\mfks_{ n })
        \bigr)
        =
        \kappa.
    \end{align*}
    This concludes the proof of Theorem~\ref{thm:ssymdual}.
\end{proof}

Now that we have identified $\ssymdual(x)$ for any arbitrary $x\in E^\infty$, we can give a more explicit description of the action of the quotient groupoid $\cGss[\Z]/\ssym$ on the dual bundle $\widehat{S}$ that we alluded to in the beginning of this section.

\subsection{Towards a new groupoid model: The action}

We have shown in Proposition~\ref{prop:Ssym:open,normal,abelian,subgpd} that, given any self-similar graph $(G,E)$, the abelian groupoid $\ssym$ is open and normal in $\cGss$ and has full unit space. This allows us to invoke the following theorem by Renault, where we write $\dot{\mfkg}$ for the class of an element $\mfkg\in\calG$ in the quotient $\calG/\calS$:
\begin{thm}[{\cite[Cor.\ A.8]{Renault:2023:Ext} and special case of \cite[Theorem 5.5(ii)]{Renault:2023:Ext}}]\label{thm:Renaults5.5}
   Let
   $\calG$ be a second countable, locally compact, Hausdorff,
   \'etale groupoid. 
   Suppose that
   $\calS$ is
   an open,
   abelian, normal subgroupoid with $\calS\z =\calG\z $.
   Then the quotient groupoid $\calG/\calS$ is locally compact (not necessarily Hausdorff), and   
   $\csr(\calG)$ is isomorphic to a \emph{twisted} reduced groupoid $C^*$-algebra of the action groupoid $\calG/\calS \ltimes \widehat{\calS}$,
   where  $\calG/\calS$ acts on the topological space $\widehat{\calS}$ by
    \begin{align} \label{eq:left action, general formula}
    (\dot{\mfkg} \HleftX \kappa)(\mfks) = \kappa (\mfkg\inv \mfks \mfkg)
    \end{align}
    for $\mfkg\in \calG_y^x$, $\kappa\in \widehat{\calS}(y)$, $\mfks\in \calS(x)$. 
\end{thm}

\begin{rem}
    \begin{enumerate}
    \item The groupoid $\calG/\calS \ltimes \widehat{\calS}$ can be seen as an improvement over the groupoid $\calG$ since one has, effectively, moved an abelian part of the dynamics (i.e., $\calS$) into the purely topological, spatial part of the groupoid (represented by the new unit space, $\widehat{\calS}$). However, this comes at the price of (potentially) needing to deal with a twisted $C^*$-algebra rather than the untwisted $\csr(\calG)$, and with a potentially non-Hausdorff groupoid.
   \item The quotient $\calG/\calS$ is Hausdorff if and only if $\calS$ is closed in $\calG$. Since many people prefer to work with Hausdorff groupoids, it is therefore reasonable to wonder when $\ssym$ is closed in $\cGss$. Provided that the group $G$ of the self-similar graph $(G,E)$ has an element of infinite order, we identify in Theorem~\ref{thm:closedness-equivalent} of Section~\ref{ssec:closedness} which conditions on $E$ are necessary and sufficient for $\ssym$ to be closed. 
   \item The properties that $\calS\z=\calG\z$, that $\calS$ is open, and that $\calG$ is locally compact \'etale can be used to show that the quotient map $q\colon  \calG\to\calG/\calS$ is a surjective continuous open groupoid homomorphism which, in turn, implies that the action of $\calG/\calS$ on $\widehat{\calS}$ is continuous. 

    \end{enumerate}
\end{rem}

We will now focus on the situation that $G=\Z$ and that we are given a self-similar graph $(\Z,E)$ that is pseudo-free, so that Theorem~\ref{thm:ssymdual} identifies the dual of the subgroup $\ssym(x)$ of $\cGss[\Z]$. We will now describe the action $\cGss[\Z]\curvearrowright \ssymdual$ as defined at~\eqref{eq:left action, general formula} more explicitly, to see the advantage of the new groupoid model in Theorem~\ref{thm:Renaults5.5}; the main take-away is that the new groupoid $\cGss[\Z]/\ssym \ltimes \ssymdual$ is much easier to understand than the original groupoid $\cGss[\Z]$.

The first thing we point out is that the action of $\cGss[\Z]/\ssym$ on $\ssymdual$ can be easily computed---no matter the choice of $E$, the action of $\Z$ on $E$, or the restriction `of $E$ on $\Z$.' In fact, $\cGss[\Z]/\ssym \curvearrowright \ssymdual$ really only depends on the copy of $\Z$ in
 $Q(\N,\Z)\rtimes \Z$ in a way that we will now make precise.

   Fix $x\in E^\infty$, and let $\mfks_{x,n}\in \ssym(x)$ be as defined at~\eqref{def:xi_n}, where we added the index $x$ to 
   distinguish the fibres. Fix another infinite path $y\in E^\infty$ and  an arbitrary element
    \[
     \mfkg
        =
        \left(
            x;
            \calT_{m}([g|_{\sigma^{\ell}(y)}]), m-\ell;
            y
        \right)
    \]
    of $(\cGss[\Z])_y^x$; in particular, $m,\ell\in\N $ and $g\in \Z$ are such that $\sigma^{m}(x)=g\cdot \sigma^{\ell}(y)$. 
    By definition of the semi-direct product, we have 
    in $Q(\N,\Z)\rtimes \Z$,
    \[
        \left(\calT_{m}([g|_{\sigma^{\ell}(y)}]), m-\ell\right)\inv =
        \left(
            \calT_{\ell}([(g|_{\sigma^{\ell}(y)})\inv]), \ell-m
        \right),
    \]
    and so since $G=\Z=\langle a\rangle$ is abelian, 
    a standard computation shows
    \begin{align}
        \label{eq:conjugation}
    \begin{split}
        &(\calT_{m}([g|_{\sigma^{\ell}(y)}]), m-\ell)\inv
        \,
        (\calT_{n}([h_{n}|_{\sigma^{n}(x)}]),0)
        \,
        (\calT_{m}([g|_{\sigma^{\ell}(y)}]), m-\ell)
        \\
        ={}&
        (
        \calT_{n+\ell-m}([h_{n}|_{\sigma^{n}(x)}]),
        0
        ).
    \end{split}
    \end{align}
     
    Let $\kappa\in\ssymdual(y)$. Since $\mfkg\HleftX \kappa \colon \mfks \mapsto \kappa(\mfkg\inv \mfks \mfkg)$, this computation shows that 
\begin{equation}\label{eq:action,general}
        \left(
            x;
            [f]
            , 
            k
            ;
            y
        \right)
        \HleftX
        \kappa
        \colon
        \quad
        \left( x;\calT_{n}([h|_{\sigma^{n}(x)}]),0;x\right)        
        \mapsto
        \kappa\left( y;\calT_{n
        -k
        }([h|_{\sigma^{n}(x)}]),  0;y\right).
    \end{equation}
    (Note that we have checked in Lemma~\ref{lem:Ssym is normal} that $\left( y;\calT_{n-k}([h|_{\sigma^{n}(x)}]),  0;y\right)$ is indeed an element of $\ssym(y)$, so that evaluating $\kappa$ at it makes sense.)
    This shows that the $Q(\N,\Z)$-component of $\mfkg$ does not feature at all.

In Appendix~\ref{app:Weyl examples}, we compute the action groupoid $\cGss[\Z]/\ssym\ltimes \ssymdual$ for specific examples of graphs $E$--namely a loop and a loop with a stem. 

\section{Cartan subalgebras for $\mathcal{O}_{\Z, E}$} 
\label{sec: Cartan subalgebras} 

In this section, we will show that for a large class of self-similar graphs, the subgroupoid $\ssym$ induces a Cartan subalgebra $C^*_{r}(\ssym)\subseteq C^*_{r}(\cGss[\Z])$. More precisely, we will prove Theorem~\ref{Thm_Ssym_Cartan} in Subsection \ref{ssec:Ssym satisfies ricc}. The terms $g$-generic and $g$-rare are established in Definition~\ref{def:GenaricRare}. We also remind the reader of our convention of writing $a$ for the generator of $\Z$; see page~\pageref{Notation and conventions}.
 
\begin{thm} 
\label{Thm_Ssym_Cartan} 
Let $(\mathbb{Z}, E)$ be a pseudo-free self-similar graph satisfying our standing assumptions
on page~\pageref{Notation and conventions}.
Assume additionally that the following conditions hold: 
\begin{enumerate}[label=\textup{(\arabic*)}] 
\item\label{it:Thm_Ssym_Cartan:generic or rare} For any $x\in E^{\infty}$ and 
$g\in \Z$,
$x$ is either $g$-generic or $g$-rare. 
\item\label{it:Thm_Ssym_Cartan:cycles}
For any cycle $C$ in $E$ with no entrance, if $p$ is the minimal positive integer such that $a^p$ fixes a vertex on $C$, we have $a^p|_{C} \neq a^{\pm p}$.
\end{enumerate} 
Then $C^*_{r}(\ssym)$ is a Cartan subalgebra of $C^*_{r}(\cGss[\Z])$. 
\end{thm} 

The strategy to prove Theorem~\ref{Thm_Ssym_Cartan} is to verify sufficient conditions provided in results of \cite{DWZ:Twist}, namely:

\begin{thm}[cf.\ {\cite[Corollary 4.5]{DWZ:Twist}}]\label{thm:DWZ}
    Let $\calG$ be a locally compact Hausdorff \'etale groupoid, and $\mathcal{S}$ be maximal among open abelian subgroupoids of $\Iso(\calG)^{\circ}$. If $\mathcal{S}$ is closed and normal in $\calG$, and if
    \begin{equation} 
    \label{eq:ricc,stronger} 
    \{\mfkg\in \Iso(\calG)^{\circ}: 1<|\{\mfks^{-1} \mfkg \mfks  : \mfks \in \mathcal{S}\}|< \infty \} = \emptyset, 
    \end{equation} 
    then $C^*_{r}(\mathcal{S})$ is a Cartan subalgebra in $C^*_{r}(\calG)$.
\end{thm} 

It is known
\cite{Self-similar:EP17}
that $\cGss$ is locally compact Hausdorff \'etale, and we have already 
shown in Lemma~\ref{lem:Ssym is subgroupoid} that $\ssym$ is an open abelian subgroupoid of $\Iso(\cGss)^{\circ}$. The next subsections are devoted to verifying that, given suitable self-similar graphs $(G,E)$, the other assumptions in Theorem~\ref{thm:DWZ} are likewise satisfied 
(To see why a condition such as~\eqref{eq:ricc,stronger} is needed, see Remark~\ref{rmk:ricc,stronger}).

In Subsection \ref{ssec:closedness}, we characterize the closedness of the groupoid $\ssym$ based on the dynamics of $(G,E)$ in Theorem~\ref{thm:closedness-equivalent}, which leads us to Assumption~\ref{it:Thm_Ssym_Cartan:generic or rare} in Theorem~\ref{Thm_Ssym_Cartan}. Assumption~\ref{it:Thm_Ssym_Cartan:cycles} is needed to show that $\ssym$ is maximal among open abelian subgroupoids of $\Iso(\cGss)^{\circ}$; see Theorem~\ref{thm:Ssym is maximal}. 

Combining Theorem~\ref{thm:closedness-equivalent} with Lemma~\ref{C:can-act} yields the following theorem for Cartan subalgebras in the genuine action case. 

\begin{thm} 
\label{Thm_genuine_action_Cartan} 
Let $(G, E)$ be a genuine action of an abelian group on a graph $E$ satisfying our standing assumptions on page~\pageref{Notation and conventions} and where $G$ has an element of infinite order. The following are equivalent:
\begin{enumerate}[label=\textup{(\arabic*)}] 
\item \label{equiv-closed} For any $x\in E^{\infty}$ and $g\in G$, $x$ is either $g$-generic or $g$-rare.
\item $C^*_{r}(\Iso(\cGss)^\circ)$ is a Cartan subalgebra in $C^*_{r}(\cGss)$. 
\end{enumerate} 
\end{thm} 
Moreover, we include an example where condition \ref{equiv-closed} fails so that, in particular, stronger versions of Theorem~\ref{thm:DWZ} imply that $\ssym$ will not yield a Cartan subalgebra of $C^*_r(\cGss[\Z])$ (see Example~\ref{exa:closed_g_generic}).

\subsection{Cycline triples and maximal abelianness of $\ssym$}
\label{ssec:cycline triples and masa} 
In this Subsection, our target is to show the following theorem. 

\begin{thm} 
\label{thm:Ssym is maximal} 
Let $(\Z, E)$ be a pseudo-free self-similar graph satisfying our standing assumptions on page~\pageref{Notation and conventions}.
Assuming Condition~\ref{it:Thm_Ssym_Cartan:cycles} of Theorem~\ref{Thm_Ssym_Cartan}, the subgroupoid $\ssym$ is maximal among open abelian subgroupoids of $\Iso(\cGss[\Z])^\circ$. 
\end{thm} 

We first obtain a more concrete description of $\Iso(\cGss[\Z])^\circ$ by exploring cycline triples at vertices with different dynamical properties. Notice first that both locally faithful and non-locally faithful vertices (Definition~\ref{def:(n)lf}) could have symmetric cycline triples $(\mu, g, \mu)$.  
In fact, the following lemma is straightforward to check. 
\begin{lem}\label{lem:symmetric implies g stabilizes}
Let $(G, E)$ be a self-similar graph. 
If $(\mu, g, \mu)$ is a symmetric cycline triple, then $g\in G_{s(\mu)}$, and if additionally $s(\mu)$ is locally faithful, then $g=1_{G}$.
\end{lem} 

In comparison, a non-symmetric cycline triple only appears at vertices with a unique infinite path, and such vertices are automatically non-locally faithful provided $G$ is not a torsion group.
To be precise:

\begin{lem} \label{lem:one path} 
Let $(G, E)$ be a self-similar graph satisfying the standing assumptions of page~\pageref{Notation and conventions}.
Let $v \in E^0$. 
\begin{enumerate}[label=\textup{(\arabic*)}]
    \item\label{it:one path:unique infinite path} If $\mu (g\cdot x) = \nu x$ for some $g\in G$, infinite path $x$ with $r(x)=s(\nu)$ and $\mu, \nu \in E^{*}$ with $\mu \neq \nu$, then $x$ is the unique infinite path with $r(x)=s(\nu)$ satisfying $\mu (g\cdot x) = \nu x$. 
    \item\label{it:one path:non-symmetric cycline triple} If there exists a non-symmetric cycline triple at $v$, then $|vE^{\infty}| = 1$. 
    \item\label{it:one path:inf order}      
    Suppose $|vE^{\infty}| = 1$. If $g \cdot v = v$ for some non-trivial group element $g\in G$, then $v$ is not locally faithful; in particular, if $G$ has an element of infinite order, then $v$ is not locally faithful.
\end{enumerate}
\end{lem}

\begin{proof} 
\ref{it:one path:unique infinite path} 
We consider only the case $|\nu|-|\mu|\defeq k>0$, so that $\nu = \mu \tilde{\nu}$ for some $\tilde{\nu}\in E^{k}$; the other case follows similarly by considering the path $y=g\cdot x$ instead. Since $\mu (g\cdot x) = \nu x = \mu \tilde{\nu} x$, it follows that $g\cdot x = \tilde{\nu} x$.
Notice that the first length-$k$ segment 
$x(0,k)$
of the infinite path
is determined by $g$ and $\tilde{\nu}$. Indeed, since $g\cdot x = \tilde{\nu} x$ implies $g\cdot x(0,k)= \tilde{\nu}$, it follows that $x(0,k) = g^{-1}\cdot \tilde{\nu}$. Similarly, for any $n\geq 1$, we have 
\begin{align*} 
    \sigma^{nk}(x)
    =
    \sigma^{(n+1)k}(\tilde{\nu} x)
    &=
    \sigma^{(n+1)k}(g\cdot x)
    &&\text{by choice of $k$ and $\tilde{\nu}$}
    \\
    &=
    g|_{x(0, (n+1)k)}\cdot \sigma^{(n+1)k}(x)
    &&\text{by \cite[Lemma 3.7(ii)]{Self-similar:LY21}},
\end{align*} 
which implies 
\[
\big(g|_{x(0, (n+1)k)}\big)^{-1}
\cdot x(nk,(n+1)k) = x((n+1)k,(n+2)k). 
\] 
Thus the $(n+1)^\text{st}$ segment of length $k$ is determined by $g$ and the $n$-th segment of length $k$ for any $n\geq 0$. Inductively, $x$ is uniquely determined by $g$ and $\tilde{\nu}$. Or in other words, $x$ is the only infinite path with $r(x)=s(\nu)$ satisfying $\mu (g\cdot x) = \nu x$. 

\ref{it:one path:non-symmetric cycline triple} 
Suppose that there exists a non-symmetric cycline triple $(\mu, g, \nu)$ at $v$, which in particular means that $\mu\neq\nu$. By~\ref{it:one path:unique infinite path}, there exists a unique infinite path $x$ with $r(x)=v = s(\nu)$ such that $\mu (g\cdot x) = \nu x$. Since $(\mu, g, \nu)$ is a cycline triple, it follows that $vE^{\infty} = \{x\}$. 

\ref{it:one path:inf order} 
Let $x$ be the unique infinite path with $r(x)=v$. If $v$ is fixed by a non-trivial group element $g\in G$, then $g$ fixes $x$.
As a result, the vertex $v$ is not locally faithful.
The last sentence in~\ref{it:one path:inf order} follows from the pigeonhole principle. Indeed, suppose there exists an element $a\in G$ of infinite order. Since $\{a^m \cdot v : m\in \N \} \subset E^{0} $ is finite by assumption on $E$, there exist distinct natural numbers $m, m'$ such that $a^m\cdot v = a^{m'}\cdot v$. Thus the non-trivial group element $a^{m-m'}$ fixes $v$. 
\end{proof} 

    We record the following corollary of Lemma~\ref{lem:one path}~\ref{it:one path:inf order}, which will be helpful for us later.

\begin{cor}\label{cor:actiononcycles:nlf}
    Let $(G, E)$ be a self-similar graph satisfying the standing assumptions of page~\pageref{Notation and conventions},
    and such that $G$ has an element of infinite order.

    If $h\in G$ is arbitrary and if $C$ is a cycle with no entrance, then
        no vertex on $h\cdot C$ is
        locally faithful.
\end{cor}

\medskip

Given a non-symmetric cycline triple $(\alpha,g,\beta)$, Lemma~\ref{lem:one path}~\ref{it:one path:non-symmetric cycline triple} states that there is a unique infinite path $x$ with $x(0,0)=s(\beta)$. Since we assume that our graph $E$ has only finitely many vertices, $x$ must thus be contained in a path of the form depicted in Figure~\ref{diag:nlf on stem}.
In particular, there are two scenarios for where $x(0,0)$ appears: Either, $x(0,0)$ lies 
on the cycle (meaning it equals a vertex such as $v$ 
or $v'$ in the figure), so that $x$ is (a shift of) the infinite cycle. 
Or $x(0,0)$ lies `on the stem', meaning it equals a vertex such as $w$ in Figure~\ref{diag:nlf on stem}.

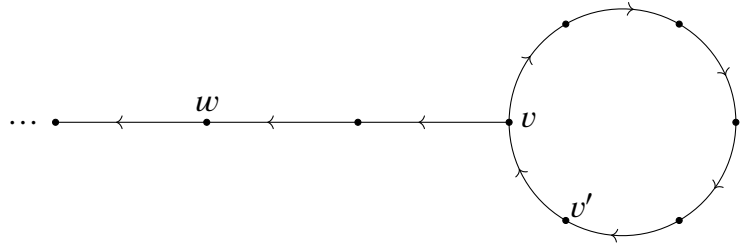
\begin{figure}[h!] 
\centering
\begin{tikzpicture}[xscale=-1]

\def \radius {1.5} 
\def \factor {2} 
\def \starting {.75} 
\def \n {6} 
\def \m {3} 
\def \margin {1} 

\foreach \s in {1,...,\n}
{
  \node[draw, circle,fill=black, scale=0.2]
    at ({360/\n * (\s - 1)}:\radius cm) {};
  \draw[->-]
    ({360/\n * (\s - 1)+\margin}:\radius cm)
    arc ({360/\n * (\s - 1)+\margin}:{360/\n * (\s)-\margin}:\radius cm);
}

\node[right] at ({\margin}:\radius cm) {$v$};
\node[above] at ({290+\margin}:\radius cm){
$v'$
} ;

\foreach \t in {\starting,...,\m}
{
  \node[draw, circle,fill=black, scale=0.2]
    at ({0}:\factor*\t cm + \factor cm) {};
  \draw[->-]
    ({0}:\factor*\t cm) -- ({0}:\factor*\t cm +\factor cm);
}

\node[above] at ({0}:\factor*\starting cm + \factor*\m cm - \factor cm ) {$w$};

\node[left] at ({0}:\factor*\m cm + \factor cm - 5mm) {$\dots$};

\end{tikzpicture}

\caption{An infinite cycle $C^\infty$ at a vertex $v$ with a stem from $v$ to $w$ and beyond}\label{diag:nlf on stem}
\end{figure}

Since $(\alpha,g,\beta)$ was assumed to be a non-symmetric cycline triple at $r(x)$, our next lemma shows that the latter situation cannot happen, i.e., that we must have $r(x)\neq w$. In other words, non-symmetric cycline triples can only appear at vertices on cycles with no entrance.

Note that if $G$ has an element of infinite order, then Lemma~\ref{lem:one path}~\ref{it:one path:inf order} asserts that 
none of the vertices in Figure~\ref{diag:nlf on stem} are
locally faithful.

\begin{lem} 
\label{it:one path:no cycle} 
Let $(G, E)$ be a self-similar graph under the standing assumptions of page~\pageref{Notation and conventions}. If $|vE^{\infty}| = 1$ and $v$ is not on a cycle, then every cycline triple at $v$ is symmetric.  
\end{lem}

\begin{proof} 
Let $(\mu , g, \nu)$ be a cycline triple at $v$ and $vE^\infty=\{x\}$. Then 
 \begin{equation}\label{eq:mu (g cdot x) = nu x}
 \mu (g\cdot x) = \nu x.
 \end{equation}
Take $n\in \mathbb{N}$ to be the minimal integer such that the vertex $w = x(n, n)$ appears in $x$ more than once, and take $m$
to be the minimal positive integer such that $x(n + m, n+m) = w$. Then $C\coloneqq x(n, n+m)$ is a cycle starting and ending at $w$ and $x(0, n) C^{\infty}$ is an infinite path with range $v$. Uniqueness of infinite paths with $x(0,0)=v$ implies that $x = x(0,n) C^{\infty}$. There are three cases to consider, depending on how the lengths of $\mu$ and $\nu$ compare:

\paragraph{Case 1: $|\mu|=|\nu|$.}
Equation~\eqref{eq:mu (g cdot x) = nu x} implies that
$\mu=\nu$, 
as desired.

\paragraph{Case 2: $|\mu|>|\nu|$.} 
Equation~\eqref{eq:mu (g cdot x) = nu x} implies that $\mu = \nu \widetilde\mu$ for some finite path $\widetilde\mu$ and 
\begin{equation}\label{eq:widetilde mu (g cdot x) = x}
\widetilde\mu \bigl(g\cdot (x(0,n) C^\infty)\bigr) = x(0,n) C^\infty. 
\end{equation} 
Again, we make a case distinction according to how the lengths of $\widetilde\mu$ and $x(0,n)$ compare: 
\begin{itemize}
\item 
If $|\tilde{\mu}| \geq n$, it follows from~\eqref{eq:widetilde mu (g cdot x) = x} that $w' = s(\tilde{\mu})$ is on $C$. Take $C'$ to be the unique cycle starting and ending at $w'$. Removing paths of length $|\tilde{\mu}|$ on both sides of~\eqref{eq:widetilde mu (g cdot x) = x}, we have 
\begin{equation} 
\label{eq:lag by length n}
g\cdot (x(0,n) C^\infty) = C'^{\infty}. 
\end{equation} 
This implies that the range vertex $v$ of $x(0,n) C^\infty$, which is not on a cycle by our assumption, is sent by the group element $g$ to the range vertex $w'$ of $C'^{\infty}$, which is on a cycle. This is impossible by Lemma~\ref{lem:cycles to cycles}~\ref{it:cycles to cycles:length}.
\item 
If $|\tilde{\mu}| < n$, we remove the first length $n$ segment from both sides of~\eqref{eq:widetilde mu (g cdot x) = x} and 
\begin{equation} 
g|_{x(0, n-|\tilde{\mu}|)} \cdot (x(n-|\tilde{\mu}|, n) C^{\infty}) = C^{\infty}. 
\end{equation} 
Note that $x(n-|\tilde{\mu}|, n-|\tilde{\mu}|)$ is not on the cycle as $|\tilde{\mu}| > 0$ and $n$ is the smallest positive integer such that $x(n,n)$ is on the cycle. Another application of 
Lemma~\ref{lem:cycles to cycles}~\ref{it:cycles to cycles:length}
shows that this case is impossible. 
\end{itemize} 
Therefore, it is impossible to have $|\mu|>|\nu|$. 
\paragraph{Case 3: $|\mu|<|\nu|$.} 
Equation~\eqref{eq:mu (g cdot x) = nu x} implies for $y:=g\cdot x$ that  $ \nu (g^{-1}\cdot y) = \mu y$. Since $|(g\cdot v)E^\infty|=|vE^\infty|=1$, this shows that $(\nu, g^{-1}, \mu)$ is a cycline triple. Again by
Lemma~\ref{lem:cycles to cycles}~\ref{it:cycles to cycles:length},
$g\cdot v$ does not lie on a cycle as $v$ is not on a cycle. We can now invoke Case~2 from above to reach another contradiction.
\end{proof}

We summarize the observations from Lemmas~\ref{lem:symmetric implies g stabilizes},~\ref{lem:one path}, and~\ref{it:one path:no cycle} as follows:

\begin{cor}\label{conclusion}  
Let $(G, E)$ be a self-similar graph, and consider the following partition of the set $E^{0}$ into four collections of vertices: 
\begin{enumerate} [label=\textup{(\arabic*)}] 
\label{item:Type1} 
\item \label{vertex_1} $v$ is locally faithful; 
\item \label{vertex_2} $v$ is not locally faithful and $|vE^{\infty}|\geq 2$; 
\item \label{vertex_3} $v$ is not locally faithful,
$|vE^{\infty}|=1$, and $v$ is not on a cycle; 
\item \label{vertex_4} $v$ is not locally faithful and $vE^{\infty}=\{C^\infty\}$ for some cycle $C$. 
\end{enumerate} 
If $G$ has an element of infinite order and $E$ has finitely many vertices and no sources, then the following hold:
\begin{itemize}
    \item If $v$ is a vertex of type~\ref{vertex_1}, then all cycline triples at $v$ are  trivial. 
    \item If $v$ is a vertex of type~\ref{vertex_2} or~\ref{vertex_3}, then all cycline triples at $v$ are symmetric.
    \item Only vertices of type~\ref{vertex_4} allow non-symmetric cycline triples.
\end{itemize}
\end{cor}

Although the type division is not essential for the main proofs, it provides an intuitive picture of vertex configurations with different dynamical properties.

\begin{figure}[h]
    \centering

 \scalebox{.9}
{    
    \begin{tikzpicture}[
      dot/.style={circle, fill=black, draw=none, inner sep=1.3pt},
      dashedsep/.style={dashed, line width=1.0pt},
      M/.style={
        matrix of nodes,
        nodes={anchor=center, inner sep=2pt},
        row sep=8mm,
        column sep=10mm
      }, 
      ampersand replacement=\& 
    ]
    
    
    \newcommand{\LoopW}{2.3cm} 
    \newcommand{\LoopH}{1.6cm} 

    
    \matrix (m) [M] {
      \text{Type~\ref{vertex_1}} \& {} \& {} \& {} \& {} \& {} \& {} \& {} \& \text{Type~\ref{vertex_4}} \\
      {} \& {} \& {} \& {} \& {} \& {} \& {} \& {} \& |[name=r1b, dot]| {} \\
      {} \& {} \& {} \& |[name=v, dot]| {} \& {} \& {} \& {} \& {} \& {} \\
      |[name=w, dot]| {} \& {} \& |[name=a, dot]| {} \& {} \& {} \& |[name=b, dot]| {} \& {} \& {} \& |[name=c, dot]| {} \\
      {} \& {} \& {} \& {} \& {} \& {} \& |[name=d, dot]| {} \& {} \& |[name=e, dot]| {} \\
      {} \& {} \& {} \& {} \& {} \& {} \& {} \& {} \& {} \\
    };
    
    
    \node at ($(m-1-2.center)!0.5!(m-1-5.center)$) {\text{Type }\ref{vertex_2}};
    \node at ($(m-1-5.center)!0.5!(m-1-8.center)$) {\text{Type }\ref{vertex_3}};
    
    
    \draw[dashedsep] (m-1-2.center) -- (m-6-2.center);
    \draw[dashedsep] (m-1-5.center) -- (m-6-5.center);
    \draw[dashedsep] (m-1-8.center) -- (m-6-8.center);
    
    
    \node[right=2pt] at (v.north) {
    };
    \node[below=2pt] at (w.south) {
    };
    
    
    \LoopRight{r1b}{}
    
    \LoopAbove{v}{}
    \LoopLeft{v}{}
    
    \LoopLeft{w}{}
    \LoopAbove{w}{}
    \LoopBelow{w}{}
    
    
    \draw[->-] (r1b.center) -- (a.center);
    \draw[->-] (v.center) -- (a.center);
    \draw[->-] (a.center) -- (w.center);
    \draw[->-] (b.center) -- (a.center);
    \draw[->-] (c.center) -- (b.center);
    \draw[->-] (d.center) -- (a.center);
    \draw[->-] (e.center) -- (d.center);
    
    \node[dot, name=cr] at ($(c)+(1.8cm,0)$) {};
    \node[dot, name=er] at ($(e)+(1.8cm,0)$) {};

    \draw[->-] (c)  to[bend left=35] (cr);
    \draw[->-] (cr) to[bend left=35] (er);
    \draw[->-] (er) to[bend left=35] (e);
    \draw[->-] (e)  to[bend left=35] (c);
    \end{tikzpicture}
}
    \caption{Configurations of vertices of different types. The column headers refer to the types introduced in Corollary~\ref{conclusion}.} 
    \label{fig:all types of vertices}
\end{figure}
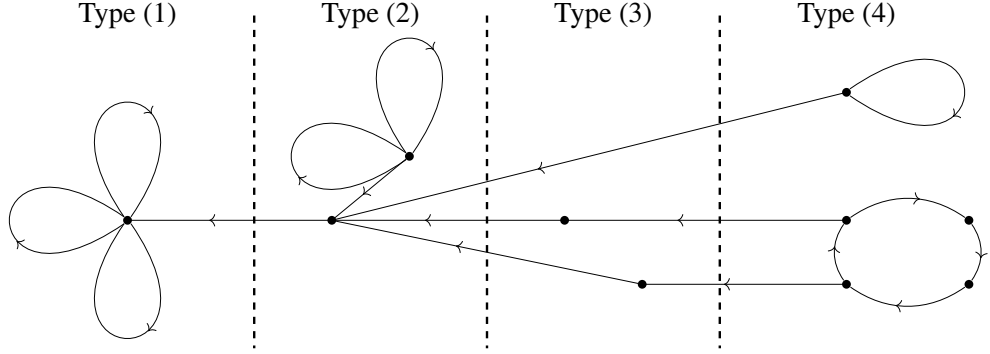

\medskip

Under the assumption of pseudo-freeness,
possible paths between vertices follow from invoking
Lemma~\ref{lem:range is nlf} and from the comparison of numbers of infinite paths ending at each vertex. For instance, there are no finite paths with the source vertex of type~\ref{vertex_1} and range vertex of type~\ref{vertex_2},~\ref{vertex_3} or~\ref{vertex_4} as a consequence of Lemma~\ref{lem:range is nlf}. Moreover, there are no paths from a type~\ref{vertex_2} vertex to a type~\ref{vertex_3} vertex, as infinite paths at the source vertex are inherited by the range vertex. 

A more concrete description of $\Iso(\cGss)^\circ$ follows from previous discussions, which is essential in the proof of both maximal abelianness (Theorem~\ref{thm:Ssym is maximal}) and closedness of $\ssym$ (Theorem~\ref{thm:closedness-equivalent}). The pseudo-freeness assumption ensures that cylinder sets form a basis for the topology, which in turn allows us to invoke \eqref{eq:int of Iso} for $\Iso(\cGss)^{\circ}$.

\begin{cor} 
\label{cor:the green corollary}
Let $(G, E)$ be a pseudo-free self-similar graph under the standing assumptions of page~\pageref{Notation and conventions}, and that $G$ has an element of infinite order.
Given a non-symmetric cycline triple $(\alpha, g, \beta)$, the associated compact open bisection $\calZ(\alpha, g, \beta)$ is a singleton and the vertex $s(\beta)$ is of type~\ref{vertex_4} as defined in Corollary~\ref{conclusion}.
Moreover, 
\begin{equation} 
\label{Iso:concrete_description} 
\Iso(\cGss)^\circ = \ssym \bigsqcup \left(
\bigcup_{\substack{(\alpha, g, \beta) \text{ non-symmetric} \\
\text{cycline triple}}} 
 \calZ(\alpha, g, \beta)\right).
\end{equation} 
\end{cor} 

\begin{proof} 
If $(\alpha, g, \beta)$ is a non-symmetric cycline triple, then by Corollary~\ref{conclusion}, $s(\beta)$ is of type~\ref{vertex_4}; in particular, $s(\beta)$ is not locally faithful and $s(\beta) E^{\infty} = 
\{C^\infty\}
$ for some cycle $C$. The latter implies directly that
$\calZ(\alpha, g, \beta)$ is a singleton consisting only of the element $(\alpha (g\cdot C^{\infty}); |\alpha|-|\beta|, \calT_{|\alpha|}([g|_{C^{\infty}}]); \beta C^{\infty})$.

By the description of $\Iso(\cGss)^\circ$ at 
Equation~\eqref{eq:int of Iso},
$\Iso(\cGss)^\circ$ is the union of all compact open bisections corresponding to cycline triples. Since the 
bisections corresponding to symmetric triples are contained in $\ssym$ by the definition of $\ssym$, this combined with the first paragraph proves the claim. 
\end{proof}

\begin{rem}
    It follows from Corollary~\ref{cor:the green corollary} that any subgroupoid of $\Iso(\cGss)^\circ$ that contains $\ssym$ is automatically open, since every point in $\Iso(\cGss)^\circ\setminus\ssym$ is open. 
\end{rem}

As a consequence of Corollary~\ref{cor:the green corollary}, to identify the circumstances under which $\ssym$ is maximal abelian in $\Iso(\cGss)^\circ$, it suffices to understand the commutation relation between elements in $\ssym$ and elements in $\calZ(\alpha,g,\beta)$ for non-symmetric cycline triples $(\alpha,g,\beta)$, which are automatically located at vertices on cycles with no entrance. 
Since only composable groupoid elements are relevant, we fix a cycle without any entrance and study locally the commutation relations among elements whose source and range infinite path ends up being an infinite repetition of this cycle.

We begin by establishing notation for the rest of the subsection. 

\begin{notation}
\label{Notation_cycle} 
In the remaining portion of Subsection \ref{ssec:cycline triples and masa}, we focus on the case $G = \mathbb{Z}$. Assume $C$ is a cycle with no entrance. Let $n$ be the number of edges on $C$, which we denote by $e_{0}, \dots, e_{n-1}$. We let $v_{i}\defeq s(e_{i})$, and we assume that the edges are ordered in such a way that $r(e_i) = v_{i+1\text{ mod }{n}}$ and that we can write the cycle as $C = e_{n-1} e_{n-2} \dots e_1 e_0$. Let $p$ be the smallest positive integer such that $a^{p}$ fixes some vertex on $C$. Since $C$ is a cycle with no entrance and the action respects sources and ranges, $a^{p}$ also fixes all finite paths and vertices on $C$ by induction. 
\end{notation} 
The following quick lemma divides potential cycline triples at a given vertex on $C$ into three types.

\begin{lem} 
\label{types_of_cycline_triples} 
Let $(\mathbb{Z}, E)$ be a self-similar graph, and let $
C$ and $p$ be as described in Notation~\ref{Notation_cycle}.
Every cycline triple at 
a given vertex $v$ on $C$
falls into exactly one of the types below: 
\begin{enumerate}[label=\textup{(\Roman*)}]
    \item\label{TypeI} symmetric cycline triples $(\nu, (a^p)^q, \nu)$ where $s(\nu) = 
    v
    $ and $q \in \Z $; 
    \item\label{TypeII} 
    non-symmmetric cycline triples  of the form
    $(\nu, (a^p)^q, \nu C^{k})$ 
 or
    $(\nu C^{k}, (a^p)^q, \nu)$ where $s(\nu) = 
    v   
    , k \in \N^+$, and $q \in \Z $; 
    \item\label{TypeIII} 
    non-symmmetric cycline triples of the form $(\mu, a^q, \nu)$ where $q \in \Z  \setminus p\Z $, $s(\nu) = v
    \neq s(\mu)$, $r(\mu) = r(\nu)$, and $a^q \cdot 
    v
    = s(\mu)$. 
 In this case, $|\mu|-|\nu|\notin |C|\Z$.
\end{enumerate} 
Moreover, for any cycline triple $(\mu, g, \nu)$ at $v$, the cylinder set $\calZ(\mu, g, \nu)$ is a singleton. 
\end{lem} 

\begin{proof} 
Let $(\mu, g, \nu)$ be a cycline triple at $v$, 
which means $\mu (g\cdot C^{\infty}) = \nu C^{\infty}$. If $g$ fixes $C^{\infty}$, then $g$ fixes $C$ in particular. By minimality of $p$, we then must have $g = (a^{p})^{q}$ for some $q\in \mathbb{Z}$. In this case, $\mu C^{\infty} = \nu C^{\infty}$, which implies $\mu = \nu$ (producing Type~\ref{TypeI}  cycline triples), or $\mu = \nu C^{k}$, or $\nu = \mu C^{k}$ for some $k\geq 1$ (producing Type~\ref{TypeII} cycline triples). If $g$ does not fix $C^{\infty}$, again by minimality of $p$, we have $g = a^{q}$ for some $q \in \Z  \setminus p\Z$ (thus giving Type~\ref{TypeIII} cycline triples). For such a cycline triple $(\mu, a^{q}, \nu)$, since $\mu (a^{q}\cdot C^{\infty}) = \nu C^{\infty}$, it follows from \cite[Lemma 3.5 (ii)]{Self-similar:LY21} that $s(\mu) = a^{q} \cdot s(\nu) = a^{q}|_{C^{\ell}} \cdot s(\nu)$ for any $\ell\in \mathbb{N}$ as $s(\nu)$ is on the cycle $C$. For large enough $\ell$, we have that $a^{q}|_{C^{\ell}} \cdot s(\nu)$ is a vertex on $C^{\infty}$, which implies that $s(\mu)$ is on the cycle $C$. By definition of Type~\ref{TypeIII}, $s(\mu)$ is a different vertex than $s(\nu)$. We thus have two finite paths $\mu, \nu$ ending at the same vertex but starting at different vertices on the same cycle $C$, so their lengths do not differ by a multiple of $|C|$.

Since $C$ has no entrance, the set $vE^\infty$ is a singleton for each vertex $v$ on $C$. In particular, the cylinder set corresponding to a cycline triple at $v$ of any of the three types is likewise a singleton.
\end{proof} 

For convenience, we refer to an element $\mfkg \in \Iso(\cGss[\mathbb{Z}])^\circ$ as Type~\ref{TypeI}, Type~\ref{TypeII}, or Type~\ref{TypeIII}  at $v$ if it belongs to a cyclinder set $\calZ(\mu,g,\nu)$ where $s(\nu) = v$ and $(\mu,g,\nu)$ is a cycline triple of Type~\ref{TypeI}, Type~\ref{TypeII}, or Type~\ref{TypeIII}, respectively.

In the rest of the section, we show that if $a^p|_C \neq a^{\pm p}$, then Type~\ref{TypeII} and Type~\ref{TypeIII}  elements do not commute with all of Type~\ref{TypeI} elements, which implies that $\ssym$ is maximal among open abelian subgroupoids. We collect computations needed in later proofs, which follow from Definition~\ref{D:ss} and Lemma \ref{lem:cycles to cycles}, and are essentially the same as Lemmas~\ref{lem:k_n} and~\ref{it:powers go in:relationship of generators of ssym}. In an effort to keep this subsection self-contained, we have again included a short proof. Note that pseudo-freeness is needed to ensure non-triviality of the power in Lemma~\ref{compute_cycle} \ref{restrict_i} below, and it is precisely through this that it enters the proof of maximal abelianness of $\ssym$ in $\Iso(\cGss)^{\circ}$.

\begin{lem} 
\label{compute_cycle}
Let $(\mathbb{Z}, E)$ be a 
self-similar graph 
under the standing assumptions of page~\pageref{Notation and conventions}, and let $C$ and $p$ be as described
in Notation~\ref{Notation_cycle}.
Then we have the following: 
\begin{enumerate} [label=\textup{(\arabic*)}] 
\item \label{restrict_i}
For each $i\in \{0, \dots, n-1\}$, there exists $r_i \in \mathbb{Z}$ such that $a^p|_{e_i} = a^{pr_i}$, and if the action is pseudo-free, then $r_{i}\neq 0$;
\item \label{restrict_power_i} 
For any finite path $\mu$ consisting of edges in $C$, $a^{pq}|_{\mu} = (a^{p}|_{\mu})^{q}$ for any $q\in \mathbb{Z}$;  
\item \label{restrict_cycle} 
Let $R\coloneqq r_{n-1} r_{n-2} \dots r_0$, then $a^p|_C = a^{pR}$; 
\item \label{restrict_cycle_other_start} 
For each $i\in \{1, \dots, n-1\}$, 
let $C_i = e_{n+i-1} \dots e_{i+1} e_i$, where all indices are $\hskip -.3cm\mod n$.
Then $a^p|_{C_{i}} = a^{pR}$.
\end{enumerate} 
\end{lem} 

\begin{proof} 
As mentioned in Notation~\ref{Notation_cycle}, $a^{p}$ fixes all finite paths and vertices in $C$. In particular, for each $i\in \{0, \dots, n-1\}$, we have $e_{i}=a^{p}\cdot e_{i}$ and hence
\begin{equation} 
s(e_{i}) = s(a^{p}\cdot e_{i}) = a^{p}|_{e_{i}} \cdot s(e_{i}). 
\end{equation} 
In other words, $a^{p}|_{e_{i}}$ fixes $s(e_{i})$ and thus fixes all edges in $C$. By minimality of $p$, we conclude that $a^{p}|_{e_{i}} = (a^{p})^{r_{i}}$ for some $r_{i}\in \mathbb{Z}$. Since $a^{p}$ fixes $e_{i}$, by pseudo-freeness of the action, we have $a^{p}|_{e_{i}}\neq 1_{G}$ and thus $r_{i}\neq 0$. This proves Statement~\ref{restrict_i}.

Statement~\ref{restrict_power_i} is Lemma~\ref{it:powers go in} applied to $n= 0$, $p= |\mu|$, $h= a^{p}$, and $x=\sigma^{i}(C^\infty)$ for $i$ such that $r(\mu)=v_i$.
This, combined with Statement~\ref{restrict_i}, can be used to prove Statement~\ref{restrict_cycle}:
\begin{align*} 
a^p|_{C} = a^{p}|_{e_{n-1}}|_{e_{n-2}}|_{e_{n-3}}\dots|_{e_{0}} & \overset{
\ref{restrict_i}}{=} a^{p r_{n-1}} |_{e_{n-2}}|_{e_{n-3}}\dots|_{e_{0}} \\
& \overset{
\ref{restrict_power_i}}{=} (a^{p}|_{e_{n-2}})^{r_{n-1}}|_{e_{n-3}}\dots |_{e_{0}} = \dots = a^{pR}. 
\end{align*} 
Statement~\ref{restrict_cycle_other_start} follows from the same computation. 
\end{proof} 

The next lemma characterizes commutativity of Type~\ref{TypeI} and 
Type~\ref{TypeII} elements. 

\begin{lem} \label{lem:TypeIITypeIDontCommute} 
Let $(\mathbb{Z}, E)$ be a pseudo-free self-similar graph under the standing assumptions of page~\pageref{Notation and conventions}, and let $C, p,n$ be as described in  Notation~\ref{Notation_cycle}. Suppose that $(\mu, g, \nu)$ is a Type~\ref{TypeII} cycline triple at a given vertex $v$ on $C$ and $\mfkg \in \calZ(\mu, g, \nu)$. Then $\mfkg$ commutes with all Type~\ref{TypeI} elements at vertices in $C$ if and only if either of the following conditions holds: 
\begin{itemize}
\item $a^p|_C = a^{p}$, or
\item $a^p|_C = a^{-p}$ and $|\mu|-|\nu| \in 2 n\mathbb{Z} \setminus \{0\}$.
\end{itemize}  
\end{lem} 

\begin{proof} 
As $(\mu, g, \nu)$ is of Type~\ref{TypeII}, it is either of the form $(\mu, a^{pq}, \mu C^{k})$ or $(\nu C^{k}, a^{pq}, \nu)$ for some $q\in \mathbb{Z}$ and $k\in \mathbb{N}^+$; we will focus on the case $(\mu, g, \nu) = (\mu, a^{pq}, \mu C^{k})$, as the other case follows from a similar argument. For every Type~\ref{TypeI} element $\mfks$ at a vertex in $C$ that is composable with $\mfkg$, we have $\{\mfks\}=\calZ(\alpha, a^{pq'},\alpha) $ (see the last line of Lemma~\ref{types_of_cycline_triples}) for some $q'\in \mathbb{Z}$ and where $\alpha = (\mu C^\infty)(0,t)$ for some $t\geq |\mu|$, so that $s(\alpha)$ is on $C$. Note that any integer value of $q'$ may be achieved by some Type~\ref{TypeI} element. We fix such a Type~\ref{TypeI} element $\mfks$. 

Let $x=C^\infty$ and $y=\sigma^{t-|\mu|} (C^\infty)$, the unique elements of  $vE^{\infty}$ and $s(\alpha) E^\infty$ respectively.
Since $|C| = n$ (as per Notation~\ref{Notation_cycle}), we have by Lemma~\ref{comm_rela} that $\mfkg \mfks = \mfks \mfkg$ if and only if
we have in $Q(\N, G)$ that
\begin{equation} 
    [
        a^{pq}|_{x(0,|\alpha|)}|_{\sigma^{|\alpha|}(x)}
        \
        a^{pq'}|_{y(0,|\mu|+nk)}|_{\sigma^{|\mu|+nk}(y)}
    ]
    =
    [
        a^{pq}|_{x(0,|\alpha|)}|_{\sigma^{|\alpha|}(x)}
        \
        a^{pq'}|_{y(0,|\mu|)}|_{\sigma^{|\mu|}(y)}
    ]. 
\end{equation} 
Canceling the first factor on both sides, we see that this
equality holds if and only if 
\begin{equation} \label{eqn: TypeII} 
[a^{pq'}|_{y(0,|\mu|+nk)}|_{\sigma^{|\mu|+nk}(y)}] = [a^{pq'}|_{y(0,|\mu|)}|_{\sigma^{|\mu|}(y)}]. 
\end{equation} 
By the definition of $Q(\N, G)$ in Section~\ref{defn_gpd}, this means that $\mfkg \mfks = \mfks \mfkg$ if and only if
\begin{equation} 
\label{restrict_apply_to_l}
a^{pq'}|_{y(0,|\mu|+nk)}|_{\sigma^{|\mu|+nk}(y)} (\ell) = a^{pq'}|_{y(0,|\mu|)}|_{\sigma^{|\mu|}(y)} (\ell)
    \text{
    \quad for all big enough $\ell$.
    }
\end{equation} 
We rewrite the right-hand side:
\begin{equation} 
a^{pq'}|_{y(0,|\mu|)}|_{\sigma^{|\mu|}(y)} (\ell) = a^{pq'}|_{y(0, |\mu|+\ell)}. 
\end{equation} 
Repetitively applying Lemma~\ref{compute_cycle}~\ref{restrict_i} and~\ref{restrict_power_i}
yields
$a^{pq'}|_{y(0, |\mu|+\ell)} = (a^{p})^{Q_{\ell}}$,
where $Q_{\ell}\in \mathbb{Z}$ is the product of $q'$ with $r_{i}$'s for each edge $e_{i}$ of $C$ that is included in $y(0, |\mu|+\ell)$.
Since the action is pseudo-free, it follows from Lemma~\ref{compute_cycle}~\ref{restrict_i} that $r_{i}\neq 0$ for all $i\in \{0, \dots, n-1\}$. Thus, $q' = 0$ if and only if $Q_{\ell}=0$ for some $\ell$
(if and only if $Q_{\ell}=0$ for all $\ell$).
 
Using $a^{pq'}|_{y(0, |\mu|+\ell)} = (a^{p})^{Q_{\ell}}$ in the third equality and Lemma~\ref{compute_cycle}~\ref{restrict_power_i},\ref{restrict_cycle_other_start} in the last equality,
the left-hand side of~\eqref{restrict_apply_to_l} is computed as follows: 
\begin{align*}  
a^{pq'}|_{y(0,|\mu|+nk)}|_{\sigma^{|\mu|+nk}(y)} (\ell) &= a^{pq'}|_{y(0, |\mu|+nk + \ell)}  \\
&=a^{pq'}|_{y(0, |\mu|+\ell)} |_{y(|\mu|+\ell, |\mu|+nk + \ell)}  \\
&= (a^{p})^{Q_{\ell}} |_{y(|\mu|+\ell, |\mu|+nk + \ell)} \\
&
=
(a^{p})^{Q_{\ell}R^{k}}, 
\end{align*}   
where 
$R=r_{n-1}r_{n-2}\dots r_0\neq 0$ is as in Lemma~\ref{compute_cycle}~\ref{restrict_cycle}.

Thus, $\mfkg$ commutes with $\mfks$ if and only if  $(a^{p})^{Q_{\ell}} = (a^{p})^{Q_{\ell}R^{k}}$ for all big enough $\ell$, which is further equivalent to either $R^k=1$ or $Q_{\ell} = 0$ for all big enough $\ell$ (and thus $q' = 0$). 
Since there exist elements of Type~\ref{TypeI} at a vertex in $C$ for which $q'\in \mathbb{Z}\setminus\{0\}$, we have shown that $\mfkg$ commutes with all Type~\ref{TypeI} elements at vertices in $C$ if and only if $R^{k} = 1$. 
Since $a^{p}|_{C}=a^{pR}$ by Lemma~\ref{compute_cycle}~\ref{restrict_cycle}, this is equivalent to either having $a^{p}|_{C}=a^{-p}$ and $k=|\nu|-|\mu|>0$ even (the case $R=-1$), or $a^{p}|_{C}=a^{p}$ (the case $R=1$).
\end{proof}

Under the assumption that $a^p|_C \neq a^{\pm p}$, we show the non-commutativity between Type~\ref{TypeIII} element and Type~\ref{TypeI} elements in the next lemma. 

\begin{lem} \label{lem:TypeIIITypeIDontCommute}
Let $(\mathbb{Z}, E)$ be a pseudo-free self-similar graph under the standing assumptions of page~\pageref{Notation and conventions}, and let $C, p$ be as described in Notation~\ref{Notation_cycle}. Suppose $(\mu, g, \nu)$ is a Type~\ref{TypeIII} cycline triple at a given vertex $v$ on $C$ and $\mfkg \in \calZ(\mu, g, \nu)$. If $a^p|_C \neq a^{\pm p}$, then there exists an element of Type~\ref{TypeI}  which does not commute with $\mfkg$. 
\end{lem}

\begin{proof} 
Since $(\mu, g, \nu)$ is of Type~\ref{TypeIII} at $v$, it follows that $g = a^{q}$ for some $q \in \Z  \setminus p\Z$,  $s(\nu) = v \neq s(\mu)$, $r(\mu) = r(\nu)$, and $a^q \cdot v = s(\mu)$. Take $\mfks$ to be the unique element of the singleton $\calZ(\nu, a^{p}, \nu)$, so that $\mfks$ is of Type~\ref{TypeI}. Let $x=C^\infty$, the unique element of  $v E^{\infty}=s(\nu)E^\infty$. By Lemma~\ref{comm_rela}, $\mfkg \mfks = \mfks \mfkg$ if and only if 
\begin{equation}
[a^q|_{x(0,|\nu|)}|_{\sigma^{|\nu|}(x)} \ a^p|_{x(0,|\nu|)}|_{\sigma^{|\nu|}(x)}] = [a^q|_{x(0,|\nu|)}|_{\sigma^{|\nu|}(x)} \ a^p|_{x(0,|\mu|)}|_{\sigma^{|\mu|}(x)}].
\end{equation} 
The equality holds if and only if
we have for all $\ell\in \N$ big enough,
\begin{align}  
a^p|_{x(0,|\nu|)}|_{\sigma^{|\nu|}(x)}(\ell) &= a^p|_{x(0,|\mu|)}|_{\sigma^{|\mu|}(x)}(\ell), 
\notag
\intertext{or in other words, }
\label{III_restrict_l}
a^p|_{x(0,|\nu|+\ell)} &= a^p|_{x(0,|\mu|+\ell)}.
\end{align} 

By symmetry, we may without loss of generality assume $|\mu| \geq |\nu|$. Note that since $(\mu, g, \nu)$ is a non-symmetric cycline triple, this in fact means $|\mu|>|\nu|$. Since $(G, E)$ is pseudo-free, by repetitively applying Lemma~\ref{compute_cycle}~\ref{restrict_i} and~\ref{restrict_power_i} as in the proof of Lemma~\ref{lem:TypeIITypeIDontCommute}, there exists $Q_{\ell}\in \mathbb{Z} \setminus \{0\}$ such that
\begin{equation} 
a^{p}|_{x(0, |\nu|+\ell)}
=  a^{pQ_{\ell}} .
\end{equation} 
 The right-hand side of~\eqref{III_restrict_l} now becomes 
\begin{equation} 
a^{p}|_{x(0, |\mu|+\ell)}
= a^{pQ_{\ell}}|_{x(|\nu|+\ell, |\mu|+\ell)}. 
\end{equation} 

If $a^p|_C \neq a^{\pm p}$, then 
$|r_{j}|\geq 2$ for some $j\in\{0, \dots, n-1\}$ by Lemma~\ref{compute_cycle}~\ref{restrict_i} and~\ref{restrict_cycle}. Note that for any $\ell$ such that $x(|\nu|+\ell, |\nu|+\ell) = r(e_{j})$, we have 
\begin{align*} 
a^{pQ_{\ell}}|_{x(|\nu|+\ell, |\mu|+\ell)} &= a^{pQ_{\ell}}|_{e_{j}} |_{x(|\nu|+\ell+1, |\mu|+\ell)} 
&&\text{since }|\nu|+1\leq |\mu|
\\&
= a^{pQ_{\ell} r_{j}} |_{x(|\nu|+\ell+1, |\mu|+\ell)} &&\text{by Lemma~\ref{compute_cycle}~\ref{restrict_i} and~\ref{restrict_power_i}}, 
\end{align*} 
which is not equal to $a^{p Q_{\ell}}$. As a result,~\eqref{III_restrict_l} fails for any such $\ell$, which can be taken as large as needed.
Thus $\mfkg \mfks \neq \mfks \mfkg$, which concludes the proof. 
\end{proof} 

We conclude this subsection with the proof that $\ssym$ is maximal among open and abelian subgroupoids of $\Iso(\cGss[\Z])^\circ$, under the assumption that $a^{p}|_{C} \neq a^{\pm p}$. 

\begin{proof}[Proof of Theorem~\ref{thm:Ssym is maximal}] 
We have seen in Proposition~\ref{prop:Ssym:open,normal,abelian,subgpd} that $\ssym$ is an open and abelian subgroupoid of $\Iso(\cGss[\Z])^{\circ}$. It remains to show that $\ssym$ is maximal: By Corollary~\ref{cor:the green corollary} and Lemma~\ref{types_of_cycline_triples}, $\Iso(\cGss[\Z])^{\circ}\setminus \ssym$ 
consists exactly of Type~\ref{TypeII} and Type~\ref{TypeIII} elements at vertices on cycles with no entrance. Under the condition that $a^{p}|_{C} \neq a^{\pm p}$, by Lemmas~\ref{lem:TypeIITypeIDontCommute} and~\ref{lem:TypeIIITypeIDontCommute}, Type~\ref{TypeII} 
respectively Type~\ref{TypeIII} elements do not commute with all Type~\ref{TypeI} elements, which are contained in $\ssym$ in particular.
Thus, any subset of $\Iso(\cGss[\Z])^{\circ}$ larger than $\ssym$ fails to be abelian.
\end{proof} 

In cases where $a^p|_C = a^p$ or $a^p|_C = a^{-p}$, elements of Type~\ref{TypeII} or Type~\ref{TypeIII} may commute with elements of $\ssym$. The resulting computations become significantly more technical and largely ad hoc, although the underlying proof strategy remains essentially the same as that presented in this section. For this reason, we do not pursue these cases in greater generality. A discussion of maximal abelian subgroupoids in the case $a^p|_C = a^p$ is included in Appendix~\ref{sec:AnotherCandidate} for interested readers.

\subsection{Closedness of $\ssym$} 
\label{ssec:closedness} 
In this Subsection, we characterize closedness in $\cGss$ for  subgroupoids $\calS$ of $\Iso(\cGss)^{\circ}$ containing $\ssym$. In particular, we provide equivalent conditions for closedness of both $\ssym$ and $\Iso(\cGss)^{\circ}$, which are necessary for proving corresponding results for Cartan subalgebras (Theorems~\ref{Thm_Ssym_Cartan}, \ref{Thm_genuine_action_Cartan}, and Appendix~\ref{Thm_Salt_Cartan}).

Given any subgroupoid $\calS$ of $\cGss$, we will write $ \calS^{c} \defeq \cGss  \setminus \calS$ for its complement in $\cGss$.
Note that if $\calS$ is further a subgroupoid of $\Iso(\cGss)$, we have the following:
\begin{align}
\label{eq:calS closed iff}
         \calS\text{ is closed in } \cGss\iff\Iso(\cGss) \subseteq \calS \sqcup (\calS^{c})^\circ.
\end{align} 
This follows since $\Iso(\cGss)$ is closed in the Hausdorff groupoid $\cGss$ (as explained in Section~\ref{defn_gpd}, Hausdorffness follows from pseudo-freeness of the self-similar action).

In the rest of this section, we translate the abstract groupoid dichotomy on the right hand side of the equivalence 
\eqref{eq:calS closed iff}
into concrete dynamical properties of the underlying self-similar
graph (Theorem~\ref{thm:closedness-equivalent}).
This perspective naturally leads to  the following notions of $g$-generic and $g$-rare infinite paths. The notion of $g$-generic infinite paths can be traced back at least to the work of Nekrashevych (see \cite[Section 4]{Nekrashevych_generic} for instance, where an infinite path $x$ is said to be \emph{generic with respect to $g\in G$}). In contrast, the notion of $g$-rare appears to be new, at least to the best of the authors' knowledge, and is intended to capture essentially the opposite dynamical 
phenomenon of $g$-genericity.

\begin{defn}\label{def:GenaricRare}
Let $(G,E)$ be a self-similar graph, and let $x\in E^{\infty}$. For any $g\in G$, we say that 
\begin{itemize} 
\item $x$ is \emph{$g$-generic} if either $g\cdot x \neq x$, or $g$ pointwise fixes a neighborhood of $x$; 
\item $x$ is \emph{$g$-rare} if  $g\cdot x = x$, and there exists $N\in\N$ such that $g$ does not pointwise fix any neighborhood of any $z\in x(0,N)E^\infty$.
\end{itemize}
\end{defn} 

\begin{rem}
    Unpacking the definition of $g$-rarity, one sees that $x$ is $g$-rare if and only if there exists $N\in\N$ with the following property:
    if $y\in x(N,N)E^\infty$ satisfies $g|_{x(0,N)}\cdot y = y$, then for all $m\in \N$, the element $g|_{x(0,N)}|_{y(0,m)}$ does not fix all of $y(m,m)E^\infty$.
\end{rem} 

The connection between the abstract groupoid condition in~\eqref{eq:calS closed iff} and these dynamical characterizations is made precise in the following lemma. Note that the pseudo-freeness assumption will be used throughout the remainder of the argument for closedness, as we make essential use of Lemma~\ref{lem:opensetexists} in what follows. 

\begin{lem} 
\label{lem:g-generic-iso} 
Let $(G,E)$ be a pseudo-free self-similar graph. Take an infinite path $x$ and some $g\in G$ such that $g\cdot x = x$, and a finite path $\mu$ such that $s(\mu) = x(0, 0)$. Let $\mfkg = \bigl(\mu x; \calT_{|\mu|}([g|_{x}]),  0; \mu x\bigr)$ be an element of $\cGss$. The following are equivalent:
\begin{enumerate}[label={\textup{(\arabic*)}}]
  \item \label{lem:g-generic-iso:it:g-gen} \(x\) is \(g\)-generic;
  \item\label{lem:g-generic-iso:it:Ssym}
        \(\mfkg \in \ssym\);
        \item\label{lem:g-generic-iso:it:Iso} 
        \(\mfkg \in \Iso(\cGss)^\circ\);
  \item \label{lem:g-generic-iso:it:calS} $\mfkg\in \calS$ for any subgroupoid $\ssym\subseteq \calS\subseteq\Iso(\cGss)^\circ$.
\end{enumerate}
Likewise, the following are equivalent:
\begin{enumerate}[label={\textup{(\roman*)}}]
  \item\label{lem:g-generic-iso:it:g-rare} \(x\) is \(g\)-rare;
  \item\label{lem:g-generic-iso:it:Ssym-c}
        \(\mfkg \in (\ssym^{c})^\circ\);
        \item\label{lem:g-generic-iso:it:Iso-c}
        \(\mfkg \in\bigl(\bigl(\Iso(\cGss)^\circ\bigr)^{c}\bigr)^\circ\);
        \item\label{lem:g-generic-iso:it:calS-c}
        \(\mfkg \in (\calS^{c})^{\circ}\) for any subgroupoid
        \(\ssym\subseteq\calS\subseteq\Iso(\cGss)^\circ\);
\end{enumerate}
\end{lem}

\begin{proof} 
The implications
\ref{lem:g-generic-iso:it:Ssym}$\implies$\ref{lem:g-generic-iso:it:Iso} and~\ref{lem:g-generic-iso:it:Iso-c}$\implies$\ref{lem:g-generic-iso:it:Ssym-c} follow from the fact that $\ssym\subseteq \Iso(\cGss)^\circ$, and the equivalences~\ref{lem:g-generic-iso:it:g-gen}+\ref{lem:g-generic-iso:it:Ssym}$\iff$\ref{lem:g-generic-iso:it:calS} and~\ref{lem:g-generic-iso:it:g-rare}+\ref{lem:g-generic-iso:it:Ssym-c}$\iff$\ref{lem:g-generic-iso:it:calS-c} are straightforward. It remains to
prove~\ref{lem:g-generic-iso:it:Iso}$\implies$\ref{lem:g-generic-iso:it:g-gen}$\implies$\ref{lem:g-generic-iso:it:Ssym} and~\ref{lem:g-generic-iso:it:Ssym-c}$\implies$\ref{lem:g-generic-iso:it:g-rare}$\implies$\ref{lem:g-generic-iso:it:Iso-c}.

\ref{lem:g-generic-iso:it:Iso} $\implies$~\ref{lem:g-generic-iso:it:g-gen}: Assume that $\mfkg \in \Iso(\cGss)^\circ$.
Since $(G, E)$ is pseudo-free, by Lemma~\ref{lem:opensetexists}, there exists $n\in \N$ such that 
\begin{equation} 
\mfkg\in \calZ\bigl(\mu (g\cdot x(0,n)), g|_{x(0,n)}, \mu x(0,n)\bigr)\subseteq \Iso(\cGss)^\circ .
\end{equation} 
By~\eqref{eq:int of Iso}, this
implies that $g|_{x(0,n)}$ fixes any $y\in x(n,n) E^{\infty}$. 
Note that our assumption $g\cdot x = x$ implies $g\cdot x(0,n) = x(0,n)$. Combining this with $g|_{x(0,n)}\cdot y=y$ yields
that $g$ pointwise fixes the neighborhood $\calZ(x(0,n))$ of $x$, so that $x$ is $g$-generic.

\ref{lem:g-generic-iso:it:g-gen} $\implies$~\ref{lem:g-generic-iso:it:Ssym}: If $x$ is $g$-generic, then since we assumed $g$ fixes $x$, $g$ pointwise fixes  a neighborhood of $x$, say $\calZ(x(0,n))$ for some $n\in\N$.  This means that for any $y\in x(n,n)E^{\infty}$, we have
\mbox{$g\cdot (x(0,n)y) = x(0,n)y$,}
which implies
\mbox{$g|_{x(0,n)} \cdot y = y$}
by self-similarity.
This means that 
$(\mu x(0,n), g|_{x(0,n)}, \mu x(0,n))$ is a symmetric cycline triple.
Since $\mfkg$ is in the associated cylinder set by construction, our definition of $\ssym$ implies $\mfkg \in \ssym$, as needed.

\ref{lem:g-generic-iso:it:Ssym-c} $\implies$~\ref{lem:g-generic-iso:it:g-rare}:
We will show the contrapositive, so suppose
that $x$ is not $g$-rare; it suffices to prove that any open neighborhood $U$ of $\mfkg$ intersects $\ssym$ non-trivially. Since $(G,E)$ is pseudo-free, again by Lemma~\ref{lem:opensetexists}, there exists $n\in\N$ such that 
\begin{equation} 
\mfkg \in  \calZ\bigl(\mu x(0,n), g|_{x(0,n)}
        , \mu x(0,n)\bigr) \subseteq U.
\end{equation} 
Since $x$ is not $g$-rare and $g\cdot x=x$ by assumption,
there exists $y\in x(n,n)E^\infty$ and $M\in \N$ such that $g|_{x(0,n)}\cdot y=y$ and $g|_{x(0,n)}|_{y(0,M)}$ fixes every element of $y(M,M)E^\infty$. 
This means that the symmetric triple 
\[\bigl(\mu x(0,n)y(0,M), g|_{x(0,n)}|_{y(0,M)}
        , \mu x(0,n)y(0,M)\bigr)\]
        is cycline, and so the corresponding cylinder set
    is fully contained in $\ssym$. The cylinder set's element
    \[
        \bigl(\mu x(0,n)y; \calT_{|\mu|+n}([g|_{x(0,n)}|_{y(0,M)}|_{\sigma^{M}(y)}]),  0; \mu x(0,n)y\bigr)
    \]
    is, by choice of $n$, also an element of $U$, proving that $U\cap \ssym\neq\emptyset$.

\ref{lem:g-generic-iso:it:g-rare}$\implies$\ref{lem:g-generic-iso:it:Iso-c}: Suppose that $x$ is $g$-rare, so
there exists $N\in\N$ with the following property: if $y\in x(N, N)E^\infty$ satisfies $g|_{x(0,N)}\cdot y = y$, then for all $m \in \N$, the group element $g|_{x(0,N)}|_{y(0,m)}$ does not fix every point in $y(m,m)E^\infty$. 
We claim that the neighborhood
\begin{equation} 
U\defeq  \calZ \bigl( \mu x(0,N) , g|_{x(0,N)}, \mu x(0,N) \bigr) 
\end{equation} 
 of $\mfkg$ 
is contained in $\bigl(\Iso(\cGss)^\circ\bigr)^{c}$, which will prove the claim.
Suppose that $U\cap \Iso(\cGss)^\circ$ is non-empty, 
so there exists $y\in x(N, N)E^{\infty}$ such that 
\begin{equation} 
\mfkt \defeq \bigl(\mu x(0,N)  (g|_{x(0,N)}\cdot y);  \calT_{|\mu|+N}([g|_{x(0,N)}|_{y}]),0; \mu x(0,N)  y\bigr)
\end{equation} 
is an element of $\Iso(\cGss)^\circ$.
In particular, 
$r(\mfkt)=s(\mfkt)$, which implies
$g|_{x(0,N)}\cdot y
=
y$ 
and thus $g|_{x(0,N)}\cdot y(0, m) = y(0, m)$ for any $m\in \N$. 
By Lemma~\ref{lem:opensetexists},  there exists $M\in\N$ such that
\begin{equation} 
\mfkt \in \calZ\bigl(\mu x(0, N) y(0,M), g|_{x(0,N)}|_{y(0, M)}, \mu x(0, N) y(0,M)\bigr) \subseteq U\cap \Iso(\cGss)^\circ.
\end{equation}
Since the cylinder set is contained in $\Iso(\cGss)$,
$g|_{x(0,N)}|_{y(0, M)}$ fixes every infinite path in $y(M, M) E^{\infty}$, which is a contradiction to our assumption that $x$ is $g$-rare.  
\end{proof} 

Combining the equivalence in~\eqref{eq:calS closed iff} and Lemma~\ref{lem:g-generic-iso}, we characterize closedness of any subgroupoid $\calS$ sitting between $\ssym$ and $\Iso(\cGss)$. 

\begin{thm}
\label{thm:closedness-equivalent} 
Let $(G,E)$ be a pseudo-free self-similar graph, where $G$ is 
a group containing
an element of infinite order, and $E$ has finitely many vertices and no sources. 
The following statements are equivalent.
\begin{enumerate}[label={\textup{(\arabic*)}}]
  \item\label{thm:closedness-equivalent:it:split into generic and rare} 
  For
  every \(g\in G\) and every \(x\in E^\infty\),
        the path \(x\) is either \(g\)-generic or \(g\)-rare.
  \item\label{thm:closedness-equivalent:it:closed} 
        Every subgroupoid $\calS$
        with
        \(\ssym\subseteq\calS\subseteq\Iso(\cGss)^{\circ}\)
    is closed in $\cGss$.
    \item\label{thm:closedness-equivalent:it:just ssym, closed}
    The subgroupoid $\ssym$ is closed in $\cGss$.
\end{enumerate}
\end{thm}

\begin{proof} 
As noted in~\eqref{eq:calS closed iff},~\ref{thm:closedness-equivalent:it:closed} holds if and only if $\Iso(\cGss) \subseteq \calS \sqcup (\calS^{c})^\circ$ for any subgroupoid $\calS$ with  $\ssym\subseteq\calS\subseteq\Iso(\cGss)^{\circ}$. 
Note further that \ref{thm:closedness-equivalent:it:closed}$\implies$\ref{thm:closedness-equivalent:it:just ssym, closed} is obvious.

\ref{thm:closedness-equivalent:it:just ssym, closed}
$\implies$\ref{thm:closedness-equivalent:it:split into generic and rare}: For any infinite path $x$ and $g\in G$ such that $g$ fixes $x$, and a finite path $\mu$ such that $s(\mu) = x(0, 0)$, the isotropy element $\mfkg = \bigl(\mu x; \calT_{|\mu|}([g|_{x}]),  0; \mu x\bigr)$ is contained in $\ssym \sqcup (\ssym^{c})^\circ$ by~\eqref{eq:calS closed iff}. By Lemma~\ref{lem:g-generic-iso}, this implies~\ref{thm:closedness-equivalent:it:split into generic and rare}. 

\ref{thm:closedness-equivalent:it:split into generic and rare}$\implies$\ref{thm:closedness-equivalent:it:closed}: Fix a subgroupoid $\calS$ with $\ssym\subseteq\calS\subseteq\Iso(\cGss)^{\circ}$. By~\eqref{eq:calS closed iff}, it suffices to show that $\Iso(\cGss)\subseteq \calS \cup (\calS^c)^\circ$. Fix an element of $\Iso(\cGss)$, say
 \begin{align}
\label{eq:arb elt of Iso}
  \mfkg 
   = \bigl(\mu(g\cdot x)   ; 
           \calT_{|\mu|}([g|_{x}])  , 
           |\mu|-|\nu|   ; 
           \nu x\bigr),
           \text{ where \(\mu(g\cdot x)=\nu x\)}.
\end{align}
If $|\mu| = |\nu|$, then 
$\mu (g\cdot x) = \nu x$ implies $\mu = \nu$ and $g \cdot x = x$,
so Lemma~\ref{lem:g-generic-iso} applies.  The assumption that every infinite path is either $g$-generic or $g$-rare allows us to conclude that $\mfkg$ is either in $\calS$ or in $(\calS^c)^\circ$, as needed.

Now assume that $|\mu| \neq |\nu|$. We first consider  the case that $\mfkg\in \Iso(\cGss)^{\circ}$. Since $(G, E)$ is pseudo-free, by Lemma~\ref{lem:opensetexists}, there exists $n \in \mathbb{N}$ such that
\begin{equation} 
\mfkg \in \calZ\bigl(\mu (g\cdot x(0,n)), g|_{x(0,n)}, \nu x(0,n)\bigr) \subseteq \Iso(\cGss)^{\circ}, 
\end{equation} 
which implies that $\bigl(\mu (g\cdot x(0,n)), g|_{x(0,n)}, \nu x(0,n)\bigr)$ is a cycline triple at $x(n,n)$ and is non-symmetric as $|\mu| \neq |\nu|$. By Corollary~\ref{cor:the green corollary}, 
\begin{equation} 
\label{singleton_neighbor} 
\calZ\bigl(\mu (g\cdot x(0,n)), g|_{x(0,n)}, \nu x(0,n)\bigr) = \{\mfkg\}. 
\end{equation} 
If $\mfkg \not\in (\calS^c)^\circ$, meaning that every neighborhood of $\mfkg$ intersects $\calS$ nontrivially, then $\mfkg \in \calS$ as a consequence of~\eqref{singleton_neighbor}. 

Lastly, we consider the case that $\mfkg \in \Iso(\cGss) \setminus \Iso(\cGss)^{\circ}$; 
in particular, the neighborhood $\calZ(\mu,g,\nu)$ of $\mfkg$ is not fully contained in $\Iso(\cGss)$. 
Consider $\mfkg\neq \mfkh \in \calZ(\mu,g,\nu)$, and let $y\in s(\nu) E^\infty \setminus\{x\}$ be such that
\(r(\mfkh)=\mu (g\cdot y)\) and \(s(\mfkh)=\nu y\).
Now, as $|\mu| \neq |\nu|$, it follows from Lemma~\ref{lem:one path}~\ref{it:one path:unique infinite path} that $x$ is the {\em unique} infinite path at $s(\nu)$ satisfying 
\mbox{$\mu (g\cdot x) = \nu x$;} in particular, $\mu (g\cdot y )\neq \nu y$. This implies that $\mfkh\notin \Iso(\cGss)$, and so we have shown $\calZ(\mu, g, \nu)\cap \Iso(\cGss) = \{\mfkg\}$. As $\mfkg\notin \Iso(\cGss)^\circ$, this implies $\calZ(\mu, g, \nu)\cap \Iso(\cGss)^\circ = \emptyset$, and therefore
\begin{equation} 
\mfkg \in \big(\Iso (\cGss)^{c}\big)^{\circ}\subseteq (\calS^c)^\circ. 
\qedhere
\end{equation}
\end{proof} 

A direct consequence is the characterization of Cartan subalgebras coming from the transformation groupoid in the genuine action case.

\begin{proof}
[Proof of Theorem~\ref{Thm_genuine_action_Cartan}] 
\label{proof:cartan_geuine_action} 
By Corollary~\ref{C:can-act}, for the genuine action case, $\Iso(\cGss)^{\circ}$ is an abelian subgroupoid. Then it follows from \cite[Corollary 4.5]{BNRSW:Cartan} that $C^*_{r}(\Iso(\cGss)^\circ)$ is a Cartan subalgebra in $C^*_{r}(\cGss)$ if and only if $\Iso(\cGss)^{\circ}$ is closed in $\cGss$. The equivalence holds by a direct application of Lemma~\ref{thm:closedness-equivalent}. 
\end{proof} 

The following example exhibits a self-similar graph which contains $g$-generic, $g$-rare, and  neither $g$-generic nor $g$-rare infinite paths. Moreover, we demonstrate that the main theorem of the section (Theorem~\ref{Thm_Ssym_Cartan}) applies to more general examples. 

\begin{exa} 
\label{exa:closed_g_generic} 
Consider the following self-similar graph $(\Z, E)$, which can be extended from the action and the restriction defined for the generator $a\in\Z$ on edges: 
\begin{align}
& a\cdot e_i=e_{i+1 \text{ mod 2}}, 
\ a\cdot f_i=f_{i+1 \text{ mod 2}}, 
\ a\cdot \mu=\mu \text{ for any edge } \mu\ne e_i, f_i,\\
&a|_{e_1}
=a|_{d_5}=a^2, \ a|_{\nu}=a \text{ for any edge }\nu \ne e_1, d_{5}. 
\end{align}

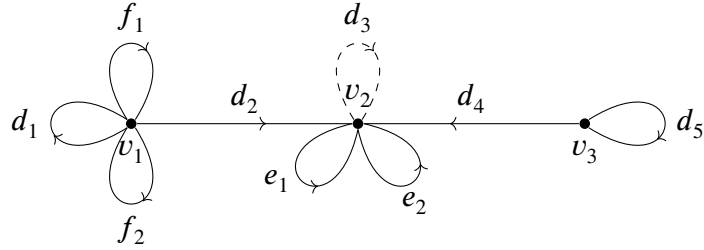
\begin{figure}[h]  
\label{fig_exa} 
    \centering
   
    \begin{tikzpicture}[ dot/.style={circle, fill=black, draw=none, inner sep=1.4pt}]
\def\LoopW{1.4cm}
\def\LoopH{0.9cm}


\node[dot, label=below:$v_1$] (v1) at (0,0) {};
\node[dot, label=above:$v_2$] (v2) at (3,0) {};
\node[dot, label=below:$v_3$] (v3) at (6,0) {};


\draw[->-] (v1) -- node[above] {$d_2$} (v2);
\draw[->-] (v3) -- node[above] {$d_4$} (v2);


\LoopLeft{v1}{$d_1$}
\LoopAbove{v1}{$f_1$}
\LoopBelow{v1}{
{$f_2$}}

\LoopAboveDash{v2}{$d_3$} 
\LoopLeftRotCCW{v2}{$e_1$}{45}
\LoopBelowRotCCW{v2}{
{$e_2$}}{45}

\LoopRight{v3}{$d_5$}

\end{tikzpicture}
    
    \caption{A self-similar graph $(\mathbb Z, E)$ with all types of infinite paths.}
    \label{closed_exam}
\end{figure}

For this self-similar graph $(\Z, E)$, there are infinite paths 
that are $a$-generic, $a$-rare, and neither $a$-generic nor $a$-rare. More specifically,
we have the following: 
\begin{enumerate}
\item \label{exa_1} 
$(e_1e_2)^\infty$ is $g$-generic for every $g \neq 1_{{G}}
$, since $(e_1e_2)^\infty$ is not fixed by any non-trivial group element. Indeed, any odd power of $a$ maps $e_{1}$ to $e_{2}$, hence moves the infinite path. Any even power of $a$, when restricted to $e_{1}$, becomes an odd power and thus also moves the remaining infinite path;
\item $d_5^\infty$ is $g$-generic for every $g\in\mathbb Z$ as it is fixed by every $g$ and $v_3E^\infty=\{d_5^\infty\}$;
\item $d_1^\infty$ is $g$-rare for every $g\in \mathbb{Z}$ that is 
an
odd power of 
$a$: For $N=0$,
any neighborhood of any $z\in d_1^\infty (0,N)E^{\infty}
=
v_{1}E^{\infty}$ coincides with $v_{1}E^{\infty}$. 
Since $g$ flips $(f_{1}f_{2})^{\infty}$ to $(f_{2}f_{1})^{\infty}$, it does not fix  $v_{1}E^{\infty}$ pointwise, proving that  $d_{1}^{\infty}$ is $g$-rare.

\item $d_3^\infty$ is neither $g$-generic nor $g$-rare for $g\neq 1_{{G}}$:  By a similar argument to~\eqref{exa_1}, non-trivial group elements do not fix 
$v_{2}E^{\infty}$ pointwise,
which implies that $d_3^\infty$ is not $g$-generic for any $g \neq 1_{{G}}$.
Secondly,
for any $N\in \mathbb{N}$, the infinite path $d_{3}^{N} d_{4} d_{5}^{\infty}$ is in $d_{3}^{N} E^{\infty}$ and its neighborhood $d_{3}^{N} d_{4} E^{\infty}$ is the singleton $\{d_{3}^{N} d_{4} d_{5}^{\infty}\}$ and is thus fixed by any group element.
Hence, $d_3^\infty$ is not $g$-rare for any $g \neq 1_{{G}}$. 
\end{enumerate}

If we remove the dashed edge $d_3$ in Figure~\ref{closed_exam}, then 
for each $g\in G$, each infinite path of the new resulting self-similar graph $(\mathbb{Z}, E')$ is either $g$-generic or $g$-rare. It is standard to check that $(\mathbb{Z}, E')$ is pseudo-free. By the main theorem of the section (Theorem~\ref{Thm_Ssym_Cartan}, whose proof is in the next subsection), the subgroupoid $\ssym$ in $\calG_{\mathbb{Z}, E'}$ gives a Cartan subalgebra of $C^*_{r}(\calG_{\mathbb{Z}, E'})$. 
Furthermore, the self-similar graph $(\mathbb{Z}, E')$ contains both a locally faithful vertex $v_{2}$ and non-locally faithful vertices $v_{1}$ and $v_{3}$.  
To the best of our knowledge, this is the first example of self-similar graphs exhibiting mixed dynamical behavior, while admitting Cartan subalgebras coming from canonical groupoid models.

\end{exa}

\subsection{Cartan subalgebras from $\ssym$}\label{ssec:Ssym satisfies ricc}
This subsection is devoted to verifying that,
when $G$ is torsion free, the subgroupoid $\ssym$ of $\cGss$ satisfies~\eqref{eq:ricc,stronger}. For the convenience of the reader, we repeat the condition here, where $\calG=\cGss$ and $\calS= \ssym$ need to be substituted:
\begin{equation}
\tag{\ref{eq:ricc,stronger}, repeated}
    \{\mfkg\in \Iso(\calG)^{\circ}: 1<|\{\mfks^{-1} \mfkg \mfks  : \mfks \in \calS\}|< \infty \} = \emptyset. 
    \end{equation} 

\begin{rem}  \label{rmk:ricc,stronger}
Equation~\eqref{eq:ricc,stronger} is an improvement on a condition called \emph{immediately centralizing}, which was introduced in \cite{DGNRW}. 
To understand why the $\mathcal{S}$-conjugates have relevance for $C^*_{r}(\mathcal{S})$ to be masa in $C^*_{r}(\calG)$, let us consider the discrete case:  If an element $\mfkg_0\in \calG(x)$ is in the set on the left-hand side of~\eqref{eq:ricc,stronger}, meaning that there exist finitely many elements $\mfkg_1,\dots,\mfkg_n$ such that $\{
               \mfks^{-1} \mfkg_0 \mfks  : \mfks \in \mathcal{S}(x) 
            \}=\{\mfkg_0,\dots,\mfkg_{n}\}$, then one can easily write down an element of $C^*_{r}(\calG)\setminus C^*_{r}(\mathcal{S})$ that commutes with all of $C^*_{r}(\mathcal{S})$. Indeed, for each $0\leq i\leq n$ and $\mfks \in\mathcal{S}(x)$, there exists a unique $0\leq j\leq n$ that satisfies $\delta_{\mfkg_{i}}\delta_{\mfks }=\delta_{\mfks }\delta_{\mfkg_{j}}$, so that the sum $\sum_{i=0}^{n}\delta_{\mfkg_{i}}$ of point-masses commutes with every point-mass $\delta_{\mfks }$ and hence with all of $C^*_{r}(\mathcal{S})$. However, since $\mathcal{S}$ is abelian, the elements $\mfkg_{i}$ are not in $\mathcal{S}$ and the sum of their point-masses is hence not in the subalgebra, so the existence of $\mfkg_0$ prevents $C^*_{r}(\mathcal{S})$ from being masa. It was further explained in \cite[Remark 3.6]{DWZ:Twist} why elements $\mfkg$ of $\Iso(\calG)^{\circ}$ that are not in the set on the left-hand side of~\eqref{eq:ricc,stronger} cannot
            stand in the way of $C^*_{r}(\mathcal{S})$ being maximal abelian. In the non-discrete case, there is a bit more leeway and one does not need to assume that~\eqref{eq:ricc,stronger} holds everywhere; see \cite[Proposition 3.14]{DWZ:Twist} for a precise statement. 
            For our purposes,~\eqref{eq:ricc,stronger} will be sufficient, as Lemma~\ref{lem:Ssym (almost) immediately centralizing} below shows. 
\end{rem}

\begin{lem}\label{lem:Ssym (almost) immediately centralizing}
Let $G$ be a countable discrete group that is abelian and torsion free, and let $E$ be a finite graph. Let $(G, E)$ be a self-similar action. Then $\ssym
\subseteq \cGss$ satisfies~\eqref{eq:ricc,stronger}.
\end{lem}
\begin{proof}
Assume that for some $x\in\cGss\z$ and $\mfkg\in \cGss(x)$, there exists $\mfkt\in\ssym(x)$ for which $\mfkt^{-1}\mfkg\mfkt\neq \mfkg$. 
Then 
\begin{align*} 
\mfkt  &= ( x ; \ \calT_{ l }([t|_{\sigma^{l}(x)}]),\ 0; \  x ), 
\end{align*} 
for some $l\in\mathbb{N}$, $t\in G_{x(l, l)}$. We write $y \coloneqq \sigma^{l}(x)$. 
Since $t$ fixes $x(l, l) E^{\infty}$,
it follows from Lemma~\ref{it:powers go in} that
\[
   \ssym(x)\ni
   \mfkt^{p}
   \overset{\ref{it:powers go in}}{=}
   ( x ; \ \calT_{ l }([t^{p}|_{ y }]),\ 0; \  x )
   \quad
   \text{ for any }p\in \Z.
\] 
It suffices to show that
\(
\mfkt^{-p}\mfkg\mfkt^{p}\neq  \mfkg 
\)
for all $p\neq 0$.
The element $\mfkg$ is of the form
\begin{align*} 
\mfkg  &= (x;  \ \calT_{ m }([g|_{\sigma^{n}(x)}]), \ k ; \  x )
    \quad\text{ where }
    k = m-n.
\end{align*}
Since $G$ and hence $Q(\N ,G)$ is abelian, we have by Lemma~\ref{it:powers go in}
\begin{align} 
\mfkt^{-p} \mfkg \mfkt^{p}
&= 
( x ; \ \calT_{ l }([t^{-p}|_{ y }]) \calT_{ m }([g|_{\sigma^{n}(x)}]) \calT_{ k }\calT_{ l }([t^{p}|_{ y }]),\  k ; \  x )
\\
&\overset{\ref{it:powers go in}}{=} 
( x ; \ \calT_{ l }([t^{p}|_{ y }])^{-1}
\calT_{ k + l }([t^{p}|_{ y }])
\calT_{ m }([g|_{\sigma^{n}(x)}]) ,\  k ; \  x ) .
\label{eq:g conjugate by t_p}
\end{align} 
Because of our assumption $\mfkt^{-1} \mfkg \mfkt\neq \mfkg$, the above computation for $p=1$ implies that
\begin{equation}
    \calT_{ l }([t|_{y}])
    \neq
    \calT_{ k + l }([t|_{y}]).
\end{equation}
In particular, there exists an infinite  sequence $k+l\leq d_{1} < d_{2} <\dots$ of increasing natural numbers such that 
\begin{equation}\label{eq:t different}
   t|_{y(0, d_{i}- k)} 
   \neq 
   t|_{y(0, d_{i} )}
   \text{ for all }i\in\N.
\end{equation} 
Fix $p\neq 0$. Since $t^{p}\in G_{r(y)}$, we have \(
    t^{p}|_{y(0, d)} = (t|_{y(0, d)})^{p}\)
for all $d\in \N$ by Lemma~\ref{it:powers go in}.
Since $G$ is assumed to be torsion free, this implies that the inequality at~\eqref{eq:t different} holds not only for $t$ but also for $t^{p}$. Combined with Equation~\eqref{eq:g conjugate by t_p}, this means exactly that $\mfkt^{-p}\mfkg\mfkt^{p}\neq \mfkg$ for $p\neq 0$.
\end{proof}

Combining results in Section~\ref{S:sym} and~\ref{sec: Cartan subalgebras}, we prove our main theorem for $\ssym$. 
\begin{proof}[Proof of Theorem~\ref{Thm_Ssym_Cartan}] 
The sufficient conditions appearing in Theorem~\ref{thm:DWZ} have all been verified for $\ssym$. Indeed, Proposition~\ref{prop:Ssym:open,normal,abelian,subgpd} shows that $\ssym$ is an open, abelian, normal subgroupoid of $\cGss$; $\ssym$ is maximal abelian due to Theorem~\ref{thm:Ssym is maximal}; closedness of $\ssym$ follows from Theorem~\ref{thm:closedness-equivalent}; and lastly,~\eqref{eq:ricc,stronger} is satisfied by Lemma~\ref{lem:Ssym (almost) immediately centralizing}. 
\end{proof}

\appendix

\section{Examples for Section 4}
\label{app:Weyl examples}
In this appendix, we will consider specific examples of self-similar graphs $(\Z,E)$ and use our results from the previous sections to describe the groupoid $\ssym$, its dual $\ssymdual$, and the action of $\cGss[\Z]$ on $\ssymdual$. As before, we write $\Z=\langle a\rangle$.

\subsection{One loop}
    Let $E^{0}=\{v\}$ and $E^1=\{e\}$, so $E$ is a single loop.

    \begin{figure}[h] 
    \centering

    \begin{tikzpicture}[ dot/.style={circle, fill=black, draw=none, inner sep=1.4pt}]
\def\LoopW{1.4cm}
\def\LoopH{0.9cm}


\node[dot, label=below:$v$] (v) at (2,0) {};


%
\LoopRight{v}{$e$}

\end{tikzpicture}
    
    \caption{A loop.}
    \label{fig:loop}
\end{figure}

    This forces the $\Z$-action on $E$ to be 
    trivial.
    We let the restriction map be determined by 
    \[
    a|_{e}=a^2.
    \]
    One can compute $a^n|_{e^m}=a^{2^mn}$, so the action is pseudo-free.
    Since  there is a single infinite path, namely $x=e^\infty$, we have $\ssym=\ssym(e^\infty)$ and $\Z_v = \Z$. This means that the elements $h_n=h_{x,n}$ defined at~\eqref{eq:def'n h_n} are all equal to the generator $a$. Using the notation from Theorem~\ref{thm:ssymdual}, we thus have that $\ssym(e^\infty)$ is generated by the elements
    \[
        \mfks_{x,n}=\mfks_{n}
        =
        \left(
            e^\infty;
            \calT_{n}([a|_{e^\infty}])
            , 0, e^\infty
        \right)
        \quad\text{ for }n\geq 0.
    \]

    For an arbitrary element
    \(        \mfkg
        =
        \left(
            e^\infty;
            [f], k ;
            e^\infty
        \right)
    \)
    of $\cGss[\Z]$,
     we see from Equation~\eqref{eq:conjugation} that
    \[
        \mfkg\inv \mfks_{n} \mfkg
        =
        \left(
            e^\infty; 
            \calT_{n-k}([a|_{e^\infty}])
            , 0, e^\infty
        \right),
    \]
    which, for $k\leq n$, is exactly $\mfks_{n-k}$. 
    For $k\geq n$, it is easy to check that our definition of the restriction map implies
    \begin{equation}\label{eq:calT on a|e^infty}
       \calT_{n-k}([a|_{e^\infty}])
       = [a^{2^{|n-k|}}|_{e^\infty}] = [a|_{e^\infty}]^{2^{|n-k|}}.
    \end{equation}
    Thus, all in all, the action on $\kappa\in\ssymdual$ is given by: 
    \[
        (\mfkg\HleftX \kappa)(\mfks_{n})
        =
        \begin{cases}
            \kappa(\mfks_{0})^{2^{|n-k|}} & \text{ if } n-k \leq 0,
            \\
            \kappa(\mfks_{n-k}) & \text{ if } n-k\geq 0.
        \end{cases}
    \]
    (Note that~\eqref{eq:mult in ssym(x)}
    implies for $n-k\geq 0$ that
    \(
     \mfks_{0}
     = 
     \mfks_{n-k}^{2^{n-k}},
    \) 
    so the above formula is very natural.)

    Because we defined $a|_{e}=a^2$, we have for $m\in\N$ that
    \[
        h_{m}|_{x(m,m+1)}
        =
        a|_{e}
        =
        a^{2}
        =
        h_{m+1}^2;
    \]
    in other words, the integers $k_n$ from Lemma~\ref{lem:k_n} are all equal to $2$. 
    Thanks to Theorem~\ref{thm:ssymdual}, we can identify $\ssymdual$ with $\varprojlim (\T, \mvisiblespace^2)$  via 
    \[
    \ssymdual \ni
    \kappa \mapsto [\kappa(\mfks_{0})\,\kappa(\mfks_{1})\dots].
    \]
    
    Since the map
    \begin{equation}\label{eq:loop:mfkg}
        \mfkg
        =
        \left(
            e^\infty;
            \calT_{m}([g|_{e^\infty}])
            , k ;
            e^\infty
        \right)
        \mapsto
        k 
    \end{equation}
    induces an isomorphism from
    $\cGss[\Z] / \ssym 
    $ to $\Z$,
    the action at~\eqref{eq:action,general} of $\cGss[\Z]/\ssym$ on $\ssymdual$ can all in all instead be understood as the following $\Z$-action on $\varprojlim (\T, \mvisiblespace^2)$:
    \begin{equation}
        \label{eq:loop:k HleftX dyadic}
        k \HleftX [z_{0} \, z_{1} \dots]
        =
        \begin{cases}
            [z_{0}^{2^{|k|}} \, z_{1}^{2^{|k|}} \dots]
            &\text{ if } k\leq 0
            \\
            [z_{k} \, z_{k+1}\dots]
            &\text{ if } k\geq 0
        \end{cases}
        ,
    \end{equation}
    and the corresponding transformation groupoid $\Z\ltimes \varprojlim (\T, \mvisiblespace^2)$ serves as a new groupoid model for the $C^*$-algebra $\csr(\cGss[\Z])$.

\subsection{One loop with a stem}

    Let $E^{0}=\{v_{0},v_{1}\}$ and $v_{0}Ev_{0}=\{e\}, v_{1}Ev_{0}=\{d\}$, so $E$ is a loop with a stem.

    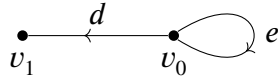
\begin{figure}[h!] 
    \centering
   
    \begin{tikzpicture}[ dot/.style={circle, fill=black, draw=none, inner sep=1.4pt}]
\def\LoopW{1.4cm}
\def\LoopH{0.9cm}


\node[dot, label=below:$v_0$] (v0) at (2,0) {};
\node[dot, label=below:$v_1$] (v1) at (0,0) {};


\draw[->-] (v0) -- node[above] {$d$} (v1);


\LoopRight{v0}{$e$}

\end{tikzpicture}
    
    \caption{A loop with a stem.}
    \label{fig:1 petal rose}
\end{figure}

We must assume that $\Z$ acts trivially on $E$. We let the restriction map be determined by 
    \[
    a|_{e}=a^2\quad\text{and}\quad a|_{d}=a^3.
    \]
    One can check this self-similar action is pseudo-free by noticing that $a^m \cdot \mu = \mu$ always and $a^m|_{de^n}$ and $a^m|_{e^n}$ are never identity unless $m=0$.
    There are two infinite paths, namely $e^\infty$ and $de^\infty$. Since $\Z$ acts trivially on $E$, we have for
    $x,y\in 
    E^\infty$ that
    \begin{align*}  
        (\cGss[\Z])_{y }^{x}
        =
        \left\{
        \left(
            x
            ;
            \calT_{
            m}[g|_{
            \sigma^{\ell }(y )}], 
            m
            -
            \ell 
            ;
            y  
        \right)
        : g\in \Z, m,\ell \in\N, \sigma^{m}(x)=\sigma^{\ell }(y )
        \right\}
        .
    \end{align*}
    Note that $m-\ell $ can be any integer by the nature of the two infinite paths $e^\infty, de^\infty$.

    As before, we let 
    \[
    \mfks_{x,n }\coloneqq (x;\calT_{n }[a|_{\sigma^n(x)}] ,0;x)
    \in \ssym(x).
    \]
    Since for $n\in\N$,
    \[
        [a|_{\sigma^n(x)}]
        =
        \calT_1 ([a|_{e^\infty}])^p
        \quad\text{ where }\quad
        p
        =
        \begin{cases}
            3 & \text{ if } x=de^\infty\text{ and } n=0,
            \\ 
            2 & \text{ otherwise} ,
        \end{cases}
    \]
    we see that for $\ell\in\N$,
    \[
        \left(
            y; \calT_{\ell}[a|_{\sigma^{n}(x)}], 0, y
        \right)
        =
        \begin{cases}
            \mfks_{y,1}^{p\cdot 2^{|\ell|}} & \text{ if } \ell\leq 0,\\
            \mfks_{y,\ell+1}^p & \text{ if } \ell\geq0 .
        \end{cases}
    \]
    Thus, for $n\geq 1$ (so that $p=2$),
    \[
        \mfkg\inv \mfks_{x,n} \mfkg
        \overset{\eqref{eq:conjugation}}{=}
        \left(
            y; \calT_{n-k}[a|_{\sigma^{n}(x)}], 0, y
        \right)
        =
        \begin{cases}
            \mfks_{y,1}^{2^{|n-k|+1}}& \text{ if } n-k\leq 0
            \\            
            \mfks_{y,n-k+1}^2 & \text{ if } n-k\geq 0
        \end{cases}
    \]   
    (The obnoxious $+1$ everywhere allows us to cover the cases $y=de^\infty$ and $y=e^\infty$ simultaneously.)
    We deduce that the action of \(
    \mfkg=\left(
            x            ; [f],  k  ;   y 
        \right)\)
    on $\kappa\in \ssymdual(y)$ is determined for $n\geq 1$ by
    \[
        \mfkg\HleftX \kappa
        \colon\quad
            \mfks_{x ,n }
            \mapsto
        \begin{cases}
            \kappa(\mfks_{y,1})^{2^{|n-k|+1}} & \text{ if }  n-k \leq 0 ,
            \\
            \kappa(\mfks_{y,n-k+1})^2 & \text{ if } n-k\geq 0 .
        \end{cases}  
    \]

    Since $r(e^\infty)=v_0\neq v_1 = r(de^\infty)$, we have $\Z_{v_{i}}=\Z$ for $i=0,1$; in particular, the generator $h_{x,n}$ of $\Z_{x(n,n)}$ for $x\in E^\infty$ and $n\in \N$ is always just equal to the generator $a$ of $\Z$. Thus \begin{equation}\label{eq:lollypop:ssym}
    \ssym(y )
    =
    \left\{
    \left(
    y 
    ;\calT_{
    \ell 
    }[h|_{
    \sigma^{\ell }(y )}],0;
    y 
    \right)
    \mid
    \ell \in\N 
    ,h\in G
    \right\}
    \end{equation}
    and the map
    \[
        \left(
            x
            ;
            [f], k ;
            y
        \right)
        \mapsto
        (x,k ,y )
    \]
    with inverse
    \begin{align*}    
        (x, k, y )
        \mapsto
        \begin{cases}
            (x; \calT_{1-k}[a|_{e^\infty}],k; y )\ssym
            &\text{ if } k\leq 0,
            \\
            (x; \calT_{1}[a|_{e^\infty}],k; y )\ssym
            &\text{ if } k\geq 0,
        \end{cases}
    \end{align*} 
    induces a bijection between $\cGss[\Z]/ \ssym$ and the Deaconu-Renault groupoid associated to the $\N$-action on $E^\infty$ induced by the shift map $\sigma$. 

    A quick computation shows that
    the integers $k_{y,n}$ from Lemma~\ref{lem:k_n} are given by
    \[
        k_{y,n}=
        \begin{cases}
            3 &\text{ if } y=de^\infty \text{ and } n=0,
            \\
            2 &\text{ otherwise}.
        \end{cases}
    \]
    As before, we use Theorem~\ref{thm:ssymdual} to identify $\ssymdual(y)$ with $\varprojlim (\T, \mvisiblespace^2)$  via 
    \[
    \ssymdual(y) \ni
    \kappa \mapsto [\kappa(\mfks_{y,1})\,\kappa(\mfks_{y,2})\dots].
    \]
    (Note that we have dropped $\mfks_{y,0}$ to avoid having to distinguish between $k_{de^\infty,0}$ and all other $k_{y,n}$, since the first map in the inverse limit can be forgotten.) This means that $\mfkg\HleftX \kappa$ is identified with
    \[
        \begin{cases}
            \left[
                \kappa(\mfks_{y,2-k})^2 \, \kappa(\mfks_{y,3-k})^2 \, \kappa(\mfks_{y,4-k})^2 \dots
            \right]
            &\text{ if } k\leq 1,
            \\            
            \left[
                \kappa(\mfks_{y,1})^{2^k} \, \kappa(\mfks_{y,1})^{2^{k-1}} \,
                \dots\, 
                \kappa(\mfks_{y,1})^{2}\,
                \kappa(\mfks_{y,2})^{2} \,
                \kappa(\mfks_{y,3})^{2} \dots
            \right]
            &\text{ if } k\geq 1.
        \end{cases}
    \]
    
    With the identification of $\cGss[\Z]/\ssym$ with the Deaconu-Renault groupoid and of $\ssymdual(y)$ with $\varprojlim (\T, \mvisiblespace^2)$, the action can now be described as
    \[
        (x,k,y)\HleftX [z_{1}\, z_{2}\, \dots]
        =
        \begin{cases}
            [z_{2-k}^{2} \, z_{3-k}^{2} \dots]
            &\text{ if } k\leq 1,
            \\
            \left[
                z_{1}^{2^k} \, z_{1}^{2^{k-1}} \,
                \dots\, 
                z_{1}^{2}\,
                z_{2}^{2} \,
                z_{3}^{2} \dots
            \right]
            &\text{ if } k\geq 1.
        \end{cases}            
    \]

\section{Alternative Candidate for a Cartan Subalgebra}
\label{sec:AnotherCandidate}

In this appendix, we identify another subgroupoid $\salt$ of $\Iso(\cGss)^{\circ}$ that yields a Cartan subalgebra $C^*(\salt) \subseteq C^*(\cGss[\Z])$ for certain self-similar
graphs $(\Z, E)$ not covered by Theorem~\ref{Thm_Ssym_Cartan}. 
To be more precise, we relax Condition~\ref{it:Thm_Ssym_Cartan:cycles} of Theorem~\ref{Thm_Ssym_Cartan} by allowing some cycles $C$ with no entrance to satisfy
$a^{p}|_{C} = a^{p}$,
where $p$ is the smallest positive integer such that $a^{p}$ fixes every vertex in $C$. Under this assumption, all Type~\ref{TypeII} elements at vertices in $C$ commute with Type~\ref{TypeI} elements at vertices in $C$ by Lemma~\ref{lem:TypeIITypeIDontCommute} (see Lemma~\ref{types_of_cycline_triples} and the subsequent remarks for different types). This motivates the definition of $\salt$ as 
an alternative candidate for a maximal abelian subgroupoid of $\Iso(\cGss)^{\circ}$. We will define $\salt$ to consist of $\ssym$
together with neighborhoods arising from certain Type~\ref{TypeII} cycline triples.
We state the main theorem of Appendix~\ref{sec:AnotherCandidate}.

\begin{thm} 
\label{Thm_Salt_Cartan} 
Let $(\Z, E)$ be a pseudo-free self-similar graph, and let $E$ 
have finitely many vertices and no sources. Assume additionally the following conditions 
hold:
\begin{enumerate}[label=\textup{(\arabic*)}] 
\item\label{it:Thm_Salt_Cartan:generic or rare} \label{AssumptionForClosedness} For any $x\in E^{\infty}$ and 
$g\in \Z$,
$x$ is either $g$-generic or $g$-rare. 
\item\label{it:Thm_Salt_Cartan:cycles}
Given any cycle $C$ in $E$ with no entrance and  $p$ the smallest positive integer such that $a^p$ fixes a vertex on $C$, we have $a^{p}|_{C} \neq a^{-p}$. Moreover, if $a^{p}|_{C} = a^{p}$, then the following holds:
        \begin{gather} \label{TypeII_cycle_hypothesis} 
        \tag{\ref{sec:AnotherCandidate}.
        $\dagger$} 
        \text{for any }
        1\leq m < |C|, \text{ there exists } \gamma \subseteq C \text{ with } |\gamma| = m \text{ and } a^{p}|_\gamma \neq a^{p}. 
        \end{gather} 
\end{enumerate}
Then the
subgroupoid
\begin{equation} 
\label{Salt:definition}  
\salt = \ssym \bigsqcup \left( 
\bigcup_{\substack{C \text{ satisfies}\\ 
a^p|_C = a^p 
}
} 
\bigcup_{\substack{(\alpha, g, \beta) \text{ Type~\ref{TypeII}} \\
\text{cycline triple at }s(\beta) \in C}} 
 \calZ(\alpha, g, \beta)\right)
\end{equation} 
of $\Iso(\cGss[\Z])^{\circ}$
satisfies the hypotheses of Theorem~\ref{thm:DWZ}.
Consequently, $C_r^*(\salt)$ is a Cartan subalgebra in $C_r^*(\cGss[\mathbb{Z}])$.
\end{thm}

For the rest of this appendix, we will fix a self-similar graph  $(\Z,E)$ and we let $\salt$ be the subgroupoid of $\cGss[\Z]$ defined in Theorem~\ref{Thm_Salt_Cartan}. For each of the following results,  we will explicitly name any other assumption  about $(\Z,E)$ from the theorem as needed.

The following quick lemma is a direct consequence of the definition of $\salt$ and our characterization of closed subgroupoids (Theorem~\ref{thm:closedness-equivalent}).

\begin{lem} 
$\salt$ is an open subgroupoid of $\Iso(\cGss[\Z])^{\circ}$. If $(\mathbb{Z}, E)$ is pseudo-free and satisfies Assumption~\ref{AssumptionForClosedness}, then $\salt$ is closed in $\cGss[\Z]$. 
\end{lem} 
\begin{proof}
    A standard (but tedious) argument shows that $\salt$ is a subgroupoid. Since $\salt$ is a union of open sets, it is open. Since $\ssym \subseteq \salt \subseteq \Iso(\cGss[\mathbb{Z}])^\circ$ and $(\mathbb{Z}, E)$ is pseudo-free,
    Assumption~\ref{AssumptionForClosedness} and Theorem~\ref{thm:closedness-equivalent} imply that~$\salt$ is closed.
\end{proof}

Since $\ssym$ is abelian by Corollary~\ref{cor:cycline comm}, to show that $\salt$ is abelian, it suffices to show that Type~\ref{TypeII} elements in $\salt$ pairwise commute and that they commute with all Type~\ref{TypeI} elements. 
We will show that this
follows from the assumption that Type~(II) elements in $\salt$ are based at cycles satisfying $a^p|_C=C$ (summed up in Theorem~\ref{thm:Smaxabelian}); to be a bit more precise,
the latter is guaranteed through an application of Lemma~\ref{lem:TypeIITypeIDontCommute}, and the former is shown in Lemma~\ref{lem:TypeIICommute} below. Assumption~\eqref{TypeII_cycle_hypothesis} ensures that no Type~\ref{TypeIII} element commutes with all elements of $\salt$ (see Lemma~\ref{lem:TypeIIITypeIDontCommute_Salt}), thus yielding maximal abelianness. 

As readers may observe, Theorem~\ref{Thm_Salt_Cartan}, though more general than Theorem~\ref{Thm_Ssym_Cartan}, still does not provide candidates for Cartan subalgebras in full generality. This reflects the inherently 
{\em ad hoc}
nature of identifying Cartan subalgebras for self-similar graph $C^*$-algebras. The aim of this appendix is to demonstrate that the analysis of cycline triples in Section~\ref{sec: Cartan subalgebras} nonetheless provides a systematic and flexible framework for finding Cartan subalgebras. From this viewpoint, $\salt$ arises naturally by adjoining those singleton elements that commute with $\ssym$ under the additional Assumption~\ref{it:Thm_Salt_Cartan:cycles}. Cases beyond the scope of Theorem~\ref{Thm_Salt_Cartan} involve complicated computations, while following a similar philosophy as Section~\ref{sec: Cartan subalgebras} and Appendix~\ref{sec:AnotherCandidate}. We therefore restrict to the setting of Theorem~\ref{Thm_Salt_Cartan} rather than pursue greater generality. 

As explained before Notation~\ref{Notation_cycle}, we will fix a cycle $C$ without entrance and a vertex $v$ on it and study locally the commutativity for elements of different types (note that later arguments do not depend on the choice of the fixed vertex).
We show first that composable Type~\ref{TypeII} elements on $C$ commute.

\begin{lem} \label{lem:TypeIICommute} 
Let $C, p, n$ be as described in Notation~\ref{Notation_cycle}. Suppose that $a^{p}|_{C} = a^{p}$. If $\mfks$ is a Type~\ref{TypeII} element at a given vertex on $C$, then it commutes with any composable Type~\ref{TypeII} element at a vertex on $C$. 
\end{lem} 
\begin{proof} 
    Since $\mfks$ is a Type~\ref{TypeII} element, we may without loss of generality assume $\mfks \in \calZ(\nu, a^{pq}, \nu    C^k   )$ for some $\nu \in E^*, q \in \mathbb{Z}, k \in \mathbb{N}$.  If $\mfkt$ is a Type~\ref{TypeII} element at    any vertex    on $C$    and is composable with $\mfks$, then $\mfkt \in \calZ(\alpha, a^{pq'},\beta)$ for some $\alpha, \beta \in E^*$ such that $r(\alpha) = r(\beta) = r(\nu)$ and $s(\alpha) = s(\beta) \in C  $, and $q' \in \mathbb{Z}$.
Without loss of generality, as $\mfkt$ is Type~\ref{TypeII},
we may assume $|\alpha| < |\beta|$ so that $|\beta| = |\alpha| + n k'$ for some $k' \in \mathbb{N}$, where $n=|C|$. By Lemma~\ref{lem:comm_rela}, we have that $\mfks \mfkt = \mfkt \mfks$ if and only if for $x=C^\infty$, the unique element in $s(\nu C^k)E^\infty$, and for $y$ the unique element of $s(\alpha)E^\infty$, 
the following equality holds
    \begin{equation} \label{condforcomm}
    [a^{pq}|_{x(0,|\alpha|)}|_{\sigma^{|\alpha|}(x)}  \ a^{pq'}|_{y(0,|\nu|+nk)}|_{\sigma^{|\nu|+nk}(y)}] = [a^{pq}|_{x(0,|\alpha|+nk')}|_{\sigma^{|\alpha|+nk'}(x)}  \ a^{pq'}|_{y(0,|\nu|)}|_{\sigma^{|\nu|}(y)}].
    \end{equation}
    We compute the first factor of the left-hand side: for $m \in \mathbb{N}$,
    \begin{equation}
    a^{pq}|_{x(0,|\alpha|)}|_{\sigma^{|\alpha|}(x)}(m) = a^{pq}|_{x(0,m+|\alpha|)}.
    \end{equation}
We compute the second factor of the left-hand side, using that $a^p|_C=a^p$ and that $y(0,nk)=\sigma^{\ell}(C^\infty)(0,nk)$ (some $\ell\in\N$) is `a rotated copy of $C^k$' 
for the equations marked $(\dagger)$: for $m\in \N$, 
    \begin{align}
    a^{pq'}|_{y(0,|\nu|+nk)}|_{\sigma^{|\nu|+nk}(y)}(m) & = a^{pq'}|_{y(0,|\nu|+nk+m)}  
    = a^{pq'}|_{y(0,nk)}|_{y(nk,|\nu|+nk+m)}
    \\
    &\overset{(\dagger)}= a^{pq'}|_{y(nk,m+|\nu|+nk)}\overset{(\dagger)}= a^{pq'}|_{y(0,m+|\nu|)}.
    \end{align}

Thus, the left-hand side is the equivalence class in $\mathcal{Q}(\Z,\N)$ of the function
\[ m\mapsto a^{pq}|_{x(0,m+|\alpha|)} \ 
a^{pq'}|_{y(0,m+|\nu|)}
\]
Next, we compute the first factor of the right-hand side, using that $a^p|_{C^k}=a^p$ and that $x(0,nk')=C^{k'}$ for the equations marked $(\ast )$:
\begin{align}
a^{pq}|_{x(0,|\alpha|+nk')}|_{\sigma^{|\alpha|+nk'}(x)}(m) &= a^{pq}|_{x(0,|\alpha|+nk'+m)}
    = a^{pq}|_{x(0,nk')}|_{x(nk',|\alpha|+nk'+m)}
    \\
    &\overset{(\ast )}= a^{pq}|_{x(nk',m+|\alpha|+nk')} \overset{(\ast )}= a^{pq}|_{x(0,m+|\alpha|)},
\end{align} 
Lastly, we compute the second factor of the right-hand side:
\begin{equation}
    a^{pq'}|_{y(0,|\nu|)}|_{\sigma^{|\nu|}(y)}(m)  = a^{pq'}|_{y(0,m+|\nu|)},
\end{equation}
which shows that the right-hand side is the equivalence class in $\mathcal{Q}(\Z,\N)$ of the same function as the left-hand side.
Thus equality holds in Equation \eqref{condforcomm}, so $\mfks$ and $\mfkt$ commute. 
\end{proof}

\begin{lem} \label{lem:TypeIIITypeIDontCommute_Salt}  
Let $C, p,n$ be as described in Notation~\ref{Notation_cycle}, and assume that $(\Z,E)$ is pseudo-free. Suppose $(\mu, g, \nu)$ is a Type~\ref{TypeIII} cycline triple at a given vertex $v$ on $C$, and $\mfkg \in \calZ(\mu, g, \nu)$. If
$C$ satisfies $a^p|_C = a^p$ and~\eqref{TypeII_cycle_hypothesis}, then there exists an element of Type~\ref{TypeI} that does not commute with $\mfkg$. 
\end{lem} 

\begin{proof} 
Since $(\mu, g, \nu)$ is of Type~\ref{TypeIII}, we have
$g = a^{q}$ for some $q \in \Z  \setminus p\Z$,  $s(\nu) =  v  \neq s(\mu)$, $r(\mu) = r(\nu)$, and $a^q \cdot  v  = s(\mu)$. Take $\mfks$ to be the unique element of the singleton $\calZ(\nu, a^{p}, \nu)$, so that $\mfks$ is of Type~\ref{TypeI}.
As explained at Equation~\eqref{III_restrict_l} in the proof of Lemma~\ref{lem:TypeIIITypeIDontCommute},
we have $\mfkg \mfks \neq\mfks \mfkg$ if and only if 
\begin{equation} 
\label{III_restrict_l_Salt} 
a^p|_{C^\infty(0,|\nu|+\ell)} 
\neq
a^p|_{C^\infty(0,|\mu|+\ell)}
\end{equation} 
for infinitely many $\ell\in\N$.
Note that our assumption that $a^p|_C = a^p$ implies for any $t\in\N$ that
\[
a^p|_{C^\infty(0,t+n)}
=
a^p|_{C}|_{C^\infty(0,t)}
=
a^p|_{C^\infty(0,t)}
\]
and so if the inequality at~\eqref{III_restrict_l_Salt} holds for {\em one} $\ell$, then it automatically holds for {\em infinitely many}.

Since $(\mu, g, \nu)
$ is a non-symmetric cycline triple, we may assume $|\mu|>|\nu|$ since the proof when $|\mu|<|\nu|$ is similar. Since $(\mu,g, \nu)$ is of Type~\ref{TypeIII}, this assumption means that $|\mu| = |\nu| +kn + m$ for some $k \in \mathbb{N}$ and $m \in \{1,2, \dots, n-1\}$.
By Lemma~\ref{compute_cycle}~\ref{restrict_i}  and~\ref{restrict_power_i} 
and an inductive argument,
for each $\ell\in \N$,
there exists $Q_{\ell}\in \mathbb{Z} $ such that
\begin{equation} 
a^{p}|_{C^\infty(0, |\nu|+\ell)}  
=  a^{pQ_{\ell}} ,
\end{equation} 
and since $(\mathbb{Z},E)$ is pseudo-free, the numbers $Q_{\ell}$ are all non-zero.
Using this in the first step of the next computation and that $(|\mu|+\ell)-(|\nu|+nk+\ell) = m >0$ in the second step, the right-hand side of \eqref{III_restrict_l_Salt} now becomes 
 \begin{align*} 
 a^{p}|_{C^\infty(0, |\mu|+\ell)}
&= a^{pQ_\ell}|_{C^\infty(|\nu|+\ell, |\mu|+\ell)}= a^{pQ_\ell}|_{C^\infty(|\nu|+\ell, |\nu| + nk+\ell)}|_{C^\infty(|\nu|+nk+\ell, |\mu|+\ell)}.
\intertext{Since $|C|=n$, we can iteratively apply Lemma~\ref{compute_cycle}~\ref{restrict_power_i}  and~\ref{restrict_cycle_other_start} and our assumption that 
$a^p|_C = a^p$
to conclude}
    a^{p}|_{C^\infty(0, |\mu|+\ell)}&=a^{pQ_\ell}|_{C^\infty(|\nu|+nk+\ell, |\mu|+\ell)}=(a^{p}|_{C^\infty(|\nu|+nk+\ell, |\mu|+\ell)})^{Q_\ell}.
\end{align*}

Note that the subscript $C^\infty(|\nu|+nk+\ell, |\mu|+\ell)$ 
is a subpath of $C$ of length $m$
and that any subpath $\gamma\subset C$ of length $m$ 
is of this form
for some $\ell\in\N$. 
Thus, since we assumed that $C$ satisfies~\eqref{TypeII_cycle_hypothesis}, there
exists $\ell$ 
such that 
\[
a^p|_{C^\infty(|\nu|+nk+\ell, |\mu|+\ell)}\neq a^p.
\]
Since $Q_\ell\neq 0$ this proves thatthe right-hand side
of~\eqref{III_restrict_l_Salt}
does not equal $a^{pQ_{\ell}}$,  which is the left-hand side of~\eqref{III_restrict_l_Salt}, and the proof is finished.
\end{proof} 

Next we prove the maximal abelianness of $\salt$. The hypothesis of pseudo-freeness is used to apply Lemma~\ref{lem:TypeIIITypeIDontCommute_Salt} and Lemma~\ref{lem:TypeIITypeIDontCommute}. 

\begin{thm} \label{thm:Smaxabelian} 
Let $(\Z, E)$ be a pseudo-free self-similar graph, and let $E$ have finitely many vertices and no sources. Suppose additionally that Assumption~\ref{it:Thm_Salt_Cartan:cycles} of Theorem~\ref{Thm_Salt_Cartan} holds. Then $\salt$ is maximal among open abelian subgroupoids of $\Iso(\cGss[\Z])^\circ$. 
\end{thm} 

\begin{proof}
    First we prove $\salt$ is abelian. By Lemma~\ref{cor:cycline comm}, we know all elements belonging to symmetric cycline cylinder sets commute with each other. By Lemma~\ref{lem:TypeIITypeIDontCommute} and Assumption~\ref{it:Thm_Salt_Cartan:cycles} of Theorem~\ref{Thm_Salt_Cartan} all Type~\ref{TypeII} elements in $\salt$ commute with all Type~\ref{TypeI} elements. Lastly, Lemma~\ref{lem:TypeIICommute} guarantees that Type~\ref{TypeII} elements of $\salt$ commute with each other.

    Now we prove that $\salt$ is maximal abelian. Lemma~\ref{lem:TypeIIITypeIDontCommute_Salt} shows that no Type~\ref{TypeIII} elements commute with all Type~\ref{TypeI} elements. Lastly, any Type~\ref{TypeII} element $\mfkg$ not in $\salt$ must be based at a vertex on a cycle $C$ such that $a^p|_C \neq a^{\pm p}.$ Lemma~\ref{lem:TypeIITypeIDontCommute} then proves that $\mfkg$ does not commute with all Type~\ref{TypeI} elements. 
\end{proof}

The following lemma will be useful in the proof that $\salt$ is normal.

\begin{lem} \label{lem:inductionstep}
    Let $C, p,n$ be as described in Notation~\ref{Notation_cycle}. If $h \in \mathbb{Z}$ and $k \in \mathbb{N}$,
    \sloppy then there exists $t \in \mathbb{Z}$ such that $(h|_{C(0,\ell)})^{-1}(h|_{C(0,kn+\ell)})
    = a^{pt}|_{C(0,\ell)}$ for all $\ell \in \mathbb{N}$.
\end{lem}

\begin{proof}
    It
    follows from~\cite[Lemma 3.5(ii)]{Self-similar:LY21} that  
    \begin{equation*}
        (
        h^{-1} h|_{C^k}
        ) \cdot s(C^k) = (h^{-1}) \cdot (h|_{C^k} \cdot s(C^k)) = (h^{-1}) \cdot (h \cdot s(C^k)).
    \end{equation*}
    As the right-hand side equals $s(C^k) = s(C)$, this proves that $h^{-1} h|_{C^k}$ fixes a vertex on $C$ and must therefore be of the form $a^{pt}$ for some $t \in \mathbb{Z}$. Now fix $\ell \in \mathbb{N}$. Note that $C(0,\ell) = a^{pt} \cdot C(0,\ell) = (h^{-1} h|_{C^k}) \cdot C(0,\ell)$ implies $h \cdot C(0,\ell) = h|_{C^k} \cdot C(0,\ell)$. This justifies the equation marked $(\dagger)$ in the following computation, while the first step follows from \cite[Lemma 3.5(iii)]{Self-similar:LY21}:
        \begin{align*}
        (h|_{C(0,\ell)})^{-1}(h|_{C(0,kn+\ell)}) 
        &= h^{-1}|_{h \cdot C(0,\ell)} \ h|_{C^k}|_{C(0,\ell)} \\
        &\overset{(\dagger)}=  h^{-1}|_{h|_{C^k} \cdot C(0,\ell)} \ h|_{C^k}|_{C(0,\ell)} \\
        &= (h^{-1}h|_{C^k})|_{C(0,\ell)} \\
        &= a^{pt}|_{C(0,\ell)}. \qedhere
        \end{align*} 
\end{proof}

\begin{lem}\label{lem:Salt is normal}
Suppose that Assumption~\ref{it:Thm_Salt_Cartan:cycles} of Theorem~\ref{Thm_Salt_Cartan} holds. Then $\salt$ is normal in $\cGss[\mathbb{Z}]$. 
\end{lem} 

\begin{proof}
Suppose $\mfkt \in \salt$ and $\mfkg \in \cGss[\mathbb{Z}]$ are composable. Since $\mfkt \in \salt$, there are two cases:
the first is that $\mfkt \in \calZ( \alpha, h, \beta)$ for some symmetric cycline triple $(\alpha, h, \beta)$, so that $\mfkg\mfkt \mfkg^{-1} \in \ssym \subset \salt$ by Lemma~\ref{lem:Ssym is normal}, and we are done. The second case is that $\mfkt \in \calZ( \alpha, h, \beta)$ for a non-symmetric triple $( \alpha, h, \beta)$ of Type~\ref{TypeII} such that the cycle $C$ with $r(C)=s(\beta)$ has no entrance and satisfies where, as always, $p$ is the smallest positive integer such that $a^p$ fixes a vertex on $C$. By Assumption~\ref{it:Thm_Salt_Cartan:cycles}, $C$ automatically also satisfies Condition \eqref{TypeII_cycle_hypothesis}. Being of Type~\ref{TypeII} further means that $(\alpha,h,\beta)$ is either of the form $(\beta C^k, a^{pq},\beta)$ or $(\alpha, a^{pq},\alpha C^k)$ where $k\in \N^+$ and $q\in\Z$.
Without loss of generality, assume $\alpha = \beta C^k$, and let $n$ be the length of $C$.
Write
\begin{align*} 
\mfkt  &= (\beta C^\infty; \ \calT_{|\beta|+nk}([a^{pq}|_{C^\infty}]), \ nk; \ \beta C^\infty)
&&\text{and}
\\
\mfkg  &= (\mu (g\cdot x);  \ \calT_{|\mu|}([g|_{x}]), \ |\mu|-|\nu|; \ \nu x)
\end{align*}
for some $\mu, \nu \in E^*$, $g\in \mathbb{Z}$, and $x\in s(\nu)E^{\infty}$ which satisfy $s(\mu) = g \cdot s(\nu)$ and $\nu x = \beta C^\infty$.
Then a standard computation shows
\begin{align*}
    &\mfkg \mfkt \mfkg^{-1} 
    = 
    \left(
        \mu (g \cdot x);
        \mathcal{T}_{|\mu|}\left(
            [g|_x]
            \
            \mathcal{T}_{|\beta|-|\nu|+nk}
        \left([a^{pq}|_{C^\infty}]\right) 
            \
            \mathcal{T}_{nk}\left([g|_x]^{-1}\right)
        \right), nk; \mu (g \cdot x)
    \right).
\end{align*}

Now we focus our attention on the second component. Since $\nu x =\beta C^\infty$, 
we may write $x=\nu'C^\infty$ where $|\nu'| = |\beta|-|\nu|+nm$ for some $m \in \mathbb{N}$.
Thus, by \cite[Lemma 5.1]{Self-similar:LY21} and since $\mathbb{Z}$ is abelian, we 
can rewrite the function at which the shift-map $\calT_{|\mu|}$ is evaluated: 
\begin{align*}
&[g|_x]\ \mathcal{T}_{|\beta|-|\nu|+nk}([a^{pq}|_{C^\infty}]) \ \mathcal{T}_{nk}([g|_x]^{-1})\\
&= \mathcal{T}_{|\nu'|}
    \left(
        [(g|_{\nu'})|_{C^\infty}] 
        \ 
        \mathcal{T}_{nk}
        \left(
            [(g|_{\nu'})|_{C^\infty}]^{-1}
        \right)
    \right)
    \
    \mathcal{T}_{|\beta|-|\nu|+nk}([a^{pq}|_{C^\infty}]).
\end{align*}

Let us zoom in further and focus on the function at which the shift-map $\mathcal{T}_{|\nu'|}$ is evaluated. For $M\geq nk$, we have
\begin{align*} 
(g|_{\nu'})|_{C^\infty}
\ 
\mathcal{T}_{nk}\left((g|_{\nu'})|_{C^\infty}^{-1}\right) \colon \quad M\mapsto
    {} &  g|_{\nu'}|_{C(0,M)}\, (g|_{\nu'}|_{C(0,M-nk)})^{-1}.
\end{align*}
By Lemma~\ref{lem:inductionstep}
applied to $h:= g|_{\nu'}$ and $\ell:=M-nk\geq 0$,
there exists $t \in \mathbb{Z}$ such that 
\[(g|_{\nu'}|_{C(0,M-nk)})^{-1} \, g|_{\nu'}|_{C(0,M)}
= a^{pt}|_{C(0,M-nk)}.\]
Since $\Z$ is abelian, we have all in all
shown that 
\[
[(g|_{\nu'})|_{C^\infty}]
\ 
\mathcal{T}_{nk}\left([(g|_{\nu'})|_{C^\infty}^{-1}]\right)
=
\mathcal{T}_{nk}([a^{pt}|_{C^\infty}]).
\]
The second component of $\mfkg \mfkt \mfkg^{-1}$ simplifies to  
\begin{align*} \label{eqn:secondcomponent}
 &
 \mathcal{T}_{|\mu|}\left(\mathcal{T}_{|\nu'| +nk}([a^{pt}|_{C^\infty}]) \mathcal{T}_{|\beta|-|\nu|+nk}([a^{pq}|_{C^\infty}])\right) \\
 &= \mathcal{T}_{|\mu|+|\nu'|+nk}\left([a^{pt}|_{C^\infty}] \mathcal{T}_{-nm}([a^{pq}|_{C^\infty}])\right) 
 &&\text{ as } |\nu'| = |\beta|-|\nu|+nm \\
 &= \mathcal{T}_{|\mu|+|\nu'|+nk}\left([a^{pt}|_{C^\infty}] [a^{pq}|_{C^\infty}]\right)&&\text{ because } a^p|_C = a^p \text{ and } |C|=n\\
 &= \mathcal{T}_{|\mu|+|\nu'|+nk}([a^{p(t+q)}|_{C^\infty}]) && \text{ as } a^{pq} \cdot C = C
 \end{align*}
 Since $\Z$ is abelian and $a^p$ fixes all edges on $C$, it is easy to verify that
 \begin{equation}
     a^{p(t+q)}|_{C^\infty} = a^{p(t+q)}|_{g|_{\nu'C^k} \cdot C^\infty}.
 \end{equation}
So all in all,
\begin{align}\label{eq:2ndcompfinal,v2}
&\mfkg \mfkt \mfkg^{-1} 
= \left(
\mu (g \cdot x); \mathcal{T}_{|\mu|+|\nu'|+nk}([a^{p(t+q)}|_{g|_{\nu' C^k} \cdot C^\infty}]), nk; \mu (g \cdot x)
\right).
\end{align}

By Lemma~\ref{lem:cycles to cycles}\ref{it:cycles to cycles:length} and~\ref{it:cycles to cycles:self-similarity}, $\tilde{C}:=g|_{\nu'} \cdot C$ is a cycle that satisfies
\begin{align}
    \label{eq:tilde-C infty,1}
    \tilde{C}^\infty= 
    (g|_{\nu'} \cdot C)^\infty
    &
    =
    g|_{\nu'C^k} \cdot C^\infty
    &&\text{ by Lemma~\ref{lem:cycles to cycles}\ref{it:cycles to cycles:self-similarity}}
    \\
    \label{eq:tilde-C infty,2}
    &=
    g|_{\nu'} \cdot C^\infty.
\end{align}
Thus, using that $x=\nu' C^\infty$
and that $a^{p(t+q)}$ fixes $C$
and hence $\tilde{C}$,
we can rewrite the range (and source) of $\mfkg \mfkt \mfkg^{-1}$ in two ways: on the one hand,
\begin{align*}
    \mu (g \cdot x) & = \mu(g \cdot \nu')(g|_{\nu'} \cdot C^k)(g|_{\nu'C^k} \cdot C^\infty)
    \overset{\eqref{eq:tilde-C infty,1}}=
    \mu(g \cdot \nu')(g|_{\nu'} \cdot C^k)(a^{p(t+q)} \cdot \tilde{C}^\infty)
    \intertext{and on the other hand,}
    \mu (g \cdot x) & = 
    \mu(g \cdot \nu')(g|_{\nu'} \cdot C^\infty)
    \overset{\eqref{eq:tilde-C infty,2}}{=}
    \mu(g \cdot \nu')\tilde{C}^\infty,
\end{align*}
proving that 
\[\mfkg \mfkt \mfkg^{-1} \in \calZ\left(\mu(g \cdot \nu')(g|_{\nu'} \cdot C^k), a^{p(t+q)}, \mu(g \cdot \nu')\right).\]
If we can show that 
\begin{align}\label{eq: the alleged cycline triple}
\left(\mu(g \cdot \nu')(g|_{\nu'} \cdot C^k), a^{p(t+q)}, \mu(g \cdot \nu')\right)
\end{align}
is a cycline triple of Type~\ref{TypeII} and that the associated cycle $\tilde{C}$ satisfies $a^{\tilde{p}}|_{\tilde{C}} = a^{\tilde{p}}$ (where $\tilde{p}$ is the smallest positive integer such that $a^{\tilde{p}} \cdot \tilde{C} = \tilde{C}$), then we can conclude that the associated cylinder set is contained in $\salt$, which finishes the proof that  $\mfkg \mfkt \mfkg^{-1}\in\salt$. 

Since $\tilde{C}$ is a cycle with no entrance (Lemma~\ref{lem:cycles to cycles}\ref{it:cycles to cycles:entrance}), we have  $s( g \cdot \nu')E^\infty = \{\tilde{C}^\infty\}$, so $\tilde{C}$ is the only infinite path we need to consider to check that the triple at~\eqref{eq: the alleged cycline triple} is cycline: Since $\Z$ is abelian and $a^p$ fixes $C$, we have
$a^{p(t+q)} \cdot \tilde{C}^\infty = \tilde{C}^\infty$, which implies
\begin{align*}
    \mu(g \cdot \nu')(g|_{\nu'} \cdot C^k)(a^{p(t+q)} \cdot \tilde{C}^\infty) 
    &= \mu(g \cdot \nu')(g|_{\nu'} \cdot C^k)\tilde{C}^\infty ,
\end{align*}
as needed. The triple is furthermore of Type~\ref{TypeII} because $g|_{\nu'} \cdot C^k = \tilde{C}^k$ as a consequence of~\eqref{eq:tilde-C infty,2}.

Next, we will show that $a^{\tilde{p}}|_{\tilde{C}} = a^{\tilde{p}}.$ Note that any $h\in \mathbb{Z}$ fixes $C$ if and only if it fixes $\tilde{C}$, for
\[
h\cdot C = C 
\iff 
g|_{\nu'}\cdot (h\cdot C) = g|_{\nu'}\cdot C 
\iff
h\cdot \tilde{C} = \tilde{C}, 
\]
where the last equivalence follows from $\Z$ being abelian. Thus $\tilde{p}=p$, the smallest positive integer such that $a^p\cdot C=C$. Then 
\begin{align*}
    (a^p|_{g|_{\nu'} \cdot C})(g|_{\nu'}|_C) = (a^p g|_{\nu'})|_C = ( g|_{\nu'}a^p)|_C
    &= (g|_{\nu'}|_C) (a^p|_C) && \text{ as } a^p \cdot C = C \\
    &=(g|_{\nu'}|_C) a^p && \text{ as } a^p|_C = a^p. 
\end{align*}
 Canceling $g|_{\nu'}|_C$ on both sides, we see that $a^p = a^p|_{g|_{\nu'} \cdot C} = a^p|_{\tilde{C}}$, as desired. 
\end{proof}

To prove the claim in Theorem~\ref{Thm_Salt_Cartan} that $\salt$ satisfies all hypotheses of Theorem~\ref{thm:DWZ}, it only remains to verify:

\begin{lem}
Suppose that Assumption~\ref{it:Thm_Salt_Cartan:cycles} of Theorem~\ref{Thm_Salt_Cartan} holds. Then $\salt$ satisfies Equation~\eqref{eq:ricc,stronger} for all $ \mfkg\in \Iso(\cGss[\mathbb{Z}])^\circ$.
\end{lem}

\begin{proof}
    We must show that for all $\mfkg \in \Iso(\cGss[\mathbb{Z}])^\circ$, the set 
    \[\mathrm{Ad}(\mfkg):= \{\mfks^{-1} \mfkg \mfks: \mfks \in \salt(s(\mfkg))\}
    \]
    is either $\{\mfkg\}$ or infinite. Fix $\mfkg \in \Iso(\cGss[\mathbb{Z}])^\circ$, so that $\mfkg \in \calZ( \mu, a^t, \nu)$ for some cycline triple $(\mu, a^t, \nu)$. If $\mfkg \in \salt$, then $\mathrm{Ad}(\mfkg)=\{\mfkg\}$ because $\salt$ is abelian by Theorem~\ref{thm:Smaxabelian}.
    
    So assume now that $\mfkg \not\in \salt$, so that $(\mu, a^t, \nu)$ is of Type~\ref{TypeII} or Type~\ref{TypeIII}, and $s(\nu)$ lies on a cycle $C$ without entrance (Corollary~\ref{conclusion}). Let $p$ be the smallest positive integer such that $a^p$ fixes a vertex on (and hence all of) $C$. By Assumption~\ref{it:Thm_Salt_Cartan:cycles},  we have two cases: The first is $a^p|_C \neq a^{\pm p}$, in which case the proof of Lemma~\ref{lem:Ssym (almost) immediately centralizing} shows that $\mathrm{Ad}(\mfkg)$ is either $\{\mfkg\}$ or infinite. The second is 
    $a^p|_C = a^{p}$, and  $C$ satisfies~\eqref{TypeII_cycle_hypothesis}. In this case, since $\mfkg\notin \salt$, $(\mu, a^t, \nu)$ must be of Type~\ref{TypeIII} by definition of $\salt$, so $t \in \mathbb{Z} \setminus p\mathbb{Z}$. Since the sets $\mathrm{Ad}(\mfkg)$ and  $\mathrm{Ad}(\mfkg\inv)$ have the same cardinality, we may replace $\mfkg$ with $\mfkg\inv$ and can therefore without loss of generality assume that $\mfkg$ is of the form
    \[\mfkg =\left(\mu (a^t \cdot C^\infty); \mathcal{T}_{|\mu|}([a^t|_{C^\infty}]), |\mu| - |\nu|; \nu C^\infty\right).\]
    For $q \in \mathbb{Z}$, define 
    \[ \mfks_q  = \left(\nu(a^{pq} \cdot C^\infty);\calT_{|\nu|}([a^{pq}|_{ C^\infty}]), \ 0; \ \nu C^\infty\right). \]
    Then $\mfks_q \in \salt(s(\mfkg))$ for all $q \in \mathbb{Z}$, and we claim that
    $\{\mfks_q^{-1} \mfkg \mfks_q \mid q \in \mathbb{Z}\}$ is infinite, proving that $\mathrm{Ad}(\mfkg)$ is infinite. For $q \in \mathbb{Z}$,  
\begin{equation*}
    \mfks_q^{-1} \mfkg \mfks_q 
    = \left(\nu  C^\infty;  \calT_{|\nu|}([a^{pq}|_{C^\infty}]^{-1})  \mathcal{T}_{|\mu|}([a^t|_{C^\infty}]) \calT_{|\mu|}([a^{pq}|_{C^\infty}]), |\mu|-|\nu|; \nu C^\infty\right).
    \end{equation*}
Thus, it suffices to show that the subset
\begin{equation}\label{eq:infinite set}
\left\{\calT_{|\nu|}([a^{pq}|_{C^\infty}]^{-1})   \calT_{|\mu|}([a^{pq}|_{C^\infty}]) \mid q \in \mathbb{Z}\right\}
\end{equation}
of $\mathcal{Q}(\N,\Z)$ contains infinitely many elements. 

Since $(\mu, a^t, \nu)$ is of Type \ref{TypeIII}, we know $|\mu| -  |\nu| \not\in |C|\mathbb{Z}$ by Lemma~\ref{types_of_cycline_triples}. We may thus let $r\in\{1,\dots,|C|-1\}$ be such that $|\mu| - |\nu| = r \mod |C|$. Since $C$ satisfies \eqref{TypeII_cycle_hypothesis}, there exists $N
\in\N
$ such that $a^p|_{C^\infty
(N,N+ r )} \neq a^p$. Since $C^\infty
(N,N+ r ) =  C^\infty
(|C|+N,|C|+N+ r )$, we can without loss of generality assume that $N> |\mu|$. Since the number $R$ in Lemma~\ref{compute_cycle}~\ref{restrict_cycle} is equal to $1$ by our assumption that $a^p|_C=a^p$, each of the integers $r_{1},\dots,r_{n}$ in Lemma~\ref{compute_cycle}~\ref{restrict_i} must be either $1$ or $-1$. As  $a^p|_{C^\infty
(N,N+r)} \neq a^p$, another application of  Lemma~\ref{compute_cycle} thus shows that $a^p|_{C^\infty
(N,N+ r )} = a^{-p}$. By 
Lemma~\ref{compute_cycle}~\ref{restrict_i} and~\ref{restrict_power_i} and since we assumed $(\mathbb{Z},E)$ to be pseudo-free,
there exists $k \in \mathbb{Z}\setminus\{0\}$ such that $a^{pq}|_{
C^\infty
(0,N)} = a^{pqk}$.

Fix $m \geq N+r+|\nu|+|\mu|$.
By Lemma~\ref{compute_cycle}~\ref{restrict_power_i}, we have \(a^{p\ell}|_{\alpha}  = (a^{p}|_{\alpha})^{\ell}\) for any $\ell \in\Z$ and finite subpath $\alpha \subset C^\infty$; in particular,
\(a^{p\ell}|_{C^\infty}  = (a^{p}|_{C^\infty})^{\ell}\). These observations explain the equations marked with $(\ast )$ in the following computations:
\begin{align*}
\calT_{|\nu|}( a^{pq}|_{C^\infty}^{-1}) (m )
&\overset{(\ast )}{=} a^{-pq}|_{C^\infty (0,m -|\nu|)} \\
&\overset{(\ast )}{=} 
(a^{pq}|_{C^\infty (0,N)})\inv|_{C^\infty (N,N+ r )}|_{C^\infty (N+ r , m -|\nu|)}\\
&=
a^{-pqk}|_{C^\infty (N,N+ r )}|_{C^\infty (N+ r , m -|\nu|)}
&\text{by choice of $k$}\\
&\overset{(\ast )}{=}
(a^{p}|_{C^\infty (N,N+ r )})^{-qk}|_{C^\infty (N+ r , m -|\nu|)}\\
&= a^{pqk}|_{C^\infty (N+ r ,  m -|\nu|)}&\text{by choice of $N$}.
\end{align*}
Similarly,
\begin{align*}
\calT_{|\mu|}( a^{pq}|_{C^\infty} )(m)
&= a^{pq}|_{C^\infty (0,N)}|_{C^\infty (N,m-|\mu|)} = a^{pqk}|_{C^\infty (N,m-|\mu|)},
\end{align*}
so that all in all
\begin{align*}
 \left(\calT_{|\nu|}( a^{pq}|_{C^\infty}^{-1}) \calT_{|\mu|}( a^{pq}|_{C^\infty} )\right)(m)  &=\left(a^{pqk}|_{C^\infty (N+ r ,  m -|\nu|)}\right)\left(a^{pqk}|_{C^\infty (N,m-|\mu|)}\right)
 \\
     &\overset{(\ast )}{=}\left(a^{p}|_{C^\infty (N+ r ,  m -|\nu|)} \, a^{p}|_{C^\infty (N,m-|\mu|)}\right)^{qk}.
\end{align*}
We now choose $m= N+|\mu|+\ell|C|$ for $\ell\in \N$ large enough (i.e., such that $|C|\ell \geq r+|\nu|$):
Since the number $R$ in Lemma~\ref{compute_cycle}~\ref{restrict_cycle} is equal to $1$ by our assumption that $a^p|_C=a^p$, the above equality combined with Lemma~\ref{compute_cycle}~\ref{restrict_cycle_other_start} implies
\begin{align*}
 \left(\calT_{|\nu|}( a^{pq}|_{C^\infty}^{-1}) \calT_{|\mu|}( a^{pq}|_{C^\infty} )\right)(N+|\mu|+\ell|C|)
 = a^{2pqk}.
\end{align*}
\sloppy 
Since $k\neq 0$,
this proves  that any two functions $\calT_{|\nu|}( a^{pq}|_{C^\infty}^{-1}) \calT_{|\mu|}( a^{pq}|_{C^\infty} )$ and $\calT_{|\nu|}( a^{pq'}|_{C^\infty}^{-1}) \calT_{|\mu|}( a^{pq'}|_{C^\infty} )$ for $q\neq q'$ in $\Z$ differ at infinitely many points $m\in\N$, meaning that they do not give rise to the same equivalence class in $\mathcal{Q}(\N,\Z)$. Thus, the set at~\eqref{eq:infinite set} is infinite, as needed.
\end{proof}

\bibliographystyle{plain}

\begin{thebibliography}{99} 

\bibitem{BarlakLi}S.\ Barlak and X.\ Li, ``Cartan subalgebras and the UCT problem,'' \emph{Adv. Math.} {\bf 316} (2017), 748--769.

\bibitem{BKQ:2017}E.\ B\'edos, S.\ Kaliszewski and J.\ Quigg, ``On Exel-Pardo algebras,'' \emph{J. Operator Theory} {\bf 78} (2017), no.~2, 309--345; MR3725509

\bibitem{BCF25} J.\ Brown, L.\ Clark and A.\ Fuller, ``Intermediate subalgebras of Cartan embeddings in rings and $C^*$-algebras", {\em Advances in Mathematics}
\textbf{481}, 2025, 110534. 

\bibitem{BEFPR:2021}J.\ Brown, R.\ Exel, A.\ Fuller, D.\ Pitts and S.\ Reznikoff, ``Intermediate $C^*$-algebras of Cartan embeddings,'' \emph{Proc. Amer. Math. Soc. Ser. B} {\bf 8} (2021), 27--41.


\bibitem{BNRSW:Cartan} J.\ Brown, G.\ Nagy, S.\ Reznikoff, A.\ Sims and D.\ Williams, ``Cartan subalgebras in $C^*$-algebras of Hausdorff \'etale groupoids,'' {\em Integral Equations Operator Theory}. \textbf{85}, 109-126 (2016). 

\bibitem{KK_duality_ss} N.\ Brownlowe, A.\ Buss, D.\ Gon\c calves, J.B.\ Hume, A. \ Sims and M.F.\ Whittaker, ``{$KK$}-duality for self-similar groupoid actions on graphs'', \textit{Trans. Amer. Math. Soc.} 377, 5513--5560 (2024).


\bibitem{Duwe:2025:Diag-pp} A.\ Duwenig, ``Non-traditional $C^*$-diagonals in twisted groupoid $C^*$-algebras'', to appear in {\em Indiana University Mathematics Journal} (2026).


\bibitem{DGNRW} A.\ Duwenig, E.\ Gillaspy, R.\ Norton, S.\ Reznikoff and S.\ Wright, ``Cartan subalgebras for
non-principal twisted groupoid $C^*$-algebras,'' {\em J. Funct. Anal.}. \textbf{279(6)} (2020).


\bibitem{DWZ:Twist} A.\ Duwenig, D.P.\ Williams and J.\ Zimmerman, ``Non-traditional Cartan Subalgebras in Twisted Groupoid $C^*$-Algebras,'' \emph{International Mathematics Research Notices}, Volume \textbf{2025}, Issue 4 (2025).


\bibitem{Self-similar:EP17} R.\ Exel and E.\ Pardo, ``Self-similar graphs, a unified treatment of Katsura and Nekrashevych $C^*$-algebras,'' Adv. Math. {\bf 306} (2017), 1046--1129.

\bibitem{Katsura_08} T.\ Katsura, ``A construction of actions on Kirchberg algebras which induce given actions
on their K-groups'', J. Reine Angew. Math., 617:27–65, (2008).


\bibitem{Self-similar:LY19} H.\ Li and D.\ Yang, ``KMS states of self-similar $k$-graph $\rm C^*$-algebras,'' \emph{J. Funct. Anal}. {\bf 276} (2019), no.~12, 3795--3831; MR3957999

\bibitem{LY21} H.\ Li and D. Yang, ``The ideal structures of self-similar $k$-graph $C^*$-algebras,''  \emph{Ergodic Theory Dynam. Systems} \textbf{41} (2021), no.~8, 2480--2507. 

\bibitem{Self-similar:LY21} H.\ Li and D.\ Yang, ``Self-similar k-graph $C^*$-algebras,'' \emph{International Mathematics Research Notices}, Volume \textbf{2021}, Issue 15, Pages 11270–11305 (2021). 


\bibitem{Li:EveryClass}X.\ Li, ``Every classifiable simple $\rm C^*$-algebra has a Cartan subalgebra,'' \emph{Invent. Math.} {\bf 219} (2020), no.~2, 653--699.

\bibitem{Xin_non-unique} X.\ Li, ``Constructing Menger manifold {$C^*$}-diagonals in classifiable {$C^*$}-algebras'', \textit{Int. Math. Res.
Not. IMRN,} 2022(23):18992–19053, (2022).

\bibitem{LiRenault} X.\ Li and J.\ Renault, ``Cartan subalgebras in ${\rm C}^*$-algebras. Existence and uniqueness,'' Trans. Amer. Math. Soc. {\bf 372} (2019), no.~3, 1985--2010.

\bibitem{MRW:1996:CtsTrace} P.\ Muhly, J.\ Renault and D.P.\ Williams, ``Continuous-trace groupoid $C^*$-algebras. III,'' {\em Trans. Amer. Math. Soc.}. \textbf{348}, 3621-3641 (1996). 

\bibitem{Nekrashevych_generic} V.\ Nekrashevych, ``Cuntz-{P}imsner algebras of group actions'', J. Operator Theory, \textbf{52}, 223--249 (2004). 

\bibitem{Nekrashev2009} V.\ Nekrashevych, ``$C^*$-algebras and self-similar groups,'' \emph{J. Reine Angew. Math.} {\bf 630} (2009), 59--123.

\bibitem{Nek05} V.\ Nekrashevych, ``Self-similar groups," American Mathematical Society, Providence, RI, 2005, xii+231.

\bibitem{pitts2021normalizers} D.\ Pitts, ``Normalizers and approximate units for inclusions of $C^*$-algebras,'' Indiana Univ. Math. J., \textbf{72}, 1849-1866 (2023).

\bibitem{Raad}A.\ Raad, ``A generalization of Renault's theorem for Cartan subalgebras,'' \emph{Proc. Amer. Math. Soc.} {\bf 150} (2022), no.~11, 4801--4809.

\bibitem{Rae05} I.\ Raeburn, ``Graph algebras," volume 103 of CBMS Regional Conference Series in Mathematics. Published for
the Conference Board of the Mathematical Sciences, Washington, DC; by the American Mathematical Society,
Providence, RI, 2005.

\bibitem{Renault:2023:Ext} J.\ Renault, ``Abelian twisted groupoid extensions,'' \textit{J.\ Operator Theory}, \textbf{89}, 249-283 (2023).

\bibitem{Renault:2008} J.\ Renault, ``Cartan subalgebras in $C^*$-algebras,'' Irish Math. Soc. Bull. No. 61 (2008), 29--63.

\bibitem{Tu:BCConj} J.\ Tu, ``La conjecture de Baum-Connes pour les feuilletages moyennables,'' \emph{$K$-Theory} {\bf 17} (1999), no.~3, 215--264.

\bibitem{Valente-Yang25} R.\ Valente and D.\ Yang, ``Semigroups of self-similar actions and higher rank Baumslag-Solitar semigroups,''
\textit{Proc. R. Soc. Edinburgh Sect. A: Math.}, Published online 2025: 1-30.
\end{thebibliography}

\end{document}